\documentclass[oneside, 12pt, reqno]{amsart}
\usepackage[T1]{fontenc}
\usepackage[condensed]{tgheros}
\DeclareRobustCommand{\SkipTocEntry}[5]{}
\def\HYPER{\relax}
\ifdefined\HYPER
\usepackage[hidelinks, pdfstartview=FitB, pdfpagemode=UseNone]{hyperref}
\else
\renewcommand{\href}[2]{\relax}
\renewcommand{\url}[1]{#1}
\fi
\usepackage{amssymb,nccmath}
\usepackage{esint}
\usepackage{mathtools}
\usepackage[english]{babel}
\usepackage{epsfig}
\usepackage{mathptmx}
\usepackage{amsmath,amsfonts,amssymb}
\usepackage{mathtools}
\usepackage{enumitem}
\usepackage{mathrsfs}
\usepackage[all]{xy}
\usepackage{graphicx}
\usepackage{latexsym}
\usepackage{verbatim}
\usepackage{xcolor}
\usepackage{xifthen}

\usepackage[normalem]{ulem}
\usepackage{accents}
\usepackage{marvosym}

\numberwithin{equation}{section}

\newcommand{\+}{\nobreak\hspace{.087em}\nopagebreak}

\def\Re{\mathop{\mathrm{Re}}}
\def\Im{\mathop{\mathrm{Im}}}

\def\D{\mathbb D}

\def\H{\mathbb{H}}

\def\R{\mathbb R}
\def\Z{\mathbb Z}
\def\C{\mathbb C}
\def\Hol{{\sf Hol}}

\def\N{\mathbb N}
\def\Q{\mathbb Q}

\def\dist{{\rm dist}}
\def\const{{\rm const}}
\def\id{{\sf id}}
\def\Arg{\mathrm{Arg}}

\newtheorem{theorem}{Theorem}[section]
\newtheorem{lemma}[theorem]{Lemma}
\newtheorem{proposition}[theorem]{Proposition}
\newtheorem{corollary}[theorem]{Corollary}

\newtheorem{addendum}[theorem]{Addendum}

\theoremstyle{definition}
\newtheorem{definition}[theorem]{Definition}
\newtheorem{example}[theorem]{Example}

\theoremstyle{remark}
\newtheorem{remark}[theorem]{Remark}

\numberwithin{equation}{section}

\newcommand{\hd}{\rho}
\newcommand{\phd}{\rho^*}
\newcommand{\phdisk}{\mathrm B_h}

\newcommand{\Zen}{\mathcal{Z}}
\newcommand{\Zn}{\mathcal{Z}_{\scriptscriptstyle\forall}}

\newcommand{\U}{{\sf Uni}}
\newcommand{\UH}{\mathbb{H}}

\let\HR=\UHR
\newcommand{\UD}{\mathbb{D}}
\newcommand{\UC}{\partial\UD}
\newcommand{\Complex}{\mathbb{C}}
\newcommand{\ComplexE}{\overline{\mathbb{C}}}
\newcommand{\Real}{\mathbb{R}}
\newcommand{\Natural}{\mathbb{N}}
\newcommand{\di}{\mathop{\mathrm{d}}\nolimits}
\newcommand{\anglim}{\angle\lim}

\renewcommand{\emptyset}{\varnothing}
\renewcommand{\ge}{\geqslant}
\renewcommand{\le}{\leqslant}
\renewcommand{\geq}{\geqslant}
\renewcommand{\leq}{\leqslant}
\newcommand{\proofof}[1]{{\fontseries{bx}\fontshape{it}\selectfont Proof of #1}}
\newcommand{\Step}[1]{\bigskip\noindent{\textsc{Step #1.}}}
\newcommand{\StepP}[1]{\medskip\noindent{\textsc{Proof of #1.}}}

\newcommand{\StepC}[2]{\medskip\noindent{\textsc{Case}~#1: #2.}}

\newenvironment{ourlist}{\begin{enumerate}[label={\bf (\arabic*)}, ref={\rm (\arabic*)}, left=0em]
\everydisplay{\makeatletter\def\@eqnum{\normalfont(\theequation)}\makeatother}}{\end{enumerate}}
\newenvironment{romlist}{\begin{enumerate}[label={\rm (\roman*)}, ref={\rm (\roman*)}, left=.7em]}{\end{enumerate}}
\newenvironment{equilist}{\begin{enumerate}[label={\rm (\alph*)}, ref={\rm (\alph*)}, left=.7em]}{\end{enumerate}}
\newenvironment{statlist}{\begin{enumerate}[label={\bf (\Alph*)}, ref={(\Alph*)}, left=0em]%
\everydisplay{\makeatletter\def\@eqnum{\normalfont(\theequation)}\makeatother}}{\end{enumerate}}

\newcommand{\dff}[1]{\textsl{#1}}

\newcommand\Wtilde[1]{\stackrel{\sim}{\smash{#1}\rule{0pt}{0.54ex}}}
\newcommand{\mtilde}[1]{\mathchoice{\,\widetilde{\!#1\!}\,}{\,\widetilde{\!#1\!}\,}%
{\raise.25ex\hbox{$\scriptstyle\Wtilde{#1}$}}{\Wtilde{#1}}}

\begin{document}
\title[Simultaneous linearization and centralizers of parabolic self-maps II]{Simultaneous linearization and centralizers of parabolic self-maps II: positive hyperbolic step}

\author[M. D. Contreras]{Manuel D. Contreras $^\dag$}

\author[S. D\'{\i}az-Madrigal]{Santiago D\'{\i}az-Madrigal $^\dag$}
\address{Camino de los Descubrimientos, s/n\\
	Departamento de Matem\'{a}tica Aplicada II and IMUS\\
	Universidad de Sevilla\\
	Sevilla, 41092\\
	Spain.}\email{contreras@us.es} \email{madrigal@us.es}

\author[P. Gumenyuk]{Pavel Gumenyuk$^\ddag$}\address{Pavel Gumenyuk: Department of
Mathematics\\  Politecnico di Milano, via E. Bonardi 9\\ Milan 20133, Italy.}
\email{pavel.gumenyuk@polimi.it}

\thanks{$^\dag$ Partially supported by Ministerio de Innovación y Ciencia, Spain, project PID2022-136320NB-I00. The author thanks IMUS-Maria de Maeztu grant CEX2024-001517-M - Apoyo a Unidades de Excelencia María de Maeztu for supporting this research, funded by MICIU/AEI/ 10.13039/501100011033.}
\thanks{$^\ddag$ Partially supported  by GNSAGA INdAM (\textit{Istituto Nazionale di Alta Matematica ``Francesco Severi''}) Italy}

\date\today

\begin{abstract}
The study of holomorphic self-mappings of the unit disc commuting under the composition  goes back to A.L.\,Shields (1964), W.A.\,Pranger (1970), D.F.\,Behan (1973), and  C.C.\,Cowen (1984). In many situations, the centralizer of a holomorphic self-map ${\varphi:\mathbb  D \to \mathbb  D}$, i.e. the semigroup $\mathcal Z_\forall(\varphi):=\{\psi\in\mathsf{Hol}(\mathbb D):\psi\circ\varphi=\varphi\circ\psi\}$ turns out to be commutative. However, this does not hold for the case of a parabolic self-map $\varphi$ of \textit{positive} hyperbolic step, which is analyzed in detail in this paper. We investigate the relationships among commutativity, simultaneous linearization, and holomorphic models. In particular, we obtain existence and uniqueness results for the simultaneous linearization of commuting pairs $\varphi$, $\psi\in \mathcal Z_\forall(\varphi)$.  Furthermore, extending this notion to arbitrary families of holomorphic self-mappings, we show that a given (finite or infinite) family in the centralizer ${\Delta\subset\mathcal Z_\forall(\varphi)}$ can be simultaneously linearized together with~$\varphi$ if and only if any two elements of~$\Delta$ commute with each other. This gives a far reaching extension of Cowen's result concerning commuting pairs in~$\mathsf{Hol}(\mathbb D)$.
\end{abstract}

\maketitle

\tableofcontents

\section{Introduction}\label{Introduction}
The study of holomorphic self-maps of the unit disc $\UD:=\{z\in\C:|z|<1\}$ and their dynamical properties is a topic of recurrent interest and rich history: from local aspects tracing back to Schr\"oder~(1870) and Koenigs~(1883) to holomorphic models by Cowen~\cite{Cowen} (see also~\cite{Canonicalmodel}) and
the operator-theoretic approach by Sarason~\cite{Sarason1988} (see also~\cite{Frej}), applications to the study of composition operators (see e.g.\,\cite{CM1995}) and to probability theory (see e.g.\,\cite{Goryainov-survey}), as well as many recent developments and extensions to more general contexts, e.g. \cite{Convex,AFGG2024,AFGK2024,MP-operatorTh,BM2004,Filippo-puntosTAMS,BKR2024,GuKouMouRoth,IvrNic,MaxBpro} just to name some.

The class $\Hol(\UD)$ formed by all holomorphic self-maps of~$\UD$ carries a natural structure of a topological semigroup with unity w.r.t. the operation of composition and the usual topology of locally uniform convergence.
Since this semigroup is not commutative, the study of the \dff{centralizers}
$$
  \Zn(\varphi):=\{\psi\in\Hol(\UD):\psi\circ\varphi=\varphi\circ\psi\},\quad \varphi\in\Hol(\UD),
$$
is one of the basic tasks in the analysis of~$\Hol(\UD)$ as a whole.

The structure of~$\Zn(\varphi)$ turns out to depend strongly on the dynamical characteristics of~$\varphi$. For locally injective self-mappings ${\varphi\in\Hol(\UD)}$ possessing a geometrically attracting\footnote{A fixed point~$\tau$ of a holomorphic function~$\varphi$ is called geometrically attracting if ${0<|\varphi'(\tau)|<1}$.} fixed point~$\tau$ in~$\UD$, Pranger~\cite{Pranger} was able to describe centralizers as isomorphic images of a class of topologically closed subsemigroups of ${\{t\in\C:|t|\le1\}}$ endowed with the usual operation of  multiplication in~$\C$. The isomorphism is defined via the \dff{simultaneous linearization}, i.e. by making use of  a non-constant holomorphic solution $h_\varphi:\UD\to\C$ to
the problem
$$
  h_\varphi\circ \varphi\,=\,\varphi'(\tau)h_\varphi,~\quad~ h_\varphi\circ \psi\,=\,th_\varphi \quad\text{for any $\,\psi\in\Zn(\varphi)\,$ and some $\,t=t(\psi)\in\overline{\UD}\,$.}
$$

The situation when $\varphi$ has no fixed point in the unit disc~$\UD$ is much more complicated. A seminal contribution was made by Cowen~\cite{Cowen-comm}, who extended his abstract approach in linearization~\cite{Cowen}, known in the modern literature as the machinery of \dff{holomorphic models}, to the study of commuting pairs of holomorphic self-maps.  For the fixed-point free case (a reformulation of) Cowen's result relates commutativity of ${\varphi,\psi}\in\Hol(\UD)$ to the simultaneous linearization problem of the form
\begin{equation}\label{EQ_SL-Cowen}
   h\circ \varphi\,=\,h+1,~\quad~ h\circ \psi\,=\,h+c \quad\text{for a suitable $\,c=c_{\varphi,\psi}\in\C$},
\end{equation}
where $h=h_{\varphi,\psi}$ depends, in general, on~$\psi\in\Zn(\varphi)$.

A question fundamental for understanding the structure of~$\Zn(\varphi)$ is whether the simultaneous linearization problem~\eqref{EQ_SL-Cowen} admits a \emph{common} solution for the whole centralizer.

The answer depends on the dynamical type of the self-map~$\varphi$ (see Section~\ref{SS_PRE-holo-selfmaps} for the relevant definitions). It is positive for all hyperbolic self-maps, as follows already from Cowen's results in~\cite{Cowen-comm}. Several attempts were made \cite{BisGen01,Gentili-Vlacci,Vlacci} to extend this to parabolic self-mappings, but each time a certain additional condition, with slight variations, was needed for the proof.

In~\cite{CDG-zero-h.step} we have recently resolved this long-standing problem by showing that~--- similarly to the hyperbolic case~--- for parabolic self-maps~$\varphi$ of \emph{zero} hyperbolic step, the Koenigs function of~$\varphi$ (see Section~\ref{Sec:modelos} for the definition) provides the simultaneous linearization for the whole centralizer.

In this paper we analyse the centralizers and simultaneous linearization for the remaining case of parabolic self-maps of \emph{positive} hyperbolic step. This case is the most interesting and complicated. It exhibits phenomena absent in the hyperbolic and parabolic zero-step settings. For example, in this case the Koenigs function of~$\varphi$ generally fails to linearize the whole centralizer and the simultaneous linearization problem~\eqref{EQ_SL-Cowen} may have no solution~$h$ that does not depend on the choice of~${\psi\in\Zn(\varphi)}$. When this happens, the centralizer is not abelian. Moreover, a univalent parabolic self-map~$\varphi$ of positive hyperbolic step can have non-univalent elements in the centralizer. This is possible neither for hyperbolic self-maps~$\varphi$, nor for parabolic self-maps of zero hyperbolic step, see \cite[Corollary~4.9]{Cowen-comm} and \cite[Theorem~4.5]{CDG-zero-h.step}.

The relation we establish in this paper between the simultaneous linearization and commutativity is two-fold.
Firstly, we show that two self-maps $\varphi,\psi\in\Hol(\UD)$ with the same Denjoy\,--\,Wolff point on the boundary (see Section~\ref{SS_PRE-holo-selfmaps} for the definition) commute if and only if they admit a simultaneous linearization by a holomorphic function $h$ satisfying some kind of local univalence condition, see Theorem~\ref{TH_SL-vs-COMM} and Addendum~\ref{ADD_1}. Moreover, we are able to show that in most of the cases the pair ${(h,c)}$ solving the simultaneous linearization problem for commuting parabolic self-maps is essentially unique, see Corollary~\ref{Thm:StrongUniqueness} and Remark~\ref{RM_StrongUniqueness} for precise statements.

The second aspect of the interplay between the commutativity and simultaneous linearization we discover in this paper concerns elements of the centralizer~$\Zn(\varphi)$: we prove that any two self-maps in a family $\Delta\subset\Zn(\varphi)$ commute if and only if there is a common (holomorphic) solution~$h$ to~\eqref{EQ_SL-Cowen} for all~${\psi\in\Delta}$, see Theorems~\ref{TH_abelian-part0} and~\ref{TH_abelian-part}. In this way, the paper shows that although global linearization of the centralizer is no longer available in the parabolic positive-step setting, the simultaneous linearization machinery still provides an exact description of its abelian substructures.

It is worth mentioning that results of this spirit are known in other contexts, see e.g. \cite{DeWitt,SL,FKh2009,SLJ1,SLJ2} and references therein. However, our proofs are based on methods specific for holomorphic dynamical systems. Necessary preliminaries and some useful results from holomorphic dynamics are collected in Section~\ref{S_prelim}. For readers' convenience, we also collect main results from~\cite{CDG-zero-h.step} concerning centralizers of parabolic self-maps in case of zero hyperbolic step. These are placed in Section~\ref{S_when-zero}.

The new results mentioned above are developed in Section~\ref{S_simmultaneous-lin}.
Furthermore, in Section~\ref{S_SLC}, we establish various identities that allow one to determine~--- without solving the simultaneous linearization problem~\eqref{EQ_SL-Cowen}~--- the unique value ${c_{\varphi,\psi}}$ of the constant~$c$ for which this problem admits a holomorphic solution~$h$. Finally, in Section~\ref{S_realSLC}, we analyze the case when~$c_{\varphi,\psi}$ is real. We show that under this condition, the two commuting parabolic self-mappings share many properties, see Theorem~\ref{TH_real-SLC=common-properties}. Moreover, for the case of $\varphi$ having positive hyperbolic step we show that ${c_{\varphi,\psi}\in\Real}$ if and only if the composition~$\varphi\circ\psi$ is also of positive hyperbolic step. This allows us to deduce Cowen's main result~\cite[Theorem~3.1]{Cowen-comm} as a corollary.

An essential role in our arguments is played by the properties of holomorphic self-maps commuting with a parabolic automorphism. We study them in Appendix~\hyperlink{AppendixA}{A}.

The paper is concluded by two examples of parabolic self-maps of positive hyperbolic step which illustrate two opposite extreme situations. In the first example, the centralizer is literally as small as possible: it coincides with the set of all natural iterates~${\{\varphi^{\circ n}:n=0,1,2,\ldots\}}$.
In the second example, the centralizer is ``almost as big as possible''. In a sense, it is comparable to the centralizer of a parabolic automorphism, although the self-map in that example is not even univalent. These examples can be found in Appendix~\hyperlink{AppendixB}{B}.

\section{Preliminaries and auxiliary results}\label{S_prelim}
Below we introduce some notations and basic theory used further in the paper. For  more details and for the proofs of the previously known results presented in this section without proof, we refer the interested readers to the recent monograph~\cite{Abate2}.

\addtocontents{toc}{\SkipTocEntry}
\subsection{Notation}\label{Notation}
As usual, we denote by $\Natural$ and $\Natural_0$ the sets of all positive and all non-negative integers, respectively. We denote the unit disc by
${\UD:=\{z\in\C:|z|<1\}}$ and the punctured unit disc by ${\UD^*:=\UD\setminus\{0\}}$. We write $\UH:={\{w\in\C:\Im w>0\}}$ for the upper half-plane and $\HR:={\{w\in\C:\Re w>0\}}$ for the right half-plane. For $w\in\C\setminus\{0\}$, $\Arg\, w$ will stand for the principal value of the argument of~$w$, i.e. the unique value of $\arg w$ belonging to ${(-\pi,\pi]}$.

Furthermore, denote by $\Hol(D,E)$ the class of all holomorphic mappings of a domain $D\subset\C$ into a set $E\subset\C$,
and let $\U(D,E)$ stand for the class of all \textit{univalent} (i.e. injective holomorphic) mappings from $D$ to~$E$.  As usual, we endow $\Hol(D,E)$ and $\U(D,E)$ with the topology of locally uniform convergence. In case $E=D$, we will write  $\Hol(D)$ and $\U(D)$ instead of $\Hol(D,D)$ and $\U(D,D)$, respectively.

For a self-map $\varphi:D\to D$ of a domain $D\subset\C$ and ${n\in\Natural}$ we denote by $\varphi^{\circ n}$ the $n$-th iterate of~$\varphi$, and let $\varphi^{\circ0}:=\id_D$, the identity map in $D$. Moreover, if $\varphi$ is an automorphism of~$D$, then for every $n\in\N$, we denote by $\varphi^{\circ-n}$ the $n$-th iterate of~$\varphi^{-1}$.

If $D$ is a hyperbolic domain in the complex plane, we denote by $\hd_{D}$ (resp. $\phd_D$) the hyperbolic distance (resp. pseudohyperbolic distance) in $D$. We write $\phdisk^{D}(z,r)$ for the pseudohyperbolic disc in~$D$ of radius~$r$ centered at~$z$. When $D$ is the unit disc, we simply write $\phdisk(z,r)$ to denote $\phdisk^{\D}(z,r)$.

\addtocontents{toc}{\SkipTocEntry}
\subsection{Holomorphic self-maps of the unit disc}\label{SS_PRE-holo-selfmaps}
The study of the dynamics of a generic holomorphic self-map $\varphi$ of the unit disc $\mathbb{D}$, different from the identity map,  is a classical and well-established branch of Complex Analysis.

The central result in the area is the Denjoy\,--\,Wolff Theorem, which states that if $\varphi$ is different from an elliptic automorphism (i.e. not an automorphism of~$\UD$ possessing a fixed point in~$\UD$), then the sequence of the iterates $(\varphi ^{\circ n})$ converges locally uniformly in~$\UD$ to a certain point~${\tau\in\overline{\mathbb{D}}}$.  This point  is called the \dff{Denjoy\,--\,Wolff point\/} of $\varphi$. Moreover, if $\tau\in \partial \D$, it is the unique boundary fixed point at which the angular derivative $\varphi'(\tau)$ is finite and belongs to $(0,1]$.

According to the position of the Denjoy\,--\,Wolff point~$\tau$ and to the value of the \textsl{multiplier} $\varphi'(\tau)$, holomorphic self-maps $\varphi\in\Hol(\UD)$ different from elliptic automorphisms are divided into three categories. Namely, $\varphi$ is called:
\begin{itemize}
\item[(a)] \dff{elliptic\/} if $\tau\in\UD$,

\item[(b)] \dff{hyperbolic\/} if $\tau\in \partial \D$ and $\varphi'(\tau )<1$, and

\item[(c)] \dff{parabolic\/} if $\tau
\in \partial \D$ such that $\varphi'(\tau )=1$.
\end{itemize}
All elliptic automorphisms of~$\UD$, including the identity mapping~$\id_\UD$,  are conventionally incorporated into the category~(a) of elliptic self-maps.
Similarly, for an elliptic automorphism different from~$\id_\UD$, its Denjoy\,--\,Wolff point is defined to be its unique fixed point in~$\UD$.

Parabolic self-maps can have very different dynamical properties depending on the so-called \textit{hyperbolic step}.
Let $\varphi\in\Hol(\UD)$ be non-elliptic. Thanks to the Schwarz\,--\,Pick Lemma, for the orbit $\big(z_n\big):=\big(\varphi^{\circ n}(z_0)\big)$ of any point ${z_0\in\UD}$, there exists a finite limit $q(z_0):=\lim_{n\to+\infty} \hd_\D(z_{n},z_{n+1})$. 
 It is known, see e.g. \cite[Corollary\,4.6.9]{Abate2}, that  either $q(z_0)>0$ for all~${z_0\in\UD}$, or $q\equiv0$ in~$\UD$.  The self-map~$\varphi$ is said to be of \dff{positive} or of \dff{zero hyperbolic step} depending on whether the former or the latter alternative occurs.
If $\varphi$ is
hyperbolic, then it is always of positive hyperbolic step. However, there exist parabolic self-maps of zero as well as of positive hyperbolic step.

\addtocontents{toc}{\SkipTocEntry}
\subsection{Commuting holomorphic self-maps}\label{SS_commuting}
It is clear that if two  holomorphic self-maps $\varphi,\psi\in\Hol(\UD)\setminus\{\id_\UD\}$ commute, i.e. ${\varphi\circ\psi}={\psi\circ\varphi}$, and if one of them is elliptic, then the other is also elliptic and they share the Denjoy\,--\,Wolff point. The situation is not so evident when we consider non-elliptic self-maps, but thanks to the contributions due to Heins \cite[Lemma~2.1]{Heins},  Behan~\cite[Theorem~6]{Behan}, and Cowen~\cite[Corollary~4.1]{Cowen-comm}, one can state the following.
\begin{theorem}[{\cite{Heins,Behan,Cowen-comm}}]\label{TH_Behan-and-Ko}
Let $\varphi\in\Hol(\UD)$ be non-elliptic. Suppose ${\psi\in\Hol(\UD)\setminus\{\id_\UD\}}$ commutes with~$\varphi$. Then:
\begin{ourlist}
\item if $\varphi$ is a hyperbolic automorphism, then $\psi$ is also a hyperbolic automorphism and has the same fixed points in~$\UC$ as~$\varphi$;
\item if $\varphi$ is not a hyperbolic automorphism, then $\psi$ and~$\varphi$ have the same Denjoy\,--\,Wolff point;
\item if $\varphi$ is hyperbolic (resp., parabolic), then $\psi$ is also hyperbolic (resp., parabolic).
\end{ourlist}
\end{theorem}

It is worth mentioning that parabolic self-maps of positive hyperbolic step can commute with parabolic self-maps of zero hyperbolic step. Moreover, in contrast to the hyperbolic case, the fact that one of them is an automorphism would not imply that the other must  also be an automorphism. To see this, one can consider the following example: ${\varphi:=h^{-1}\circ(h+1)}$, ${\psi:=h^{-1}\circ(h+i)}$, where $h$ is a conformal map of~$\UD$ onto~$\UH$.

The current state of the art and historical notes concerning the study of commuting holomorphic self-maps of~$\UD$ can be found in~\cite[Section~4.10]{Abate2}. It is also worth mentioning that relations between commutativity and the \emph{embedding problem} (i.e. the problem to embed  discrete iteration into a continuous-time holomorphic dynamical system) have been recently explored in~\cite{CDG-Centralizer,CDG-Embedding}.

\addtocontents{toc}{\SkipTocEntry}
\subsection{Holomorphic models for holomorphic self-maps}\label{Sec:modelos}
An indispensable role in our study is played by the concept of a holomorphic model, which goes back to Pommerenke~\cite{Pom79}, Baker and Pommerenke~\cite{BakerPommerenke}, and Cowen~\cite{Cowen}. The terminology we use is mainly borrowed from~\cite{Canonicalmodel}.

\begin{definition}\label{DF_absirbing} Let  $\varphi\in\Hol(\D)$. A domain $V \subset \D$ is \emph{invariant} for $\varphi$ (or $\varphi$-invariant) if $\varphi(V)\subset V$; it is \emph{absorbing} for $\varphi$ (or $\varphi$-absorbing) if it is $\varphi$-invariant and
$$
\D=\bigcup_{n\in\Natural}\big(\varphi^{\circ n}\big)^{-1}(V).
$$
In other words, a $\varphi$-invariant domain is $\varphi$-absorbing if it eventually contains the orbit of any point of $\D$.
\end{definition}

\begin{definition}\label{DF_holomorphic-model} A \dff{holomorphic semimodel} of $\varphi\in\Hol(\D)$  is a triple $\mathcal M:=(S,h,\alpha)$, where $S$ is a Riemann surface, $\alpha$ is an automorphism of $S$, and $h$ is a holomorphic map from $\UD$ into $S$ satisfying the following two conditions:
	\begin{enumerate}[left=2.5em]
		\item[(HM1)] $h\circ \varphi=\alpha \circ h$  {}~and
		\item[(HM2)] $S\,=\,\bigcup_{n\geq0}  \big(\alpha^{\circ n}\big)^{-1}(h(\D))$.
	\end{enumerate}
		A holomorphic \dff{model} of $\varphi\in \Hol (\D)$ is a semimodel $\mathcal M:=(S,h,\alpha)$ for which
	\begin{enumerate}[left=2.5em]
		\item[(HM3)] there exists a $\varphi$-absorbing domain $V\subset\UD$ in which $h$ is injective.
	\end{enumerate}
The Riemann surface $S$ is called the \dff{base space}, and the map $h$ is called the \dff{intertwining map} of the holomorphic model~$\mathcal M$.
\end{definition}

Given a holomorphic model $\mathcal M$ of~$\varphi$, every holomorphic semimodel of~$\varphi$ factorises via the model~$\mathcal M$ in the following sense.

\begin{lemma}[{\cite[Lemma 3.5.8 and Remark 3.5.6]{Abate2}}]\label{Le:morphism}
Assume that $\varphi\in\Hol(\UD)$ admits a holomorphic model $\mathcal M:=(S,h,\alpha)$.
If $(S_1,h_1,\alpha_1)$ is another holomorphic semimodel for~$\varphi$, then there exists  a surjective holomorphic map $\beta:S\to S_1$ such that ${h_1=\beta\circ h}$ and ${\beta\circ\alpha}={\alpha_1\circ\beta}$.
\end{lemma}

Every
non-elliptic self-map $\varphi\in\Hol(\UD)$ admits an essentially unique holomorphic model. More precisely, the following fundamental theorem holds.

\begin{theorem}[\protect{\cite[Corollary~4.5.5]{Abate2}}] \label{Thm:uniqness} Every
non-elliptic self-map $\varphi\in\Hol(\UD)$ admits a holomorphic model. Moreover such a model is unique up to a model isomorphism; i.e., if $(S_1,h_1,\alpha_1)$ and $(S_2,h_2,\alpha_2)$ are holomorphic models for $\varphi$, then there exists a biholomorphic map $\beta$ of~$S_1$ onto~$S_2$ such that
	$$
	h_2=\beta\circ h_1,\quad \alpha_2=\beta\circ\alpha_1\circ \beta^{-1}.
	$$
\end{theorem}

The proof of the existence of holomorphic models depends strongly on the existence of absorbing sets (see, i.e., \cite[Theorem 3.5.10]{Abate2}).
In the next result we summarize what we need about the models for self-maps of the unit disc.
The type of a self-map is reflected in, and actually can be fully determined from the kind of holomorphic model $\varphi$ admits.  For an open interval $I\subset\Real$, we define
$$
 S_I:=\Real\times I=\{x+iy:x\in\Real,\,y\in I\}.
$$

\begin{theorem}[\protect{\cite{Cowen}, see also \cite[Theorem 4.6.8]{Abate2}}]\label{Thm:model}
Let $\varphi\in\Hol(\UD)$. The following statements hold.
\begin{ourlist}
	\item\label{IT_HM-hyp} $\varphi$ is a hyperbolic self-map with multiplier $\lambda\in(0,1)$ if and only if $\varphi$ admits a holomorphic model of the form $\mathcal M_\varphi:=(S_{I},h,z\mapsto z+1)$, where $I=(a,b)$ is a bounded open interval of length ${b-a=\pi/|\log\lambda|}$.
	\item\label{IT_HM-para-PHS} $\varphi$ is a parabolic self-map of positive hyperbolic step if and only if $\varphi$ admits a holomorphic model of the form ${\mathcal M_\varphi:=(S_{I},h,z\mapsto z+1)}$, where $I$ is an open unbounded interval different from the whole~$\Real$.
	\item\label{IT_HM-para-0HS} $\varphi$ is a parabolic self-map of zero hyperbolic step if and only if $\varphi$ admits a holomorphic model of the form ${\mathcal M_\varphi:=(\C,h,z\mapsto z+1)}$.	
\end{ourlist}	
\end{theorem}

\begin{remark}\label{RM_normalization}
In the above theorem, we may assume that:
\begin{itemize}
 \item[-]in cases~\ref{IT_HM-hyp} and~\ref{IT_HM-para-0HS}, $h(0)=0$;
 \item[-]in case~\ref{IT_HM-para-PHS}, $\Re h(0)=0$ and ${S_I=S_{(0,+\infty)}=\UH}$ or ${S_I=S_{(-\infty,0)}=-\UH}$.
\end{itemize}
Using the uniqueness part of Theorem~\ref{Thm:uniqness}, one can show (see e.g. \cite[Corollary 4.6.12]{Abate2} for details) that the above assumptions play the role of a normalization under which the holomorphic model $\mathcal M_\varphi$ for a given non-elliptic self-map~$\varphi$ is unique. Note that the normalization for cases \ref{IT_HM-hyp} and~\ref{IT_HM-para-0HS} would also work in case~\ref{IT_HM-para-PHS}, but we prefer to use another normalization, so that for parabolic self-maps of positive hyperbolic step,  the base space~$S_I$ of~$\mathcal M_\varphi$ coincides with $\UH$ or~$-\UH$. Moreover, replacing, if necessary, $\varphi$ with $z\mapsto\overline{\varphi(\bar z)}$ we may assume in most of the proofs that ${S_I=\UH}$.
\end{remark}

\begin{definition}\label{DF_canonical}
 The unique holomorphic model $\mathcal M_\varphi$ of a non-elliptic self-map  $\varphi\in\Hol(\UD)$ defined in Theorem~\ref{Thm:model} and normalized as in Remark~\ref{RM_normalization} is called \dff{the canonical (holomorphic) model} for~$\varphi$. The intertwining map~$h$ of the canonical model~$\mathcal M_\varphi$ is called the \dff{Koenigs function}, and we denote it, from now on, by $h_{\varphi}$. Similarly, we write $S_\varphi$ for the base space of the canonical model.
\end{definition}

At some point we will deal also with elliptic self-maps. Therefore, it is worth mentioning that every elliptic self-map ${f\in\Hol(\UD)}$ satisfying ${\lambda:=f'(\tau)\in\UD^*}$, where $\tau\in\UD$ is the Denjoy\,--\,Wolff point of~$f$, admits a holomorphic model of the form ${\big(\C,h_f,z\mapsto\lambda z\big)}$. Under the normalization ${h_f'(\tau)=1}$ such a holomorphic model is unique and will be referred to as the \dff{canonical (holomorphic) model} for the elliptic self-map~$f$. Similarly to the non-elliptic case, the intertwining map~$h_f$ will be called the Koenigs function of~$f$.

\begin{remark}
An immediate consequence of the definition is that in the elliptic and non-elliptic cases, the Koenigs function satisfies
\begin{align*}
  &h_f\circ f = f'(\tau) h_f \tag{Schr\"oder's equation}\\
\intertext{and}
  &h_\varphi\circ \varphi = h_\varphi+1,  \tag{Abel's equation}
\end{align*}
respectively.
\end{remark}

\begin{remark}\label{RM_h.step=h(D)}
Let $\varphi\in\Hol(\UD)$ be a parabolic self-map. It follows immediately from Theorem~\ref{Thm:model} and Definition~\ref{DF_canonical} that the self-map~$\varphi$ is of positive hyperbolic step if and only if $h_\varphi(\UD)$ is contained in a half-plane of the form ${\{z:\Im z>A\}}$ or ${\{z:\Im z<A\}}$,  ${A\in\Real}$.
\end{remark}

\begin{remark}\label{RM_factorization}
Let $\varphi\in\Hol(\UD)$ be a non-elliptic self-map with canonical model $\big(S_\varphi,h_\varphi,z\mapsto z+1\big)$. Suppose that a holomorphic function ${h:\UD\to\C}$ satisfies the functional equation
$$
  h\circ\varphi=h+c,
$$
where~${c\in\C}$ is a constant. Then  $h$ factorizes via the Koenigs map of~$\varphi$. More precisely,
$$
  h=f\circ h_\varphi
$$
for a suitable holomorphic function $f:S_\varphi\,\to\, S[h,c]:=\bigcup_{n\in\Natural}\big(h(\UD)-nc\big)$ satisfying the identity ${f(w+1)=f(w)+c}$ for all~${w\in S_\varphi}$. This follows at once from Lemma~\ref{Le:morphism} if we notice that $\big(S[h,c],\,h,\,z\mapsto z+c\big)$ is a semimodel for~$\varphi$.
\end{remark}

\begin{remark} \label{Re:almostasemigroup}
Let $\varphi\in\Hol(\UD)$ be a non-elliptic self-map with canonical holomorphic model   $(S_\varphi,h_{\varphi},z\mapsto z+1)$. By the very definition of holomorphic model, we have that for any compact set $K$ in~$S_\varphi$, there exists  ${N>0}$ such that for every natural number $n>N$, it holds that $K+n\subset h_{\varphi}(\D)$. In fact, something stronger can be proved: Take $K$ a compact set in $S$, and consider the compact set $$\mtilde K=\{w+s:\, w\in K, s\in [0,2]\}\subset S_\varphi.$$ Then   there is $N>0$ such that $\mtilde K+n \subset h_{\varphi}(\D)$ for all  $n>N$. Therefore,  ${K+t\subset h_{\varphi}(\D)}$ for any real number ${t>N}$. Roughly speaking, this means that asymptotically every Koenigs domain behaves like the Koenigs domain of a non-elliptic semigroup. This fact has been used several times in the literature, see e.g.  \cite[Lemma~7.6]{Bracci-Roth} and   \cite[Lemma~2.2]{Poggi}.
A bit more generally, if ${V\subset\UD}$ is any $\varphi$-absorbing open set, then  repeating the same argument with $\UD$ replaced by~$V$, we see that for each compact ${K\subset S_\varphi}$ there exists $t_0={t_0(K,V)\ge0}$ such that ${K+t\subset h_{\varphi}(V)}$ for all ${t\in[t_0,+\infty)}$.
\end{remark}

\begin{remark} \label{Re:Pommerenke}
	If  $\varphi$ is a parabolic self-map of positive hyperbolic step, $\tau$ is its Denjoy\,--\,Wolff point,  and $h_{\varphi}$ is its Koenigs map, Pommerenke proved that
	$$
	\angle\lim_{z\to \tau} |\Im h_{\varphi}(z)|=+\infty
	$$
	(see \cite[Theorem 3(ii)]{Pom79}).
\end{remark}

\begin{proposition}[\protect{\cite[Proposition 5.6]{CDG-zero-h.step}}] \label{positivo2}
Let $\varphi\in\Hol(\UD)$ be a parabolic self-map with Denjoy\,--\,Wolff point ${\tau\in\partial\D}$ and Koenigs function $h_{\varphi}$. Let $\psi\in\Zn(\varphi)\setminus\{\id_\D\}$.
Then
	$$
	\angle \lim_{z\to\tau}\dfrac{h_\varphi(\psi(z))-h_\varphi(z)}{h^\prime_\varphi(z)(\psi(z)-z)}=1.
	$$
\end{proposition}

\begin{theorem} [\protect{\cite[Theorem 5.1]{CDP}}]
\label{Thm:V-sets}
Let $\varphi \in \Hol(\D)$ be
parabolic  with Denjoy\,--\,Wolff point $\tau\in \partial \D$ and Koenigs function $h_{\varphi}$.
Given $0<r <1,$ write\begin{equation*}
V_{\varphi}(r )=\left\{ z\in \mathbb{D}:\phd_\UD(z,\varphi(z))<r \right\}.
\end{equation*}Then,
\begin{ourlist}
\item\label{IT_V-sets(1)}
     $V_{\varphi}(r )$ is  $\varphi$-invariant, that is, $\varphi (V_{\varphi}(r ))\subset V_{\varphi}(r )$ for all $0<r <1.$

\item\label{IT_V-sets(2)}
     $V_{\varphi}(r)$ is a simply connected domain.

\item\label{IT_V-sets(3)} $V_{\varphi}(r )$ is isogonal at $\tau$, that is, for every sequence ${(z_n)\subset\UD}$ converging to~$\sigma$ non-tangentially is contained, except maybe for a finite number of terms, in~$V$.

\item\label{IT_V-sets(4)}
    If $0<r \leq \frac{1}{3},$ then both $\varphi $ and its  Koenigs map $h_{\varphi}$ are univalent in $V_{\varphi}(r ).$

\item\label{IT_V-sets(5)}
     If $\varphi $ is of zero hyperbolic step and $0<r \leq \frac{1}{3}$, then $V_{\varphi}(r )$ is $\varphi$-absorbing.
\end{ourlist}
\end{theorem}

\begin{lemma}[\protect{\cite[Lemma 2.10]{CDG-zero-h.step}}] \label{Le:V-sets}
	Let $\varphi \in \Hol(\D)$ be parabolic. If $\psi\in  \Zn(\varphi)$, then $V_{\varphi}(r)$ is $\psi$-invariant for all ${r\in(0,1)}$.
\end{lemma}

We recall the following uniqueness result for the Koenigs function of parabolic self-maps of zero hyperbolic step.

\begin{theorem} [\protect{\cite[Theorem 3.1]{CDP-Abel}}]\label{Thm:uniqueness}
Let $\varphi \in \Hol(\D)$ be a parabolic self-map of
zero hyperbolic step with Koenigs function  ${h_{\varphi} :\mathbb{D}\rightarrow \mathbb{C}}$.
Let $z_{0}\in \mathbb{D}$ such that ${h'_{\varphi}(z_{0})\neq 0}$. Further, let
${h_{1}: \mathbb{D}\rightarrow \mathbb{C}}$ be a holomorphic function satisfying the Abel
functional equation ${h_{1} \circ \varphi =h_{1} +1}$. Then the following conditions are equivalent:
\begin{romlist}
 \item $h_{1}=h_{\varphi}+c$ for some constant~${c\in\C}$;
 \item  there exist ${r>0}$ and ${N\in \mathbb{N}}$ such that for all~$n>N$, $h_{1}$ is
univalent in the hyperbolic disc of
radius~${r}$ centered at $\varphi^{\circ n}(z_{0})$.
\end{romlist}
\end{theorem}

Finally, we will make use of the following result concerning special cases of simultaneous linearization.
\begin{proposition}[\protect{\cite[Proposition 3.6]{CDG-zero-h.step}}]  \label{Prop:firstpropertiescoefficient}
Let $\varphi\in\Hol(\D)$ be parabolic with Koenigs function~$h_\varphi$ and Denjoy\,--\,Wolff point ${\tau\in\partial\UD}$.  Suppose $\psi\in\Hol(\UD)$ is another holomorphic self-map with the same Denjoy\,--\,Wolff point and such that ${h_\varphi\circ\psi}={h_\varphi+c}$ for some constant~${c\in\R}$. The following statements hold.
\begin{ourlist}
		\item\label{firstpropereties1}  $c=0$ if and only if $\psi=\id_\D$.
		\item\label{firstpropereties2}
           Let $m,n\in\N$. Then $c=\frac{m}{n}$ if and only if $\psi^{\circ n}=\varphi^{\circ m}$.
		\item\label{firstpropereties3}
           If $c<0$, then $\varphi$ as well as $\psi$ are parabolic automorphisms. Moreover, given ${m,n\in\N}$, ${c=-\frac{m}{n}}$ if and only if ${\psi^{\circ n}=\varphi^{\circ -m}}$.	
	\end{ourlist}		
\end{proposition}

\section{When the self-map $\varphi$ is of zero hyperbolic step}\label{S_when-zero}

As noted in the introduction, the paper \cite{CDG-zero-h.step} is dedicated to the study of the centralizer of parabolic functions having zero hyperbolic step. To aid in the readability of this paper, we summarize below some of the results and notations already introduced in that work.

\begin{theorem}[\protect{\cite[Theorem 4.1]{CDG-zero-h.step}}] \label{Prop:pseuso-semigroup}
Let $\varphi\in \Hol(\D)$ be a parabolic self-map of zero hyperbolic step with Koenigs function~$h_\varphi$. Then for any ${\psi\in \Hol(\D)\setminus\{\id_\UD\}}$, the following two conditions are equivalent:
\begin{romlist}
\item\label{IT-EQUIV_commute} $\psi\in \Zn(\varphi)$;
\item\label{IT-EQUIV_SimLin} $\varphi$ and $\psi $ have the same Denjoy\,--\,Wolff point and
$\,h_\varphi\circ \psi=h_\varphi+c\,$ for some constant $c\in\C$.
\end{romlist}
\end{theorem}

\begin{definition} \label{simul-linea-coefficient-zero}
Let $\varphi \in \Hol(\D)$ be a parabolic self-map of zero hyperbolic step  and let $\psi\in\Zn(\varphi)$. We will denote by~$c_{\varphi,\psi}$ the constant ${c\in\C}$ defined in a unique way by the relation ${h_\varphi\circ\psi=h_\varphi+c}$ in Theorem~\ref{Prop:pseuso-semigroup}\+\ref{IT-EQUIV_SimLin}. This constant will be called the {\sl simultaneous linearization coefficient} of~$\psi$ w.r.t.~$\varphi$.
\end{definition}

The relation $h_\varphi\circ \psi=h_\varphi+c$ for self-maps commuting with a parabolic self-map of zero hyperbolic step~$\varphi$ implies that the semigroup $\Zn(\varphi)$ is \textit{abelian}, that is, ${\psi_1\circ\psi_2=\psi_2\circ\psi_1}$ for any $\psi_1,\psi_2\in\Zn(\varphi)$; see \cite[proof of Theorem~4.5\+(4)]{CDG-zero-h.step}. Moreover, $\Zn(\varphi)$ is isomorphic to a topologically closed subsemigroup of $[\C,+]$:

\begin{theorem}[\protect{\cite[Theorem 4.8]{CDG-zero-h.step}}] \label{Thm:homeomorphism}
Let $\varphi\in\Hol(\UD)$ be parabolic of zero hyperbolic step. Then the map
$$
  \Zn(\varphi)\ni\psi \mapsto c_{\varphi,\psi}\in\C
$$
is a homeomorphism onto a closed subset $\mathcal A_\varphi\subset\C$. Moreover,
\begin{equation}
 c_{\varphi,\psi_1\circ\psi_2}=c_{\varphi,\psi_1}+c_{\varphi,\psi_2}\quad\text{for any~$~\psi_1,\psi_2\in\Zn(\varphi)$}
\end{equation}
and, as a consequence, $[\mathcal A_\varphi,+]$ is an additive semigroup containing~$\N_0$.
\end{theorem}

The first statement in the following theorem provides a direct way to calculate the simultaneous linearization coefficient~$c_{\varphi,\psi}$, that is, without making use of the Koenigs function of~$\varphi$. The second statement explains better the choice of the term ``simultaneous linearization coefficient''.

\begin{theorem}[\protect{\cite[Theorem 5.1]{CDG-zero-h.step}}] \label{TH_SLC(2formulas)}
Let $\varphi\in\Hol(\UD)$ be a parabolic self-map of zero hyperbolic step with Denjoy\,--\,Wolff point $\tau\in\partial\D$ and let $\psi\in\Zn(\varphi)$. The following two statements hold.
\begin{statlist}
 \item\label{IT_SLC1}
   The following angular limit exists and equals the simultaneous linearization coefficient $c_{\varphi,\psi}$ of~$\psi$ \,w.\,r.\,t.~$\varphi$:
   \begin{equation}\label{EQ_SLC}
     \anglim_{z\to\tau}\frac{\psi(z)-z}{\varphi(z)-z}~=~c_{\varphi,\psi}.
   \end{equation}
  \item\label{IT_SLC2}
    Suppose that
	\begin{equation}\label{EQ_SLsystem}
	h\circ \varphi=h+c_1,\qquad h\circ \psi=h+c_2,
	\end{equation}
	for some holomorphic function $h:\UD\to\C$ and  some ${c_1,c_2\in\C}$ with $|c_1|^2+|c_2|^2\neq0$.\\ Then ${c_1\neq0}$ and ${c_2/c_1=c_{\varphi,\psi}}$.
\end{statlist}
\end{theorem}

\section{Simultaneous linearization}\label{S_simmultaneous-lin}
\subsection{Statement of main results}
For non-elliptic holomorphic self-mappings of~$\UD$, the notion of simultaneous linearization  can be introduced as follows.
\begin{definition}
Let $\varphi_1,\ldots,\varphi_n$ be holomorphic self-maps of~$\UD$ such that:
\begin{itemize}
\item[(a)] each of $\varphi_k$'s, ${k=1,\ldots,n}$, is either non-elliptic, or the identity map~$\id_\UD$;
\item[(b)] at least one of these self-maps is different from~$\id_\UD$.
\end{itemize}
We say that the self-maps $\varphi_1,\ldots,\varphi_n$ admit \textsl{simultaneous linearization} if there exists $h\in\Hol(\UD,\C)$ and $c=(c_1,\ldots,c_n)\in\C^n\setminus\{(0,\ldots,0)\}$ satisfying the following system of functional equations:
\begin{equation}\label{EQ_AbelSystem-n}
h\circ\varphi_k=h+c_k,\quad k=1,\dots,n.
\end{equation}
\end{definition}
Throughout this paper, we mostly will need the case of $n=2$ in the above definition. At the same time, this concept can be naturally extended to any (possibly infinite) family~$\mathcal F\subset\Hol(\UD)$ of non-elliptic self-mappings: we say that such a family $\mathcal F$ admits simultaneous linearization if there exist a function ${h\in\Hol(\UD,\C)}$ and a map $c:\mathcal F\to\C$ not vanishing identically such that
$$
{h\circ\psi=h+c(\psi)}\qquad\text{for every~$\,{\psi\in\mathcal F}$.}
$$

One of the main results of this section, stated below, relates commutativity of two non-elliptic self-mappings with their simultaneous linearizability.

Recall that a domain $V\subset\UD$ is said to be \textsl{isogonal} at a point~${\sigma\in\UC}$ if every sequence ${(z_n)\subset\UD}$ converging to~$\sigma$ non-tangentially is contained, except maybe for a finite number of terms, in~$V$.
\begin{theorem}\label{TH_SL-vs-COMM}
Two parabolic self-mappings $\varphi_1,\varphi_2\in\Hol(\UD)$ commute  if and only if the following two conditions hold:
\begin{romlist}
\item\label{IT_SL-vs-COMM-tau} $\varphi_1$ and $\varphi_2$ have the same Denjoy\,--\,Wolff point~$\tau$;
\item\label{IT_SL-vs-COMM-h}  there exists a domain~$V\subset\UD$ isogonal at~$\tau$ and ${(c_1,c_2)\in\C^2\setminus\{(0,0)\}}$ such that the simultaneous linearization problem~\eqref{EQ_AbelSystem-n} for the two self-mappings $\varphi_1$, $\varphi_2$ admits a solution~$h\in\Hol(\UD,\C)$ univalent in~$V$.
\end{romlist}
\end{theorem}

An analogue of this theorem for hyperbolic self-maps holds even without assuming univalence of~$h$ in an isogonal domain~$V$ (see Remark~\ref{RM_for-all-non-elliptic}). We do not know whether such a univalence assumption can be completely avoided in the parabolic case. Nevertheless, it can be considerably weakened, and in fact, we do not have to assume that ${(c_1,c_2)\neq(0,0)}$ when deducing commutativity from conditions \ref{IT_SL-vs-COMM-tau} and~\ref{IT_SL-vs-COMM-h}; see Addendum~\ref{ADD_1}.

For parabolic self-maps of \textit{zero hyperbolic step}, Theorem~\ref{TH_SL-vs-COMM} was essentially established in~\cite{CDG-zero-h.step}. The proof of commutativity, given that conditions~\ref{IT_SL-vs-COMM-tau} and~\ref{IT_SL-vs-COMM-h} are satisfied, does not differ much from the proof of \cite[Proposition~3.3]{CDG-zero-h.step} (regardless of the hyperbolic step). Recall also that commuting self-maps have the same Denjoy\,--\,Wolff point (see Theorem~\ref{TH_Behan-and-Ko}). Therefore, our main task will be to see that if ${\varphi_1\circ\varphi_2=\varphi_2\circ\varphi_1}$, where $\varphi_1$ is parabolic of \textit{positive} hyperbolic step, then $\varphi_1$ and~$\varphi_2$ admit simultaneous linearization, with $h$ being univalent in some domain isogonal at~$\tau$.
Comparing with Theorem~\ref{Prop:pseuso-semigroup}, note that in this case, it is in general not true that the Koenigs map of~$\varphi_1$ solves the simultaneous linearization problem~\eqref{EQ_AbelSystem-n}. Note, however, that every solution to~\eqref{EQ_AbelSystem-n} factorizes via the Koenigs map, see Remark~\ref{RM_factorization}. As a consequence, the following result not only implies, but is in fact \textit{equivalent} to simultaneous linearizability in the parabolic-positive case.

\begin{proposition}\label{PR_beta-exists}
  Let $\varphi\in\Hol(\UD)$ be a parabolic self-map of positive hyperbolic step with canonical model $(S_\varphi,h_\varphi,z\mapsto z+1)$, and let $\psi\in\Zn(\varphi)$. Then there exists $\beta\in\Hol(S_\varphi,\C)$ and a number ${c\in\overline{S_\varphi}}$ such that the following three conditions hold:
  \begin{subequations}\label{EQ_beta(sys)}
	\begin{align}
		\label{EQ_beta(a)} &\beta(w+1)=\beta(w)+1\quad \text{for all $w\in S_\varphi$};\\
		\label{EQ_beta(b)} &\beta \circ h_{\varphi}\circ\varphi = \beta \circ h_{\varphi} + 1;\\
		\label{EQ_beta(c)} &\beta \circ h_{\varphi}\circ\psi = \beta \circ h _{\varphi}+ c.
	\end{align}
   \end{subequations} 	
\end{proposition}
\begin{remark}\label{RM_second-follows-from-first}
It is easy to see that equality~\eqref{EQ_beta(a)} is equivalent to~\eqref{EQ_beta(b)}. Indeed, if \eqref{EQ_beta(a)} holds, then for any~${z\in\UD}$, we have
$
  {\beta \circ h_{\varphi}\circ\varphi(z)=\beta \big(h_{\varphi}(z)+1\big)=\beta\big(h_{\varphi}(z)\big)+1.}
$
Conversely, if \eqref{EQ_beta(b)} holds, then
$
  {\beta\big(h_\varphi(z)+1\big)=\beta\big(h_\varphi(\varphi(z))\big)=\beta\big(h_\varphi(z)\big)+1}
$
for any~${z\in\UD}$ and hence \eqref{EQ_beta(a)} follows by the Identity Principle.
\end{remark}

In general, solutions to~\eqref{EQ_beta(sys)} can be essentially non-unique. However, as we will prove, this may happen only in a rather particular situation. More precisely, we establish the following uniqueness result for simultaneous linearization of commuting parabolic self-maps.
\begin{proposition}\label{PR_beta-uniq}
  Under the hypothesis of Proposition~\ref{PR_beta-exists}, the following two statements hold.
  \begin{statlist}
   \item\label{IT_c-uniq}  There exists a unique constant $c=c_{\varphi,\psi}\in\C$ for which the system~\eqref{EQ_beta(sys)} admits a solution~$\beta\in\Hol(S_\varphi,\C)$.
   \item\label{IT_beta-uniq} If $c_{\varphi,\psi}\not\in\mathbb{Q}$, then any two solutions to the system~\eqref{EQ_beta(sys)} belonging to ${\Hol(S_\varphi,\C)}$ differ by a constant.
  \end{statlist}
\end{proposition}
Note that the uniqueness statement analogous to assertion~\ref{IT_c-uniq} of the above proposition also holds in the case of zero hyperbolic step, see Theorem~\ref{TH_SLC(2formulas)}\+\ref{IT_SLC2}. Therefore, the notion of the simultaneous linearization coefficient defined for parabolic self-maps~$\varphi$ of zero hyperbolic step using the Koenigs map~$h_\varphi$, see Definition~\ref{simul-linea-coefficient-zero}, can be extended to the case of positive hyperbolic step in the following way.

\begin{definition}[simultaneous linearization coefficient]\label{DF_SimLinCoeff}
Let $\varphi\in\Hol(\UD)$ be a parabolic self-map of positive hyperbolic step and let $\psi\in\Zn(\varphi)$. The unique value $c_{\varphi,\psi}$ of the constant~$c$ for which~\eqref{EQ_beta(sys)} admits a solution, see Proposition~\ref{PR_beta-uniq}\+\ref{IT_c-uniq}, will be called the \textsl{simultaneous linearization coefficient} of~$\psi$ w.r.t.\,$\varphi$.
\end{definition}

It turns out that the function~$\beta$ which we construct in the proof of  Proposition~\ref{PR_beta-exists} has some remarkable properties. More precisely:

\begin{proposition}\label{PR_beta-properties}
Under the hypothesis of Proposition~\ref{PR_beta-exists}, every function $\beta\in\Hol(S_\varphi,\C)$ satisfying  for some~${c\in\C\setminus\mathbb Q}$ the system~\eqref{EQ_beta(sys)} has the following properties:
\begin{ourlist}
\item\label{IT_beta-uni} $\beta$ is univalent in~$\big\{w\in S_\varphi\colon|\Im w|>A\big\}$ for ${A>0}$ large enough;
\item\label{IT_beta-limit} $\beta(w)-w$ has finite limit as~${|\Im w|\to+\infty}$, ${w\in S_\varphi}$.
\end{ourlist}
\end{proposition}

Note that~\ref{IT_beta-uni} and~\ref{IT_beta-limit} may fail for \textit{some} of the solutions\footnote{The solution set of~\eqref{EQ_beta(sys)} in the case ${c_{\varphi,\psi}\in\mathbb Q}$ is described in Remark~\ref{RM_rational-SLC}.} to~\eqref{EQ_beta(sys)} if ${c_{\varphi,\psi}\in\mathbb Q}$. At the same time, if the system~\eqref{EQ_beta(sys)} admits for a given~${c\in\Real}$ at least one solution belonging to~${\Hol(S_\varphi,\C)}$, then ${\beta:=\id_{S_\varphi}}$ is necessarily among the solutions; see Remark~\ref{RM_identity-is-a-solution2}.
This allows us to select, with the help of the normalization
\begin{equation}\label{EQ_beta-normalization}
 \lim_{\substack{w\in S_\varphi,\\|\Im w|\to+\infty}}\!\!\! \big(\beta(w)-w\big)=0,
\end{equation}
a sort of ``canonical'' solution~$\beta_{\varphi,\psi}$, which we will call the simultaneous linearization function:
\begin{definition}\label{DF_SimLinFunc}
Let $\varphi$ and $\psi$ be as in Definition~\ref{DF_SimLinCoeff}. By the \textsl{simultaneous linearization function}~$\beta_{\varphi,\psi}$ associated with $\varphi$ and $\psi$ we mean the unique holomorphic function ${\beta:S_\varphi\to\C}$ satisfying \eqref{EQ_beta(sys)} and~\eqref{EQ_beta-normalization} if ${c_{\varphi,\psi}\not\in\mathbb Q}$, see Proposition~\ref{PR_beta-uniq}\+\ref{IT_beta-uniq}, or the identity map~$\beta:=\id_{S_\varphi}$ if~${c_{\varphi,\psi}\in\mathbb Q}$.
\end{definition}

Proposition~\ref{PR_beta-uniq} allows us to extend Theorem~\ref{TH_SLC(2formulas)}\+\ref{IT_SLC2} to the case of parabolic self-maps with positive hyperbolic step.

\begin{corollary}\label{Thm:StrongUniqueness} Let $\varphi,\psi\in\Hol(\D)$ be two commuting  parabolic self-maps. Suppose that a holomorphic function ${h:\UD\to\C}$ satisfies the identities
	\begin{equation}\label{EQ_h-system}
	h\circ \varphi=h+c_1\quad\text{and}\quad h\circ \psi=h+c_2
	\end{equation}
	for some $(c_1,c_2)\in\C^2\setminus\{(0,0)\}$. Then ${c_1\neq0}$ and $c_2/c_1=c_{\varphi,\psi}$. Moreover, if $\varphi$ is of positive hyperbolic step and ${c_{\varphi,\psi}\notin\Q}$, then there exists $c_0\in\C$ such that
	\begin{equation}\label{EQ_StrongUniqueness}
	h=c_1\beta_{\varphi,\psi}\circ h_\varphi+c_0.
	\end{equation}	
\end{corollary}

Note that a relation similar to~\eqref{EQ_StrongUniqueness} holds also when $\varphi$ is of zero hyperbolic step and ${c_{\varphi,\psi}\notin\Q}$; see \cite[Proposition~5.2]{CDG-zero-h.step}.

\begin{remark}\label{RM_StrongUniqueness}
Corollary~\ref{Thm:StrongUniqueness} connects simultaneous linearization we study in this paper with the ideas developed in~\cite{Simultaneous}.
Indeed, for a parabolic $\varphi\in\Hol(\UD)$ and $\psi\in\Zn(\varphi)$, formula~\eqref{EQ_StrongUniqueness} means that every holomorphic function ${h:\UD\to\C}$ satisfying system~\eqref{EQ_h-system} factorizes via our ``canonical'' solution ${\beta_{\varphi,\psi}\circ h_\varphi}$ to the simultaneous linearization problem. Moreover, this conclusion holds even if ${(c_1,c_2)=(0,0)}$ or $c_{\varphi,\psi}\in\mathbb{Q}$. To see this, note that every solution to~\eqref{EQ_h-system} can be written as $h=h_0+c_1 \beta_{\varphi,\psi}\circ h_\varphi$, where $h_0\in\Hol(\UD,\C)$ satisfies the homogeneous version of~\eqref{EQ_h-system}, i.e.
\begin{equation}\label{EQ_homo}
  {h_0\circ\varphi}\,=\,{h_0\circ\psi}\,=\,h_0.
\end{equation}
Using a modification of the argument in the proof of Corollary~\ref{Thm:StrongUniqueness}, it is not difficult to show that if ${c_{\varphi,\psi}\not\in\mathbb Q}$, then all holomorphic solutions to~\eqref{EQ_homo} are constant. Similarly, if ${c_{\varphi,\psi}=p/q}$ for some relatively prime ${p\in\mathbb Z}$ and~$q\in\Natural$, then holomorphic solutions to~\eqref{EQ_homo} are of the form $h_0={f\circ h_\varphi}={f\circ\beta_{\varphi,\psi}\circ h_\varphi}$, where $f$ is a $1/q$-periodic holomorphic function in~$S_\varphi$.
\end{remark}

\smallskip
The existence of simultaneous linearization for commuting parabolic self-mappings, if considered on its own, is not a new result. Indeed, for a pair of such self-mappings $\varphi$ and~$\psi$, the result by Cowen \cite[Theorem~3.1]{Cowen-comm} yields\footnote{Cowen's result is stated using the notion of \textsl{pseudo-iteration semigroup}. We briefly explain its relation to simultaneous linearization in~\cite[Section~3]{CDG-zero-h.step}.} that the Koenigs map~$h_{\varphi\circ\psi}$ of the composition ${\varphi\circ\psi}$ solves the simultaneous linearization problem.
Although  $\beta_{\varphi,\psi}\circ h_\varphi\,$ \textit{a posteriori}  turns out to coincide, up to a suitable linear transformation with~$h_{\varphi\circ\psi}$, see Corollary~\ref{CR_Cowen}, our construction seems to be quite different from that by Cowen. Besides the uniqueness of the simultaneous linearization provided by Proposition~\ref{PR_beta-uniq}, our approach allows us to characterize ``abelian'' subsets of the centralizer $\Zn(\varphi)$. Namely, we will prove the following theorem.
\begin{theorem}\label{TH_abelian-part0}
Let $\varphi\in\Hol(\UD)$ be a parabolic self-map of positive hyperbolic step, and let $\Delta\subset\Zn(\varphi)$ be a non-empty set. Then the following two conditions are equivalent:
\begin{equilist}
  \item\label{IT_abelian}
     $\psi_1\circ\psi_2=\psi_2\circ\psi_1$ for any $\psi_1,\psi_2\in\Delta$;
  \item\label{IT_comm-sim-lin}
   the family $\mathcal F:=\Delta\cup\{\varphi\}$ admits simultaneous linearization.
\end{equilist}
\end{theorem}
\begin{remark}
Note that the above theorem holds trivially when~$\varphi$ is parabolic of zero hyperbolic step, because in such a case, conditions \ref{IT_abelian} and~\ref{IT_comm-sim-lin} hold for any ${\Delta\subset\Zn(\varphi)}$, see Section~\ref{S_when-zero}.
\end{remark}
For the case of positive hyperbolic step, Theorem~\ref{TH_abelian-part0} can be deduced from a more technical Theorem~\ref{TH_abelian-part} (see page~\pageref{TH_abelian-part}).

\begin{remark}\label{RM_comparing-two-theorems}
It is worth emphasizing some principal differences between Theorems~\ref{TH_SL-vs-COMM} and~\ref{TH_abelian-part0}. Both theorems relate commutativity with simultaneous linearization. However, note that ``pairwise'' simultaneous linearizability, i.e. simultaneous linearizability for any pair of elements, does not \textit{a priori} imply simultaneous linearization for the whole family. Moreover, in Theorem~~\ref{TH_abelian-part0}, in contrast to Theorem~\ref{TH_SL-vs-COMM}, we do not need any univalence condition on~$h$. At the same time, the latter theorem applies to arbitrary pairs of parabolic self-mappings, while the former is restricted to families of self-mappings commuting with a fixed parabolic self-map~$\varphi$ of positive hyperbolic~step.
\end{remark}

The last result we include in this section concerns the injectivity of the map
$$
\mathfrak S_\varphi:\Zn(\varphi)\to \Complex;\,\psi \mapsto c_{\varphi,\psi}.
$$
If the parabolic self-map $\varphi$ is of positive hyperbolic step, then in general this map is not injective; see e.g. \cite[Section~8 and Remark~8.8]{CDG-zero-h.step} for relevant examples. However, restrictions of this map to any ``abelian'' subset of the centralizer, i.e. to any subset ${\Delta\subset\Zn(\varphi)}$ satisfying condition~\ref{IT_abelian} in Theorem~\ref{TH_abelian-part0}, is injective, as we show in the following corollary of Theorem~\ref{TH_abelian-part}.
\begin{corollary}\label{CR_injective-on-abelian} 	
Let $\varphi\in\Hol(\D)$ be a parabolic self-map and let ${\psi_1,\psi_2\in\Zn(\varphi)}$.
If ${\psi_1\circ\psi_2}={\psi_2\circ\psi_1}$ and ${c_{\varphi,\psi_1}=c_{\varphi,\psi_2}}$, then ${\psi_1=\psi_2}$.
\end{corollary}

\medskip

\subsection{Auxiliary statements and proofs}
We start with a couple of simple but useful observations.
\begin{remark}\label{Re:another}
	Let $\varphi$ be a self-map of a domain~$\Omega$, $\psi$ another self-map of~$\Omega$ commuting with~$\varphi$, and let $h:\Omega\to \C$ be a solution to Abel's equation for $\varphi$, that is, ${h\circ \varphi=h+1}$. Then the function $h_1:=h\circ \psi$ is another solution to Abel's equation for $\varphi$. Indeed,
	$$
	h_{1}\circ \varphi=h\circ \psi\circ \varphi=h\circ \varphi\circ \psi=h\circ \psi+1=h_{1}+1.
	$$
\end{remark}

Using the above remark, one can directly check the following result.
\begin{lemma}\label{LM_NSemimodel}
Let $\varphi\in\Hol(\D)$ be a non-elliptic self-map with canonical holomorphic model $\mathcal M_{\varphi}:={\big(S_\varphi,h_\varphi,z\mapsto z+1\big)}$ and let ${\psi\in\Zn(\varphi)}$. Then
\begin{equation}\label{EQ_NSemimodel1}
       S_{\varphi}^{\psi}:=\bigcup_{n\geq0}\left(h_\varphi(\psi(\D))-n\right)
\end{equation}
	is a subdomain of the domain~$S$. Moreover, the triple
\begin{equation}\label{EQ_NSemimodel2}
\mathcal M_\varphi^\psi:=\big(S_\varphi^\psi,h_\varphi\circ\psi,z\mapsto z+1\big)
\end{equation}
 is a semimodel of $\varphi$. 	
\end{lemma}

\begin{definition}\label{DF_g-psi} We will call the semimodel~$\mathcal M_\varphi^\psi$ defined by~\eqref{EQ_NSemimodel1}--\eqref{EQ_NSemimodel2} the \textsl{natural semimodel of $\varphi$ associated with $\psi$}. Moreover, the unique morphism of semimodels ${g_\psi: S_\varphi \to S_\varphi^\psi}$ (see \cite[Lemma 3.5.8]{Abate2})  will be referred to as the \textsl{semimodel morphism associated with~$\psi$}.
\end{definition}

\begin{remark}\label{RM_morphism-properties}
Let $\varphi$ and $\psi$ be as in Lemma~\ref{LM_NSemimodel}. It is known (see e.g. \cite[Definition~3.5.5 and Remark~3.5.6]{Abate2}) that the semimodel morphism~$g_\psi$ associated with~$\psi$ has the following properties:
\begin{align*}
    & g_\psi\in\Hol(S_\varphi), \hphantom{\quad\text{and}} g_\psi(S_\varphi)=S_\varphi^\psi, &&\\
	& g_\psi\circ h_\varphi=h_\varphi\circ \psi, \quad\text{and}\quad g_\psi(w+1)=g_\psi(w)+1~
                                                                   \text{~for all $~w\in S_\varphi$}.&&
\end{align*}
\end{remark}

Now we are ready to prove Proposition~\ref{PR_beta-exists} asserting that commuting parabolic self-maps $\varphi$ and~$\psi$, with $\varphi$ being of positive hyperbolic step, admit simultaneous linearization of the form ${h=\beta\circ h_\varphi}$.
\begin{proof}[\proofof{Proposition~\ref{PR_beta-exists}}]\label{PG_beta-exists}
 We may assume that the base space of the canonical model of $\varphi$ is the upper half-plane $\H$, i.e. ${S_\varphi=\H}$.  The same argument, with obvious modifications, applies also for the other case, i.e. for the case ${S_\varphi=-\H}$.
	
Suppose first that the semimodel morphism~$g_\psi$ associated with~$\psi\in\Zn(\varphi)$ is an automorphism of~$\UH$. By the last equality in Remark~\ref{RM_morphism-properties} and by~\cite[Corollary~1.6.21\+(iii)]{Abate2}, in this case we have that ${g_\psi=\id_\H+\theta}$ for some constant ${\theta\in\R}$. It follows that $\beta:=\id_\UH$ satisfies~\eqref{EQ_beta(sys)} with $c:=\theta$.

Now suppose that $g_\psi\not\in\mathsf{Aut}(\H)$. Then according to Propositions~\ref{Prop:com-with-autoprevio} and~\ref{Prop:com-with-auto}, $g_\psi$ is a parabolic self-map of $\H$ of zero hyperbolic step  with Denjoy\,--\,Wolff point at~$\infty$. Moreover, there exists $F_\psi\in\Hol(\D,\H)$ such that
	$$
	g_\psi(w)=w+F_\psi(e^{2\pi i w}), \quad w\in\H.
	$$
	Consider the canonical holomorphic model $\big(\C,h_{g_\psi},w\mapsto w+1\big)$ of~$g_\psi$. 	Then $\beta:={F_\psi(0)h_{g_\psi}}\in{\Hol(\H,\C)}$ satisfies~\eqref{EQ_beta(sys)} with ${c:=F_\psi(0)}$. Indeed, identities \eqref{EQ_beta(a)} and~\eqref{EQ_beta(b)} hold by Proposition~\ref{Prop:com-with-auto}\+(6\+c) and Remark~\ref{RM_second-follows-from-first}.
Moreover, recalling that $h_{\varphi}\circ \psi=g_\psi\circ h_\varphi$  (see Remark~\ref{RM_morphism-properties}) and using Abel's equation ${h_{g_\psi}\circ g_\psi}={h_{g_\psi}+1}$, we obtain
$$
 F_\psi(0)\, h_{g_\psi}\circ h_{\varphi}\circ \psi(z)=F_\psi(0)\, h_{g_\psi}\circ g_{\psi}\big( h_{\varphi}(z)\big)=F_\psi(0) \, h_{g_\psi} \big( h_{\varphi}(z)\big)+F_\psi(0)\quad\text{for all~$\,z\in\UD$},
$$
that is, $\beta$ satisfies also~\eqref{EQ_beta(c)} with~$c:=F_\psi(0)$, as desired.
\end{proof}

\begin{proof}[\proofof{Propositions~\ref{PR_beta-uniq} and~\ref{PR_beta-properties}}]
As before, we will suppose that the base space~$S_\varphi$ of the canonical model for~$\varphi$ coincides with~$\UH$. The argument can be easily adapted to give the proof in the case ${S_\varphi=-\UH}$.

Suppose that some $\beta\in\Hol(\UH,\C)$ and ${c\in\C}$ satisfy~\eqref{EQ_beta(sys)}. As in the proof of~Proposition~\ref{PR_beta-exists}, we can write ${h_\varphi\circ\psi= g_\psi\circ h_\varphi}$. At the same time, by~\eqref{EQ_beta(c)} we have ${\beta\circ h_\varphi\circ\psi=\beta\circ h_\varphi+c}$. Since the function $h_\varphi$ is not constant, it follows with the help of the Identity Principle that
\begin{equation}\label{EQ_beta-g-psi}
\beta\circ g_\psi=\beta+c.
\end{equation}

The remaining  part of the argument depends on whether $g_\psi$ is an automorphism of~$\UH$ or not.

\StepC{1}{$g_\psi\in\mathsf{Aut}(\UH)$}\label{PG_irrational} Recall from the proof of Proposition~\ref{PR_beta-exists} that in this case, ${g_\psi(w)=w+\theta}$ for all ${w\in\UH}$ and  some constant ${\theta\in\Real}$. By~\eqref{EQ_beta(a)} and~\eqref{EQ_beta-g-psi}, we have
\begin{equation}\label{EQ_when-automor}
\beta(w+1)=\beta(w)+1\quad\text{and}\quad\beta(w+\theta)=\beta(w)+c\quad\text{for all~$~w\in\UH$}.
\end{equation}
As a consequence, $\beta'$ is $1$-periodic and at the same time $\theta$-periodic. If ${\theta\not\in\mathbb Q}$,  then the set $\{m+n\theta:\,  {m,n\in \Z}\}$ is dense in $\R$ and by continuity, it follows that $\beta'$ is constant on each straight line in~$\UH$ parallel to~$\Real$. Since $\beta'$ is holomorphic, it further follows that~$\beta'$ is constant in~$\UH$. Taking into account~\eqref{EQ_when-automor} again, we conclude that~${\beta'\equiv1}$ and hence, ${c=\theta}$. This immediately implies assertions~\ref{IT_c-uniq}, \ref{IT_beta-uniq}, \ref{IT_beta-uni} and~\ref{IT_beta-limit} in the case when ${g_\psi(w)=w+\theta}$ with~${\theta\in\Real\setminus\mathbb Q}$.

Suppose now that $\theta$ is a rational number, i.e. $\theta=p/q$ for some $p\in\mathbb Z$ and ${q\in\Natural}$. Then from~\eqref{EQ_when-automor}  it follows that for any~${w\in\UH}$,
$$
\beta(w)+p=\beta(w+p)=\beta(w+q\theta)=\beta(w)+qc.
$$
Therefore, we have $p=qc$, i.e. the only value of~$c$ for which \eqref{EQ_beta(sys)} admits solutions is again ${c=\theta}$. This proves assertion~\ref{IT_c-uniq}, while~\ref{IT_beta-uniq} and the whole Proposition~\ref{PR_beta-properties} hold trivially in this case, because~${c\in\mathbb Q}$.

\StepC{2}{$g_\psi\not\in\mathsf{Aut}(\UH)$}\label{PAGE_Case2} In the proof of Proposition~\ref{PR_beta-exists} we have seen that in this case, $g_\psi$ is a parabolic self-map of~$\UH$ with the canonical model of the form ${\big(\C,h_{g_\psi},w\mapsto w+1\big)}$. According to Lemma~\ref{Le:morphism}, identity~\eqref{EQ_beta-g-psi} implies that $\beta$ factorizes via~$h_{g_\psi}$, see Remark~\ref{RM_factorization}, i.e. ${\beta=f\circ h_{g_\psi}}$ for an entire function ${f\in\Hol(\C)}$ satisfying
\begin{equation}\label{EQ_f-prop1}
f(\zeta+1)=f(\zeta)+c,\quad \zeta\in\C.
\end{equation}

Moreover, according to Propositions~\ref{Prop:com-with-autoprevio} and~\ref{Prop:com-with-auto},
$h_{g_\psi}(w+1)=h_{g_\psi}(w)+1/c_0$ for all ${w\in\UH}$, where $c_0:={\lim\limits_{\Im w\to+\infty}\big(g_\psi(w)-w\big)\in\UH}$. In combination with relation~\eqref{EQ_beta(a)}, this yields
$$
f\big(h_{g_\psi}(w)\big)+1=\beta(w)+1=\beta(w+1)=f\big(h_{g_\psi}(w+1)\big)=f\big(h_{g_\psi}(w)+1/c_0\big)
$$
for all ${w\in\UH}$. With the help of the Identity Principle, we see that
\begin{equation}\label{EQ_f-prop2}
f(\zeta+1/c_0)=f(\zeta)+1,\quad \zeta\in\C.
\end{equation}

According to \eqref{EQ_f-prop1} and~\eqref{EQ_f-prop2}, the entire function~$f'$ is doubly periodic with periods ${\omega_1=1}$ and ${\omega_2=1/c_0}$. Since $\omega_2/\omega_1\notin \R$  and since $f'$ has no poles,  this function must be constant; see e.g. \cite[Chapter~1, \S4]{Akh}. Using~\eqref{EQ_f-prop2}, we then conclude that ${f'/c_0\equiv1}$. Similarly, by~\eqref{EQ_f-prop1} we have~${f'\equiv c}$. It follows that the only value of~$c$ for which~\eqref{EQ_beta(sys)} admits a holomorphic solution ${\beta\in\Hol(\UH,\C)}$ is ${c=c_0}$ and that any two solutions belonging to ${\Hol(\UH,\C)}$ differ by a constant, i.e. assertions~\ref{IT_c-uniq} and~\ref{IT_beta-uniq} hold.

It remains to prove properties~\ref{IT_beta-uni} and~\ref{IT_beta-limit}. To this end, we appeal to  Propositions~\ref{Prop:com-with-autoprevio} and~\ref{Prop:com-with-auto}. In view of the fact that ${\beta=f\circ h_{\psi_g}}$, Proposition~\ref{Prop:com-with-auto}\+(6\+a) implies property~\ref{IT_beta-uni}, while~\ref{IT_beta-limit} follows from Proposition~\ref{Prop:com-with-auto}\+(1)~and~(6\+d).
\end{proof}

It is now worth making some further (simple but nevertheless useful) observations.

\begin{remark}\label{RM_identity-is-a-solution}
Let $\varphi\in\Hol(\UD)$ be a parabolic self-map of positive hyperbolic step and let ${\psi\in\Zn(\varphi)}$. Then ${\beta:=\id_{S_\varphi}}$ is a solution to the system~\eqref{EQ_beta(sys)} for some fixed~${c\in\Complex}$ if and only if the semimodel morphism associated with~$\psi$ is given by $g_{\psi}(w)={w+c}$ for all ${w\in S_\varphi}$.
Indeed, the identity map satisfies the first two equations in~\eqref{EQ_beta(sys)}, while the third equation~\eqref{EQ_beta(c)}  is equivalent, in view of the identity ${h_\varphi\circ \psi}={g_\psi\circ h_\varphi}$ (see Remark~\ref{RM_morphism-properties}), to
$$
  \beta\circ g_\psi\circ h_\varphi=\beta\circ h_\varphi+c,
$$
and the desired conclusion follows immediately, thanks to the Identity Principle.
\end{remark}

\begin{remark}\label{RM_identity-is-a-solution2}
It is evident from the proofs of Propositions~\ref{PR_beta-exists}, \ref{PR_beta-uniq} and~\ref{PR_beta-properties}, given above, that if the simultaneous linearization coefficient~$c_{\varphi,\psi}$ is real, then $g_\psi(w)={w+c_{\varphi,\psi}}$ for all~${w\in S_\varphi}$, and hence, in view of the above remark, we have that ${\beta:=\id_{S_\varphi}}$ satisfies the system~\eqref{EQ_beta(sys)} with~${c:=c_{\varphi,\psi}}$.
\end{remark}

\begin{remark}\label{RM_real-SLC}
The previous remark means, in other words, that if $c_{\varphi,\psi}\in\Real$, then:
\begin{itemize}
\item[(a)] the simultaneous linearization function $\beta_{\varphi,\psi}$ coincides with~$\id_{S_\varphi}$, and
\item[(b)] the Koenigs map~$h_\varphi$ of~$\varphi$ solves the simultaneous linearization problem~\eqref{EQ_AbelSystem-n} for the two self-maps ${\varphi_1:=\varphi}$, ${\varphi_2:=\psi}$ and the constant vector $c:={\big(1,\,c_{\varphi,\psi}\big)}$.
\end{itemize}
\end{remark}

\begin{remark}\label{RM_rational-SLC}
If furthermore, $c_{\varphi,\psi}\in\mathbb Q$, i.e. if $c_{\varphi,\psi}=p/q\in\mathbb Q$ with some relatively prime ${p\in\mathbb Z}$ and ${q\in\mathbb N}$, then $\beta\in\Hol(S_\varphi,\C)$ solves~\eqref{EQ_beta(sys)} with ${c:=c_{\varphi,\psi}}$ if and only if it satisfies the identity
\begin{equation}\label{EQ_identity-rational}
\beta(w+1/q)=\beta(w)+1/q \quad\text{for all~$\,w\in S_\varphi$}.
\end{equation}
Indeed, if this identity holds, then we also have ${\beta(w+1)=\beta(w)+1}$ and ${\beta(w+c_{\varphi,\psi})}={\beta(w)+c_{\varphi,\psi}}$ for all ${w\in S_\varphi}$. In combination with assertion~(b) of the previous remark, this implies~\eqref{EQ_beta(sys)}.

Conversely, if \eqref{EQ_beta(sys)} holds with ${c:=c_{\varphi,\psi}}$, then using again Remark~\ref{RM_real-SLC}\+(b) and the Identity Principle, from~\eqref{EQ_beta(c)} we deduce ${\beta(w+p/q)}={\beta(w)+p/q}$ for all ${w\in S_\varphi}$. Note also that ${\beta(w+1)}={\beta(w)+1}$ for all ${w\in S_\varphi}$ by~\eqref{EQ_beta(a)}. Since $p$ and~$q$ are relatively prime,  these two identities imply\footnote{Indeed, since $p$ and $q$ are relatively prime, by B\'ezout's Identity, see e.g. \cite[Theorem~2-3]{Burton},  there exists $m,n\in \Z$ such that $mp+nq=1$. Then
 $$
 \beta(w+1/q)=\beta(w+mp/q+n)= \beta\big(w\big)+mp/q+n=\beta(w)+1/q.
 $$
}~\eqref{EQ_identity-rational}.
\end{remark}

\begin{remark}\label{Re:real-coefficient2} Note that ${c_{\varphi,\id_\UD}=0}$, because  ${\beta:=\id_{S_\varphi}}$ satisfies~\eqref{EQ_beta(sys)} for ${c:=0}$ and ${\psi:=\id_\D}$. Moreover, ${c_{\varphi,\psi}=0}$ if and only if~${\psi=\id_\UD}$. Indeed, suppose that $c_{\varphi,\psi}=0$. Then by Remark~\ref{RM_real-SLC}\+(b),  ${h_\varphi=h_\varphi\circ\psi}$. According to Theorem~\ref{Thm:V-sets} and Lemma~\ref{Le:V-sets}, $h_\varphi$ is univalent in $V:=V_\varphi(1/3)$ and $\psi(V)\subset V$. Therefore, $\psi|_V=\id_V$ and so, by the Identity Principle, ${\psi=\id_\D}$.
\end{remark}

Now we start preparation for the proof of Theorem~\ref{TH_SL-vs-COMM}.
\begin{proposition}\label{PR_univ-in-isogonal-dom}
   Let $\varphi\in\Hol (\UD)$ be a parabolic self-map of positive hyperbolic step with Denjoy\,--\,Wolff point~$\tau\in\UC$ and let ${\psi\in\Zn(\varphi)}$. Then there exists a domain ${V\subset\UD}$ isogonal at~$\tau$ such that the function $h:={\beta_{\varphi,\psi}\circ h_\varphi}$ is univalent in~$V$.
\end{proposition}
\begin{proof}
By Theorem~\ref{Thm:V-sets}, the Koenigs function~$h_{\varphi}$ is univalent in the domain $V_1:={V_{\varphi}(1/3)}$ and $V_1$~is isogonal at~$\tau$. Moreover, recalling the definition of the simultaneous linearization function~$\beta_{\varphi,\psi}$ and appealing to Proposition~\ref{PR_beta-properties},  we see that for a suitable ${A>0}$, the function~$\beta_{\varphi,\psi}$ is univalent in the horizontal half-plane $\Pi_A:={\{w\in S_{\varphi}\colon|\Im w|>A\}}$. It follows that $h$ is univalent in $V_2:={V_1\cap h_{\varphi}^{-1}(\Pi_A)}$. We have to show that $V_2$ contains a domain isogonal at~$\tau$. To this end it suffices to fix an arbitrary ${\theta\in(0,\pi/2)}$ and check that there exists ${r>0}$ such that the truncated Stolz angle $A_{r,\theta}:={\{z\in\UD\colon |\Arg(1-\overline\tau z)|<\theta,~|z-\tau|<r\}}$ is contained in~$V_2$. The latter is indeed the case because $V_1$ is isogonal at~$\tau$ and because ${\angle\lim_{z\to\tau} |\Im h_{\varphi}(z)|=+\infty}$; see Remark~\ref{Re:Pommerenke}.
\end{proof}
\begin{lemma}\label{LM_close-to-linear}
Let $f_1,f_2\in\Hol(\UD)$ and let~${\sigma\in\UC}$. Suppose that $f_1(\sigma)=f_2(\sigma)=\sigma$ in the sense of non-tangential limits and that the angular derivatives\footnote{The existence of angular derivative, finite or infinite, follows from the Julia\,--\,Wolff\,--\,Carathéodory Theorem, see e.g. \cite[Sect.\,2.3]{Abate2}.} satisfy~${f'_1(\sigma)=f'_2(\sigma)\neq\infty}$. Then
\begin{equation}\label{EQ_almost-linear-hyperbolic}
 \anglim_{z\to\sigma}\rho_\UD\big(f_1(z),f_2(z)\big)=0.
\end{equation}
\end{lemma}
\begin{proof}
 Using the Cayley transform $\zeta=(1+\overline\sigma z)/(1-\overline\sigma z)$ we can replace $f_1$, $f_2$ by two holomorphic self-maps $F_1$, $F_2$ of the right half-plane~$\HR$ such that
 $$
   \anglim_{\zeta\to\infty}F_1(\zeta)=\anglim_{\zeta\to\infty}F_2(\zeta)=\infty\quad\text{and}\quad
     a:=\anglim_{\zeta\to\infty}\frac{F_1(\zeta)}{\zeta}=\anglim_{\zeta\to\infty}\frac{F_2(\zeta)}{\zeta}>0.
 $$
We have
\begin{equation}\label{EQ_euclid-termsP}
  \anglim_{\zeta\to\infty}\left(\tfrac{F_k(\zeta)}{a|\zeta|}-\tfrac{\zeta}{|\zeta|}\right)=
  \anglim_{\zeta\to\infty}\left[\tfrac{\zeta}{|\zeta|}\Big(\tfrac{F_k(\zeta)}{a\zeta}-1\Big)\right]=0
     \qquad\text{for~$\,k=1,2$.}
\end{equation}

Fix some arbitrary $\theta\in(0,\pi/2)$ and denote $A_\theta:={\{\zeta\in\HR: |\arg\zeta|\le\theta\}}$.  Obviously, the set ${\{\zeta/|\zeta|:\zeta\in A_\theta\}}$ is a compact subset of~$\HR$.
Since the hyperbolic distance is invariant w.r.t. automorphisms and since it is locally equivalent to the euclidean distance, from~\eqref{EQ_euclid-termsP} it follows that
$$
   \rho_{\HR}\big(F_1(\zeta),F_2(\zeta)\big)
           =
   \rho_{\HR}\big(\tfrac{F_1(\zeta)}{a|\zeta|},\,\tfrac{F_2(\zeta)}{a|\zeta|}\big)
          \le
   \rho_{\HR}\big(\tfrac{F_1(\zeta)}{a|\zeta|},\,\tfrac{\zeta}{|\zeta|}\big)
          +
    \rho_{\HR}\big(\tfrac{F_2(\zeta)}{a|\zeta|},\,\tfrac{\zeta}{|\zeta|}\big)
         ~\to~0
$$
as $A_\theta\ni\zeta\to+\infty$. Taking into account arbitrariness of~$\theta$ and returning back to~$\UD$, the desired conclusion~\eqref{EQ_almost-linear-hyperbolic} follows easily.
\end{proof}

Now we are ready to prove Theorem~\ref{TH_SL-vs-COMM}. Note that one of the implications can be proved under a weaker condition concerning the solution~$h$ of the simultaneous linearization problem. Namely, we have the following addendum to Theorem~\ref{TH_SL-vs-COMM}.

\begin{addendum}\label{ADD_1}
  Condition~\ref{IT_SL-vs-COMM-h} in Theorem~\ref{TH_SL-vs-COMM} can be replaced by the following weaker condition:
  \begin{itemize}
  \item[{\rm(ii${}'$\hspace{.05em})}] there exists ${(c_1,c_2)\in\C^2}$ and ${h\in\Hol(\UD,\C)}$ satisfying the simultaneous linearization problem~\eqref{EQ_AbelSystem-n} for~$\varphi_1$,~$\varphi_2$ and having the following property: there is ${r>0}$ and a sequence ${(z_n)\subset\UD}$ converging non-tangentially to~$\tau$ such that $h$ is univalent in each of the hyperbolic discs~$\phdisk(z_n,r)$, ${n\in\Natural}$.
  \end{itemize}
\end{addendum}
Note that in contrast to condition~\ref{IT_SL-vs-COMM-h} of Theorem~\ref{TH_SL-vs-COMM}, the case ${(c_1,c_2)=(0,0)}$ in the above addendum is \textit{not} excluded.

\begin{proof}[\proofof{Theorem~\ref{TH_SL-vs-COMM} and Addendum~\ref{ADD_1}}]
First assume that $\varphi_1$ and $\varphi_2$ commute. Then condition~\ref{IT_SL-vs-COMM-tau} holds by~Theorem~\ref{TH_Behan-and-Ko}.
Furthermore, if one of these two parabolic self-maps is of zero hyperbolic step, then by Theorem~\ref{Prop:pseuso-semigroup}, its Koenigs map solves the simultaneous linearization  problem~\eqref{EQ_AbelSystem-n}. Hence, in this case, condition~\ref{IT_SL-vs-COMM-h} holds in view of Theorem~\ref{Thm:V-sets}. Suppose now that the self-maps $\varphi_1$, $\varphi_2$ are of positive hyperbolic step. Then according to Proposition~\ref{PR_beta-exists}, see also Definitions~\ref{DF_SimLinCoeff} and~\ref{DF_SimLinFunc}, the function $h:={\beta_{\varphi_1,\varphi_2}\circ h_{\varphi_1}}$ satisfies~\eqref{EQ_AbelSystem-n} with  $c:=(1,\,c_{\varphi_1,\varphi_2})$, and moreover by Proposition~\ref{PR_univ-in-isogonal-dom}, $h$ is univalent in some domain~$V\subset\UD$ isogonal at~$\tau$. Thus, condition~\ref{IT_SL-vs-COMM-h} holds as well.

Now we prove the converse implication assuming the weaker condition~{\rm(ii${}'$\hspace{.05em})} from Addendum~\ref{ADD_1} instead of~\ref{IT_SL-vs-COMM-h}. From~\eqref{EQ_AbelSystem-n} it follows that
\begin{equation}\label{EQ_h-with-compositions}
h\circ\varphi_1\circ\varphi_2=h\circ\varphi_2+c_1=h+c_1+c_2=h\circ\varphi_1+c_2=h\circ \varphi_2\circ\varphi_1.
\end{equation}
Note that $f_1:=\varphi_1\circ\varphi_2$ and $f_2:=\varphi_2\circ\varphi_1$ are holomorphic self-maps with the Denjoy\,--\,Wolff point at~$\tau$, and ${f_1'(\tau)=f_2'(\tau)}$. In particular, it follows that there is a sequence~$(\zeta_n)\subset\UD$ converging non-tangentially to~$\tau$ such that ${f_1(\zeta_n)=z_n}$ for all ${n\in\Natural}$ large enough, where $(z_n)$ is the sequence from condition~{\rm(ii${}'$\hspace{.05em})}. (This follows, e.g., from \cite[Proposition~A.6]{GuKouMouRoth}.) Using Lemma~\ref{LM_close-to-linear}, we may conclude that there exist~${n_0\in\Natural}$ and $r>0$ such that both points ${f_1(\zeta_{n_0})}$ and ${f_2(\zeta_{n_0})}$ are contained in~${\phdisk(z_{n_0},r)}$. In view of univalence of~$h$ in~${\phdisk(z_{n_0},r)}$, equality~\eqref{EQ_h-with-compositions} then implies that $f_1$ coincides with~$f_2$ in some neighbourhood of~$\zeta_{n_0}$. Thanks to the Identity Principle, it further follows that ${f_1=f_2}$ in~$\UD$, i.e. $\varphi_1$ commutes with~$\varphi_2$ as desired.
\end{proof}
\begin{remark}\label{RM_for-all-non-elliptic}
We note that Theorem~\ref{TH_SL-vs-COMM} holds for all non-elliptic self-maps. Indeed, if $\varphi_1$ and $\varphi_2$ commute and one of these self-maps is hyperbolic, then similarly to the case of parabolic self-maps of zero hyperbolic step, the simultaneous linearization problem is solved by the Koenigs map. (This can be deduced from Cowen's result~\cite[Theorem~3.1]{Cowen-comm}; see also \cite[Proposition~4.3]{CDG-Centralizer} for the univalent case.) It remains to notice that for the hyperbolic case, the absorbing domain~${V\subset\UD}$ in condition~(HM3) of Definition~\ref{DF_holomorphic-model} is isogonal at the Denjoy\,--\,Wolff point~$\tau$ because there are orbits converging to~$\tau$ at any slope in~${(-\pi/2,\pi/2)}$, see e.g. \cite[Property~(2)\,(b)]{BracciPC}. As for the converse implication, note that the argument given in the above proof of Theorem~\ref{TH_SL-vs-COMM} is valid for any pair non-elliptic self-maps $\varphi_1$,~$\varphi_2$. Moreover, if one of these self-maps of hyperbolic, then the requirement that~$h$ is univalent on the hyperbolic discs ${\phdisk(z_n,r)}$, see condition~{\rm(ii${}'$\hspace{.05em})} in Addendum~\ref{ADD_1}, holds automatically for any non-constant holomorphic function ${h:\UD\to\C}$ satisfying Abel's equation ${h\circ\varphi_1\circ\varphi_2=h+c}$ for some constant ${c\in\Complex}$, provided that ${r>0}$ is small enough and $(z_n)$ is a properly chosen orbit of~$\varphi_1\circ \varphi_2$. Indeed, according to Remark~\ref{RM_factorization}, we have $h={f\circ h_{\varphi_1\circ \varphi_2}}$, where $f$~is a holomorphic function in the strip~$S_{\varphi_1\circ \varphi_2}$ satisfying $f(w+1)=f(w)+c$ for all ${w\in S_{\varphi_1\circ \varphi_2}}$. Let $V$ be the $\varphi_1\circ\varphi_2$\,-\,absorbing domain from condition~(HM3) and choose $z_0\in V$ so that $f'\big(h_{\varphi_1\circ \varphi_2}(z_0)\big)\neq0$. Then $f$ is univalent in the euclidian discs of some fixed radius centered at the points ${h_{\varphi_1\circ \varphi_2}(z_0)+n}$, ${n\in\Natural}$. It is then not difficult to see that $h$ is univalent in ${\phdisk(z_n,r)}$, $z_n:=(\varphi_1\circ\varphi_2)^{\circ n}$, for all ${n\in\Natural}$ and some ${r>0}$ small enough. It remains to recall that $(z_n)$ converges to~$\tau$ non-tangentially because $\varphi_1\circ\varphi_2$ is hyperbolic, see e.g. \cite[Proposition~4.3.2]{Abate2}.
\end{remark}

Now we turn to the proof of the results concerning the  commutativity problem for elements of the centralizer~$\Zn(\varphi)$.
The proof of Theorem~\ref{TH_abelian-part0}, given at the end of this section, is based on the following more technical result.
\begin{theorem}\label{TH_abelian-part}
Let $\varphi\in\Hol(\UD)$ be a parabolic self-map of positive hyperbolic step and let ${\psi_1,\psi_2\in\Zn(\varphi)}$. Then $\psi_1\circ\psi_2=\psi_2\circ\psi_1$ if and only if (at least) one of the following three conditions holds:
\begin{equilist}
  \item\label{IT_abelian-betas}
     $\beta_{\varphi,\psi_1}=\beta_{\varphi,\psi_2}$;\smallskip

  \item\label{IT_abelian-rational}
     $c_{\varphi,\psi_1}=p/q\in\mathbb Q$, where ${p\in\mathbb Z}$ and ${q\in\mathbb N}$ are relatively prime, and $\beta_{\varphi,\psi_2}$ satisfies identity~\eqref{EQ_identity-rational}, i.e.
     $
       \beta_{\varphi,\psi_2}(w+1/q)=\beta_{\varphi,\psi_2}(w)+1/q;
     $\smallskip

  \item\label{IT_abelian-viceversa} condition~(b) holds with $\psi_1$ and $\psi_2$ interchanged.
\end{equilist}
\end{theorem}
\begin{proof}
Taking into account that $c_{\varphi,\id_\UD}=0$ (see Remark~\ref{Re:real-coefficient2}) and that ${\beta_{\varphi,\psi_k}(w+1)}={\beta_{\varphi,\psi_k}(w)+1}$, ${k=1,2}$, for any ${w\in S_\varphi}$, it is not difficult to see that if $\id_\UD\in\{\psi_1,\psi_2\}$, then at least one of the conditions \ref{IT_abelian-rational} or~\ref{IT_abelian-viceversa} holds. Hence, Theorem~\ref{TH_abelian-part} is trivially true in this case, and so we may assume that both $\psi_1$ and $\psi_2$ are different from~$\id_\UD$.

Suppose first that one of the conditions \ref{IT_abelian-betas}, \ref{IT_abelian-rational}, or~\ref{IT_abelian-viceversa} holds.  Recalling the definition of the simultaneous linearization function (Definition~\ref{DF_SimLinFunc}) and using Remarks~\ref{RM_identity-is-a-solution2}\,--\,\ref{RM_rational-SLC}, it is not difficult to see that the simultaneous linearization problem~\eqref{EQ_AbelSystem-n} for $\psi_1$ and~$\psi_2$ with $c:={\big(c_{\varphi,\psi_1},\,c_{\varphi,\psi_2}\big)}$ is solved by $h:={\beta_{\varphi,\psi_1}\circ h_{\varphi}}$ or by $h:={\beta_{\varphi,\psi_2}\circ h_{\varphi}}$: if~\ref{IT_abelian-betas} holds, then $h:={\beta_{\varphi,\psi_1}\circ h_{\varphi}}={\beta_{\varphi,\psi_2}\circ h_{\varphi}}$ solves~\eqref{EQ_AbelSystem-n}, and if \ref{IT_abelian-rational} or~\ref{IT_abelian-viceversa} holds, then we can take $h:={\beta_{\varphi,\psi_2}\circ h_{\varphi}}$ or $h:={\beta_{\varphi,\psi_1}\circ h_{\varphi}}$, respectively. In each of these cases, by Proposition~\ref{PR_univ-in-isogonal-dom}, $h$~is~univalent in some domain ${V\subset\UD}$ isogonal at the Denjoy\,--\,Wolff point~$\tau$ of the self-map~$\varphi$. Since $\tau$ is also the Denjoy\,--\,Wolff point of~$\psi_1$ and~$\psi_2$, we may conclude with the help of Theorem~\ref{TH_SL-vs-COMM} that $\psi_1$ commutes with~$\psi_2$.

Suppose now that $\psi_1\circ\psi_2=\psi_2\circ\psi_1$. To simplify the notation, for ${k=1,2}$ we denote by~$g_k$ the semimodel morphism~$g_{\psi_k}$ associated with~$\psi_k$ (as an element of~$\Zn(\varphi)$). Note that $g_1$ and~$g_2$ commute. Indeed,
$$
  g_1\circ g_2\circ h_\varphi=g_1\circ h_\varphi\circ\psi_2=h_\varphi\circ\psi_1\circ\psi_2=
  h_\varphi\circ\psi_2\circ\psi_1=g_2\circ h_\varphi\circ\psi_1=g_2\circ g_1\circ h_\varphi,
$$
and it remains to recall again the Identity Principle.

Consider the following three cases.

\StepC{1}{\textit{$c_{\varphi,\psi_1}$ is real}} By Remark~\ref{RM_identity-is-a-solution2}, $g_1(w)={w+c_{\varphi,\psi_1}}$ for all~$w\in S_\varphi$. Taking into account that $g_2$ commutes with~$g_1$ and recalling Remark~\ref{RM_morphism-properties}, we have
\begin{equation}\label{EQ_g2-periodic}
  g_2(w+c_{\varphi,\psi_1})=g_2(w)+c_{\varphi,\psi_1}
\quad\text{and}\quad
  g_2(w+1)=g_2(w)+1
   \quad\text{for all~$\,w\in S_\varphi$.}
\end{equation}

If $c_{\varphi,\psi_1}$ is irrational, then arguing as in the proof of Propositions~\ref{PR_beta-uniq} and~\ref{PR_beta-properties}, see Case~1 on page~\pageref{PG_irrational}, we may conclude that  $g_2(w)={w+c}$ for all~${w\in S_\varphi}$ and some constant~${c\in\Real}$. In view of Remark~\ref{RM_identity-is-a-solution} it follows that ${c_{\varphi,\psi_2}=c\in\Real}$, and in this case, we have  $${\beta_{\varphi,\psi_1}=\beta_{\varphi,\psi_2}=\id_{S_\varphi}}$$ by  Remark~\ref{RM_real-SLC}, i.e. condition~\ref{IT_abelian-betas} holds.

If $c_{\varphi,\psi_1}$ is rational, namely, if ${c_{\varphi,\psi_1}=p/q}$ for some relatively prime ${p\in\mathbb Z}$ and ${q\in\mathbb N}$, then in the same way as in Remark~\ref{RM_rational-SLC}, from~\eqref{EQ_g2-periodic} we deduce that
$$
\beta_{\varphi,\psi_2}(w+1/q)=\beta_{\varphi,\psi_2}(w)+1/q \quad\text{for all~$\,w\in S_{\varphi}$},
$$
i.e. condition~\ref{IT_abelian-rational} holds.

\StepC{2}{\textit{$c_{\varphi,\psi_2}$ is real}} This case reduces to the previous one by swapping $\psi_1$ with~$\psi_2$.

\StepC{3}{\textit{both $c_{\varphi,\psi_1}$ and $c_{\varphi,\psi_2}$ are non-real}}
Recall that by Remark~\ref{RM_morphism-properties}, $g_1$ and $g_2$ are holomorphic self-mappings of~$\UH$ commuting with ${w\mapsto w+1}$. According to Remark~\ref{RM_identity-is-a-solution},  these self-mappings are not of the form ${w\mapsto w+\theta}$, where ${\theta\in\Real}$. Appealing to Propositions~\ref{Prop:com-with-autoprevio} and~\ref{Prop:com-with-auto}, we therefore conclude that $g_1$ and~$g_2$ are parabolic of zero hyperbolic step. Recalling that these self-mappings commute, by Theorem~\ref{Prop:pseuso-semigroup} we have
\begin{equation}\label{EQ_linear-by-g_1}
h_{g_1}\circ g_2=h_{g_1}+c_0
\end{equation}
for a certain constant $c_0\in\Complex$.

As we saw in the proof of Proposition~\ref{PR_beta-exists}, the function $\beta:=c_{\varphi,\psi_1}h_{g_1}$ satisfies~\eqref{EQ_beta(sys)} for ${\psi:=\psi_1}$ and ${c:=c_{\varphi,\psi_1}}$. Moreover, using~\eqref{EQ_linear-by-g_1}, we obtain
$$
  \beta\circ h_\varphi\circ\psi_2=\beta\circ g_2\circ h_\varphi=\beta\circ h_\varphi+c_{\varphi,\psi_1}c_0,
$$
i.e. $\beta$ satisfies~\eqref{EQ_beta(sys)} also for ${\psi:=\psi_2}$ and ${c:=c_{\varphi,\psi_1}c_0}$.
This easily implies condition~\ref{IT_abelian-betas}. Indeed, by  Proposition~\ref{PR_beta-uniq} and Definition~\ref{DF_SimLinFunc},
$$
  \beta_{\varphi,\psi_1}~=~\beta_{\varphi,\psi_2}~=~\beta~-\!\!\!\!\!\lim_{\substack{w\in S_\varphi,\\|\Im w|\to+\infty}}\!\!\!\!\! \big(\beta(w)-w\big). \eqno\qedhere
$$
\end{proof}

Using the theorem we have just proved, we can now  deduce Corollary~\ref{CR_injective-on-abelian}.

\begin{proof}[\proofof{Corollary~\ref{CR_injective-on-abelian}}]
If $\varphi$ is of zero hyperbolic step, then having ${c_{\varphi,\psi_1}=c_{\varphi,\psi_2}}$ implies that ${\psi_1=\psi_2}$ by Theorem~\ref{Prop:pseuso-semigroup} and Proposition~\ref{Prop:firstpropertiescoefficient}\+\ref{firstpropereties2}.
For this reason, in the rest of the proof we assume that $\varphi$ is of \textit{positive} hyperbolic step.

 We claim that ${\beta:=\beta_{\varphi,\psi_1}=\beta_{\varphi,\psi_2}}$. Indeed, since $\psi_1,\psi_2\in\Zn(\varphi)$ commute with each other, by Theorem~\ref{TH_abelian-part}, these self-mappings satisfy at least one of the conditions~\ref{IT_abelian-betas}, \ref{IT_abelian-rational}, or \ref{IT_abelian-viceversa}. If conditions~\ref{IT_abelian-betas} holds, then we are done. If one of the other two conditions holds, then recalling that by the hypothesis ${c_{\varphi,\psi_1}=c_{\varphi,\psi_2}}$, we have that both $c_{\varphi,\psi_1}$ and $c_{\varphi,\psi_2}$ are real, and as a consequence, by Remark~\ref{RM_real-SLC}, ${\beta_{\varphi,\psi_1}=\beta_{\varphi,\psi_2}=\id_{S_\varphi}}$.

Therefore, by Definition~\ref{DF_SimLinFunc} and Remark~\ref{RM_real-SLC}\+(b), we have
\begin{equation}\label{EQ_psi-psi}
  \beta\circ h_\varphi\circ\psi_1=\beta\circ h_\varphi+c_{\varphi,\psi_1} =\beta\circ h_\varphi+c_{\varphi,\psi_2}=\beta\circ h_\varphi\circ\psi_2.
\end{equation}

By Proposition~\ref{PR_univ-in-isogonal-dom}, there exists a domain ${V\subset\UD}$ isogonal at the Denjoy\,--\,Wolff point~$\tau$ of the self-map~$\varphi$ and such that $\beta\circ h_\varphi$ is univalent in~$V$.

By Theorem~\ref{TH_Behan-and-Ko}, $\tau$ is also the Denjoy\,--\,Wolff point for~$\psi_1$ and~$\psi_2$. In particular, both $\psi_1(r\tau)$ and $\psi_2(r\tau)$ tend to~$\tau$ orthogonally to~$\UC$  as ${(0,1)\ni r\to1}$. Therefore, there exists ${r_0\in(0,1)}$ such that
$$
 \psi_k(E)\subset V,~k=1,2,\quad\text{where $\,E:=\big\{r\tau:r\in[r_0,1)\big\}$}.
$$
Taking into account~\eqref{EQ_psi-psi}, we immediately obtain $\psi_1|_E=\psi_2|_E$, and the desired conclusion that $\psi_1$ and $\psi_2$ coincide (in the whole disc~$\UD$) follows by the Identity Principle.	
\end{proof}

To conclude Section~\ref{S_simmultaneous-lin}, it remains to prove Corollary~\ref{Thm:StrongUniqueness} and Theorem~\ref{TH_abelian-part0}.

\begin{proof}[\proofof{Corollary~\ref{Thm:StrongUniqueness}}] First we will prove the corollary under the additional assumption that ${c_1\neq0}$. After that, we will see that  the case~${c_1=0}$, in fact, cannot occur.

So let $c_1\neq0$.  We may suppose that $\varphi$ is of positive hyperbolic step, because for the case of zero hyperbolic step the statement of the corollary is already known, see Theorem~\ref{TH_SLC(2formulas)}\+\ref{IT_SLC2}. Moreover, replacing $h$ with $c_1^{-1}h$, we may suppose that~${c_1=1}$. Then by Remark~\ref{RM_factorization}, $h$~factorizes via the Koenigs map of~$\varphi$. More precisely,
$$
  h=f\circ h_\varphi\quad\text{with some $f\in\Hol(S_\varphi,\C)$ satisfying ${f(w+1)=f(w)+1}$ for all~${w\in S_\varphi}$.}
$$
By Proposition~\ref{PR_beta-uniq}, it follows that ${c_2=c_{\varphi,\psi}}$ and that the function~$f$, in case ${c_{\varphi,\psi}\in\C\setminus\mathbb Q}$, may differ from~$\beta_{\varphi,\psi}$ only by an additive constant.

The above argument gives the proof under the additional assumption that~${c_1\neq0}$.
To see that this assumption always holds, suppose on the contrary that ${c_1=0}$. Then by the hypothesis, ${c_2\neq0}$. Therefore, by what we have already proved, ${c_1\,=\,c_{\psi,\varphi}\,c_2}$. By the hypothesis, $\varphi$  and~$\psi$ are parabolic self-maps. In particular, both of them are different from~$\id_\UD$. Recalling the definitions of the simultaneous linearization coefficient (Definitions~\ref{simul-linea-coefficient-zero} and~\ref{DF_SimLinCoeff}) and appealing to Remark~\ref{Re:real-coefficient2}  in case~$\psi$ is of positive hyperbolic step, or to Proposition~\ref{Prop:firstpropertiescoefficient}\+\ref{firstpropereties1} in case~$\psi$ is of zero hyperbolic step, we easily conclude that~${c_{\psi,\varphi}\neq0}$. Thus, ${c_1\neq0}$ as required.
\end{proof}

\begin{proof}[\proofof{Theorem~\ref{TH_abelian-part0}}]\label{PROOF_abelian-part0}
Let us first assume that the family ${\Delta\cup\{\varphi\}}$ admits simultaneous linearization.
This means that there is~${h\in\Hol(\UD,\C)}$ and a map $c:\Delta\cup\{\varphi\}\to\Complex$ not vanishing identically such that
\begin{equation}\label{EQ_Delta-linearization}
   h\circ\varphi=h+c(\varphi)\quad\text{and}\quad
   h\circ\psi=h+c(\psi)~\text{~for any~$\psi\in\Delta$.}
\end{equation}

We have to show that any two elements $\psi_1,\psi_2\in\Delta$ commute. Since~$\id_\UD$ commutes with any self-map of~$\UD$, we may suppose that ${\id_\UD\not\in\Delta}$. Then all self-maps~$\psi\in\Delta$ are parabolic~ \cite[Corollary~4.1]{Cowen-comm}.

Using~\eqref{EQ_Delta-linearization} and appealing to Corollary~\ref{Thm:StrongUniqueness}, we see that ${c(\varphi)\neq0}$. Hence, replacing~$h$ by~$c(\varphi)^{-1}h$, we may suppose that  ${c(\varphi)=1}$. Then, again by Corollary~\ref{Thm:StrongUniqueness}, ${c(\psi)=c_{\varphi,\psi}}$ for all ${\psi\in\Delta}$, and moreover, if ${c_{\varphi,\psi}\in\C\setminus\mathbb Q}$, then
$\beta_{\varphi,\psi}\circ h_\varphi$ coincides with~$h$ up to an additive constant. Using the Identity Principle and taking into account that $\beta_{\varphi,\psi}$ satisfies (by the very definition) normalization~\eqref{EQ_beta-normalization}, we see that
$
  {\beta_{\varphi,\psi_1}=\beta_{\varphi,\psi_2}}
$
for any $\psi_1,\psi_2\in\Delta$ with $c_{\varphi,\psi_k}\in\C\setminus\mathbb Q$, ${k=1,2}$. By Theorem~\ref{TH_abelian-part}, it follows that any two such self-mappings $\psi_1$ and~$\psi_2$ commute. It remains to show that $\psi_1,\psi_2\in\Delta$ commute in the case~${c_{\varphi,\psi_1}\in\mathbb Q}$ as well.

If ${c_{\varphi,\psi_2}}$ is also a rational real number, then by Remark~\ref{RM_real-SLC}\+(a), $\beta_{\varphi,\psi_1}=\beta_{\varphi,\psi_2}=\id_{S_\varphi}$, and commutativity follows again by~Theorem~\ref{TH_abelian-part}. So we may suppose that ${c_{\varphi,\psi_2}\in\C\setminus\mathbb Q}$.

Since $h$ solves Abel's equation for~$\varphi$, that is, $h\circ\varphi=h+1$, by Remark~\ref{RM_factorization} we have ${h=\beta\circ h_\varphi}$ for a suitable ${\beta\in\Hol(S_\varphi,\C)}$. Using~\eqref{EQ_Delta-linearization} and Remark~\ref{RM_real-SLC}\+(b), we obtain:
\begin{align*}
  &\beta\circ\big(h_\varphi+1\big)=\beta\circ h_\varphi\circ \varphi=h\circ\varphi=h+1=\beta\circ h_\varphi\,+\,1,\\
  &\beta\circ\big(h_\varphi+c_{\varphi,\psi_1}\big)=\beta\circ h_\varphi\circ \psi_1=h\circ\psi_1=h+c(\psi_1)=\beta\circ h_\varphi\,+\,c(\psi_1)=\beta\circ h_\varphi\,+\,c_{\varphi,\psi_1}.
\end{align*}
Writing $c_{\varphi,\psi_1}=p/q$, where $p\in\mathbb Z$ and $q\in\Natural$ are relatively prime, and arguing in the same way as in Remark~\ref{RM_rational-SLC}, from the above identities we obtain
\begin{equation}\label{EQ_beta_1/q}
\beta(w+1/q)=\beta(w)+1/q\quad\text{for all~$\,w\in S_\varphi$}.
\end{equation}
At the same time, according to~\eqref{EQ_Delta-linearization} the function~$\beta$ satisfies the system~\eqref{EQ_beta(sys)} for~${\psi:=\psi_2}$ and~${c:=c(\psi_2)=c_{\varphi,\psi_2}}$. Hence, by Proposition~\ref{PR_beta-uniq}, $\beta$ and~$\beta_{\varphi,\psi_2}$ coincide up to an additive constant. It follows that the equality~\eqref{EQ_beta_1/q} holds also with~$\beta$ replaced by~$\beta_{\varphi,\psi_2}$, and hence, the pair of the self-maps $\psi_1$ and~$\psi_2$ satisfies condition~\ref{IT_abelian-rational} in Theorem~\ref{TH_abelian-part}. Thus, $\psi_1$ and~$\psi_2$ commute, as desired.

\medskip
To prove the converse implication, suppose now that any two elements of $\Delta$ commute. We have to show that the family ${\Delta\cup\{\varphi\}}$ admits simultaneous linearization. Recall that for any given ${\psi\in\Zn(\varphi)}$, the function ${\beta:=\beta_{\varphi,\psi}}$ solves the system~\eqref{EQ_beta(sys)} for ${c:=c_{\varphi,\psi}}$.
In particular, if ${\beta_{\varphi,\psi}=\id_{S_\varphi}}$ for all~${\psi\in\Delta}$, then our task is trivial, because in this case, ${h_\varphi\circ\varphi=h_\varphi+1}$ and ${h_\varphi\circ\psi=h_\varphi+c_{\varphi,\psi}}$ for any~${\psi\in\Delta}$. Therefore, we may suppose that there exists ${\psi_0\in\Delta}$ such that ${\beta_{\varphi,\psi_0}\neq\id_{S_\varphi}}$.

Let us show that ${h:=\beta_{\varphi,\psi_0}\circ h_\varphi}$ linearizes each element of~${\Delta\cup\{\varphi\}}$. First of all, $h\circ\varphi=h+1$ by~\eqref{EQ_beta(b)}. Further, fix an arbitrary ${\psi\in\Delta}$. Since $\psi$ and~$\psi_0$ commute, we may apply Theorem~\ref{TH_abelian-part} to conclude that (at least) one of the following two statements hold:
\begin{romlist}
\item\label{IT_(i)} $\beta_{\varphi,\psi}=\beta_{\varphi,\psi_0}$, or \smallskip

\item\label{IT_(ii)} $c_{\varphi,\psi}\in\mathbb Q$ and $\beta_{\varphi,\psi_0}(w+c_{\varphi,\psi})=\beta_{\varphi,\psi_0}(w)+c_{\varphi,\psi}$ for all $w\in S_\varphi$.
\end{romlist}
If~\ref{IT_(i)} takes place, then $h\circ\psi=h+c_{\varphi,\psi}$ by~\eqref{EQ_beta(c)}. If~\ref{IT_(ii)} takes place, then ${\beta_{\varphi,\psi}=\id_{S_\varphi}}$ by Remark~\ref{RM_real-SLC} and hence, ${h_\varphi\circ\psi=h_\varphi+c_{\varphi,\psi}}$. This in turn implies that for all~${z\in\UD}$,
$$
 h\big(\psi(z)\big) = \beta_{\varphi,\psi_0}\big(h_\varphi(\psi(z))\big) =
 \beta_{\varphi,\psi_0}\big(h_\varphi(z)+c_{\varphi,\psi}\big) =
 \beta_{\varphi,\psi_0}\big(h_\varphi(z)\big)+c_{\varphi,\psi} = h(z)+c_{\varphi,\psi},
$$
i.e. $h\circ\psi=h+c_{\varphi,\psi}$ holds in case~\ref{IT_(ii)} as well. This completes the proof.
\end{proof}

\section{Analysis of the simultaneous linearization coefficient}\label{S_SLC}
In this section, we establish several identities for the simultaneous linearization coefficient of commuting parabolic self-mappings. Part of these results extend what we already obtained for univalent self-mappings~\cite{CDG-Centralizer} and for parabolic self-maps of zero hyperbolic step~\cite{CDG-zero-h.step}.
\begin{theorem}\label{Formula c_varphi_psi} Let $\varphi\in\Hol(\UD)$ be parabolic with  Denjoy\,--\,Wolff point~$\tau\in\partial\D$ and Koenigs function~$h_{\varphi}$. Let ${\psi\in\Zn(\varphi)}$. Then
	\begin{align}
		\label{Eq:Formula c_varphi_psi1} c_{\varphi,\psi}&=\angle\lim_{z\to\tau} (h_\varphi\circ\psi(z)-h_\varphi(z))\\
		\label{Eq:Formula c_varphi_psi2} &=\angle \lim_{z\to\tau}h_\varphi^\prime(z)(\psi(z)-z)\\
		\label{Eq:Formula c_varphi_psi3} &=\anglim_{z\to\tau}\frac{\psi(z)-z}{\varphi(z)-z}.
	\end{align}
\end{theorem}
\noindent It is worth mentioning here that in Section~\ref{S_realSLC}, we will also obtain an expression for~$c_{\varphi,\psi}$ in terms of the Koenigs functions of~$\varphi$ and~$\psi$; see~\eqref{EQ_another-formular-SLC} in Remark~\ref{RM_another-formular-SLC}.\smallskip

Before proving the theorem above, it is worth stating two corollaries.
The first of them gives yet another formula for the simultaneous linearization coefficient provided that~$\varphi$ has finite non-vanishing second angular derivative at the Denjoy\,--\,Wolff point.
Recall that ${\varphi\in\Hol(\UD)}$ is said to belong to the class~$C_A^2(\tau)$ if there exist ${a_0,a_1,a_2\in\Complex}$ such that
$$
 \anglim_{z\to\tau}\frac{\varphi(z)-\sum\limits_{k=0}^{2}a_k(z-\tau)^k}{(z-\tau)^2}=0.
$$
If this condition holds, then clearly, ${\anglim_{z\to\tau}\varphi(z)=a_0}$ and ${\varphi'(\tau)=a_1}$. The number~$2a_2$ is called the \dff{second angular derivative at~$\tau$} and as usual, we denote it by~$\varphi''_A(\tau)$.
\begin{corollary}\label{CR_second-der}
Let $\varphi\in\Hol(\UD)$ be  parabolic  with  Denjoy\,--\,Wolff point $\tau\in\partial\D$ and Koenigs function~$h_{\varphi}$.
Let $\psi\in\Zn(\varphi)\setminus\{\id_\D\}$. Then:
 \begin{ourlist}
	\item\label{IT_second-der1}
      $\varphi\in C^2_A(\tau)$ if and only if $\psi\in C^2_A(\tau)$. In this case, $\varphi^{\prime\prime}_A(\tau) c_{\varphi,\psi}=\psi^{\prime\prime}_A(\tau).$
	\item\label{IT_second-der2}
        If $\varphi\in C^2_A(\tau)$ with $\varphi^{\prime\prime}_A(\tau)\neq0$,  then $\tau$ is a simple pole of $h_\varphi$ in the angular sense, that is,
		\begin{equation*}
            \mathop{\mathrm{Res}_A}(h_\varphi,\tau):=\angle \lim_{z\to\tau}h_\varphi(z)(z-\tau)
        \end{equation*}
		is finite and non-zero. Moreover, in this situation,
		\begin{equation*}
		     c_{\varphi,\psi}=-\tfrac12\mathop{\mathrm{Res}_A}(h_\varphi,\tau) \psi^{\prime\prime}_A(\tau).
		\end{equation*}
 \end{ourlist}
\end{corollary}

In the second corollary of Theorem~\ref{Formula c_varphi_psi}, we establish two (elementary but) fundamental properties of the simultaneous linearization coefficient.
\begin{corollary}\label{CR_Constante c_varphi_psi}  Let $\varphi\in\Hol(\UD)$ be parabolic.
	\begin{ourlist}
		\item\label{IT_SimCoeff-additive} Assume $\psi_1,\psi_2\in \Zn(\varphi)$. Then
		\begin{equation}\label{EQ_SimCoeff-additive}
		c_{\varphi,\psi_1\circ\psi_2}\,=\,c_{\varphi,\psi_2\circ\psi_1}\,=\,
                                          c_{\varphi,\psi_1}\,+\,c_{\varphi,\psi_2}.
		\end{equation}
		\item\label{IT_SimCoeff-multiplicative} Let $\psi_1,\psi_2,\psi_3\in\Hol(\D)$ be parabolic such that every two of these self-maps commute. Then
	\begin{equation}\label{EQ_SimCoeff-multiplicative}
		c_{\psi_1,\psi_3}\,=\,c_{\psi_1,\psi_2}\,\,c_{\psi_2,\psi_3}.
	\end{equation}
		In particular,
for any $\psi\in \Zn(\varphi)\setminus\{\id_\UD\}$, we have
	\begin{equation}\label{EQ_SimCoeff-1}
		c_{\psi,\varphi}\,\,c_{\varphi,\psi}\,=\,1.
	\end{equation}
	\end{ourlist}	
\end{corollary}

\begin{remark}
Suppose $\varphi,\psi\in\Hol(\UD)$ are commuting parabolic self-maps. Suppose ${c_{\varphi,\psi}={m}/{n}}$ for some ${m,n\in\Natural}$. Then by~\eqref{EQ_SimCoeff-additive}, we have $${c_{\varphi,\psi^{\circ n}}=m=c_{\varphi,\varphi^{\circ m}}}.$$ Applying Corollary~\ref{CR_injective-on-abelian}
to~${\psi_1:=\psi^{\circ n}}$ and ${\psi_2:=\varphi^{\circ m}}$, we then conclude that
${\psi^{\circ n}=\varphi^{\circ m}}$. This generalizes Theorem~\ref{Prop:firstpropertiescoefficient}\+\ref{firstpropereties2}.
\end{remark}
\begin{remark}
Combining the previous remark with Corollary~\ref{CR_second-der}\+\ref{IT_second-der1},
 we obtain the following rigidity result: if ${\varphi,\psi\in C_A^2(\tau)}$ are two commuting parabolic self-maps with Denjoy\,--\,Wolff point at~$\tau$ and if ${m\varphi''_A(\tau)=n\psi''_A(\tau)\neq0}$ for some ${m,n\in\Natural}$, then ${\varphi^{\circ m}=\psi^{\circ n}}$. This improves~\cite[Theorem~2.4\,(6)]{BTV}.
\end{remark}

Now we pass to the proofs of the above statements, which will be followed by one more statement of a bit technical nature used in Section~\ref{S_realSLC}.

\begin{proof}[\proofof{Theorem~\ref{Formula c_varphi_psi}}] If $\varphi$ is of zero hyperbolic step, the result follows from \cite[Theorem~5.1 and its proof]{CDG-zero-h.step}. Therefore, beyond this point, we assume that $\varphi$ is parabolic of positive hyperbolic step.

We recall that $\psi=\id_\D$ if and only if $c_{\varphi,\psi}=0$ (see Remark~\ref{Re:real-coefficient2}). Hence, we assume that $\psi\neq \id_\D$.  We also may assume the base space of the canonical model of~$\varphi$ is $S_\varphi=\H$. (The proof for the other case, that is for ${S_\varphi=-\H}$ is similar.)

We firstly prove \eqref{Eq:Formula c_varphi_psi1} and then we will deduce the others two equalities from the first one.	
If~${c_{\varphi, \psi}\in \R}$, then~\eqref{Eq:Formula c_varphi_psi1} holds trivially, because by Remark~\ref{RM_real-SLC}\+(b), we have ${h_\varphi\circ\psi-h_\varphi=c_{\varphi,\psi}}$.

Suppose now that $c_{\varphi, \psi}\notin \R$. In the proof of Propositions~\ref{PR_beta-uniq} and~\ref{PR_beta-properties} (see Case~2 on page~\pageref{PAGE_Case2}) we found out  that in this case,
\begin{equation}\label{EQ_SLC-via-g_psi}
  c_{\varphi, \psi}=\lim_{\Im w\to+\infty}\big(g_\psi(w)-w\big),
\end{equation}
where $g_\psi$ stands for the semimodel homomorphism associated with~$\psi$.
Recalling that
$$
 \anglim_{z\to\tau}\Im h_\varphi(z)=+\infty \quad\text{and}\quad {h_\varphi\circ\psi}={g_\psi\circ h_\varphi},
$$
see Remarks~\ref{Re:Pommerenke} and~\ref{RM_morphism-properties}, we easily deduce~\eqref{Eq:Formula c_varphi_psi1} from~\eqref{EQ_SLC-via-g_psi}:
	$$
	\anglim_{z\to\tau}\big(h_\varphi(\psi(z))-h_\varphi(z)\big)
  =\anglim_{z\to\tau}\big(g_\psi(h_\varphi(z))-h_\varphi(z)\big)=c_{\varphi,\psi}.
	$$

\smallskip
In turn,~\eqref{Eq:Formula c_varphi_psi1} implies~\eqref{Eq:Formula c_varphi_psi2} because by \cite[Proposition~5.6]{CDG-zero-h.step},
$$
\anglim_{z\to\tau}\frac{h_\varphi(\psi(z))-h_\varphi(z)}{h'_\varphi(z)(\psi(z)-z)}=1.
$$	

Finally, to obtain~\eqref{Eq:Formula c_varphi_psi3}, we make a trivial observation that  ${\varphi\in\Zn(\varphi)}$ and ${c_{\varphi,\varphi}=1}$.  Together with~\eqref{Eq:Formula c_varphi_psi2}, this gives  $$\anglim_{z\to\tau}h'_\varphi(z)(\psi(z)-z)=c_{\varphi,\psi} \quad\text{and}\quad \anglim_{z\to\tau}h'_\varphi(z)(\varphi(z)-z)=1.$$ As a consequence, the angular limit in~\eqref{Eq:Formula c_varphi_psi3} exists and equals
$$
  \anglim_{z\to\tau}\frac{\psi(z)-z}{\varphi(z)-z}=
\frac{\anglim_{z\to\tau}h'_\varphi(z)(\psi(z)-z)}{\anglim_{z\to\tau}h'_\varphi(z)(\varphi(z)-z)}
                                                  =c_{\varphi,\psi}. \eqno\qedhere
$$
\end{proof}

\begin{proof}[\proofof{Corollary~\ref{CR_second-der}}] We follow the argument in the proof of~\cite[Corollary~7.6]{CDG-Centralizer}. Since $\varphi$ is parabolic with Denjoy\,--\,Wolff point~$\tau$, we have that ${\varphi\in C_A^2(\tau)}$ if and only if the limit
$$
a:=\anglim_{z\to\tau} \frac{\varphi(z)-z}{(z-\tau)^2}
$$
exists finitely, and if it does then $\varphi''_A(\tau)=2a$. In view of Theorem~\ref{TH_Behan-and-Ko}, the same holds also for~$\psi$.

\StepP{\ref{IT_second-der1}} The result follows at once if we observe that according to Theorem~\ref{Formula c_varphi_psi}, the left hand side in the identity
	$$
	\frac{\psi(z)-z}{\varphi(z)-z}=\frac{(\psi(z)-z)/(z-\tau)^2}{(\varphi(z)-z)/(z-\tau)^2},\quad z\in \D,
	$$
has finite angular limit at~$\tau$ equal to~$c_{\varphi,\psi}$. This quantity is different from zero, because ${\psi\neq\id_\UD}$ by the hypothesis; see Remark~\ref{Re:real-coefficient2} and \cite[Theorem~4.8]{CDG-zero-h.step}.
	
\StepP{\ref{IT_second-der2}} By the hypothesis, $\varphi\in C^2_A(\tau)$ and $\varphi^{\prime\prime}_A(\tau)\neq0$.
Hence, by what we have proved above, also ${\psi\in C_A^2(\tau)}$ and ${\psi''_A(\tau)\neq0}$. Taking formula~\eqref{Eq:Formula c_varphi_psi2} into account, we then obtain
\begin{equation}\label{EQ_L1}
L_1:=\anglim_{z\to\tau}h'_\varphi(z)(z-\tau)^2=
   \anglim_{z\to\tau}\frac{h'_\varphi(z)\big(\psi(z)-z\big)}{\big(\psi(z)-z\big)/(z-\tau)^2}=
   \frac{2c_{\varphi,\psi}}{\psi''_A(\tau)}.
\end{equation}
It remains to observe that existence of finite limit in the l.h.s. of~\eqref{EQ_L1} implies the existence of
$$
	L_2:=\angle\lim_{z\to\tau}h_\varphi(z)(z-\tau),
$$
and that $L_1=-L_2$; see, e.g., the proof of~\cite[Theorem~6.2]{CDP}.
\end{proof}

\begin{proof}[\proofof{Corollary~\ref{CR_Constante c_varphi_psi}}] Let $\tau\in \partial \D$ be the Denjoy\,--\,Wolff point of~$\varphi$.
	
\StepP{\ref{IT_SimCoeff-additive}} Clearly, we may assume that none of $\psi_1$ and $\psi_2$ is the identity. Then by Theorem~\ref{TH_Behan-and-Ko}, both $\psi_1$ and $\psi_2$ are parabolic and have~$\tau$ as their Denjoy\,--\,Wolff point. In particular, given any sequence ${(z_n)\subset\D}$ converging non-tangentially to~$\tau$, the conformality of $\psi_{2}$ at~$\tau$ guarantees that  $(\psi_2(z_n))$ also converges non-tangentially to~$\tau$. Therefore, by Theorem~\ref{Formula c_varphi_psi},
	$$
	c_{\varphi,\psi_1}=\lim_{n\to\infty} \big(h_\varphi\circ\psi_1(\psi_2(z_n))-h_\varphi(\psi_2(z_n)\big),\quad c_{\varphi,\psi_2}=\lim_{n\to\infty}\big(h_\varphi\circ\psi_2(z_n)-h_\varphi(z_n)\big).
	$$
	Summing these equalities, we obtain
	$$
	c_{\varphi,\psi_1}+c_{\varphi,\psi_2}=\lim_{n\to\infty} \big(h_\varphi\circ\psi_1(\psi_2(z_n))-h_\varphi(z_n)\big)=\lim_{n\to\infty} \big(h_\varphi\circ\psi_1\circ\psi_2(z_n)-h_\varphi(z_n)\big)=c_{\varphi,\psi_1\circ\psi_2}.
	$$
	Switching the roles of $\psi_1$ and $\psi_2$, we also obtain that  $	c_{\varphi,\psi_1}+c_{\varphi,\psi_2}=c_{\varphi,\psi_2\circ\psi_1}.$

\StepP{\ref{IT_SimCoeff-multiplicative}} Note that $\psi_1$, $\psi_2$, and~$\psi_3$ have the same Denjoy\,--\,Wolff point, see Theorem~\ref{TH_Behan-and-Ko}. Therefore, equality~\eqref{EQ_SimCoeff-multiplicative} follows at once from Theorem~\ref{Formula c_varphi_psi} if we observe that
    $$ \frac{\psi_3(z)-z}{\psi_1(z)-z}=\frac{\psi_2(z)-z}{\psi_1(z)-z}\cdot\frac{\psi_3(z)-z}{\psi_2(z)-z}
             \quad\text{for all $z\in\D$.}
	$$

Now if $\psi\in \Zn(\varphi)\setminus\{\id_\D\}$, then by Theorem~\ref{TH_Behan-and-Ko}, $\psi$~is parabolic. Applying~\eqref{EQ_SimCoeff-multiplicative} for ${\psi_1=\psi_3:=\psi}$ and~${\psi_2:=\varphi}$, we immediately obtain~\eqref{EQ_SimCoeff-1}.
\end{proof}

\label{PG_tildes}We conclude the section with one statement of a bit technical kind, but the question it answers is quite natural. Let $\varphi\in\Hol(\D)$ be parabolic and let ${\psi\in\Zn(\varphi)}$. Denote ${V:=V_\varphi(1/3)}$. Then $V$ is a simply connected domain and moreover, ${\varphi(V)\subset V}$ and ${\psi(V)\subset V}$; see Theorem~\ref{Thm:V-sets} and Lemma~\ref{Le:V-sets}.

Consider any conformal map~$f$ of~$\D$ onto ${V:=V_\varphi(1/3)}$.
It is easy to see that the self-maps
\begin{equation}\label{EQ_def-with-tildes}
 \widetilde \varphi:=f^{-1}\circ \varphi\circ f \qquad\text{and}\qquad \widetilde \psi:=f^{-1}\circ \psi\circ f
\end{equation}
commute. Note also that $\widetilde \varphi$ is parabolic by \cite[Proposition~2.11]{CDG-zero-h.step}.

Furthermore, it follows from Theorem~\ref{Thm:V-sets}\+\ref{IT_V-sets(4)} that~$\widetilde \varphi$ is univalent in~$\UD$.
For this reason, in some cases, it turns out to be convenient to replace the original self-mappings $\varphi$ and~$\psi$ with the new commuting pair $\widetilde \varphi$ and~$\widetilde \psi$. Below we show that the simultaneous linearization coefficient does not change under this transformation.
\begin{proposition}\label{Prop:Coefwidetilde}
In the above assumptions and notation, we have
\begin{equation}\label{EQ_Coefwidetilde}
c_{\widetilde\varphi,\widetilde\psi}=c_{\varphi,\psi}.
\end{equation}
\end{proposition}
\begin{proof}
In case of $\varphi$ having zero hyperbolic step, equality~\eqref{EQ_Coefwidetilde} was established in the proof of \cite[Theorem~4.8]{CDG-zero-h.step}. Therefore, we may suppose that~$\varphi$ is of positive hyperbolic step. Then by \cite[Proposition~2.11]{CDG-zero-h.step}, $\widetilde\varphi$ is a parabolic self-map of positive hyperbolic step as well.

As we proved in Section~\ref{S_simmultaneous-lin}, $\varphi$ and~$\psi$ admit simultaneous linearization. More precisely, according to Propositions~\ref{PR_beta-exists} and~\ref{PR_beta-uniq}\+\ref{IT_c-uniq}, there exists ${h\in\Hol(\UD,\C)}$ such that
$$
  h\circ\varphi=h+1 \qquad\text{and}\qquad h\circ\psi=h+c_{\varphi,\psi}.
$$
Using~\eqref{EQ_def-with-tildes}, we immediately obtain
$$
  \widetilde h\circ\widetilde\varphi=\widetilde h+1, \quad \widetilde h\circ\widetilde\psi=\widetilde h+c_{\varphi,\psi}, \quad\text{where~}\,\widetilde h:=h\circ f,
$$
and now equality~\eqref{EQ_Coefwidetilde} follows by Corollary~\ref{Thm:StrongUniqueness}.
\end{proof}

\section{Case of real simultaneous linearization coefficient}\label{S_realSLC}

\subsection{Statements of results}
In the special case $c_{\varphi,\psi}\in\Real$, the two commuting parabolic self-maps~$\varphi,\psi$ turn out to share many properties.
\begin{theorem}\label{TH_real-SLC=common-properties}
Let $\varphi,\psi\in\Hol(\UD)\setminus\{\id_\UD\}$ be two commuting parabolic self-maps.
As before, let ${(S_\varphi,h_\varphi,z\mapsto z+1)}$ and ${(S_\psi,h_\psi,z\mapsto z+1)}$ stand for the canonical holomorphic models for $\varphi$ and~$\psi$, respectively.
Suppose that the simultaneous linearization coefficient of~$\psi$ w.r.t.\,$\varphi$ is real, i.e. ${c_{\varphi,\psi}\in\Real}$. Then the following statements hold:
\begin{ourlist}
\item\label{IT_real-SLC=h.step} $\psi$ is of  positive hyperbolic step if and only if~so~is~$\varphi$;
\smallskip

\item\label{IT_real-SLC=shift} $\psi$ is of finite shift if and only if~so~is~$\varphi$;
\smallskip

\item\label{IT_real-SLC=Koenigs} $h_\psi=c_{\varphi,\psi}^{-1}h_\varphi$;
\smallskip

\item\label{IT_real-SLC=positive} if $c_{\varphi,\psi}>0$, then $S_\psi=S_\varphi$;
\smallskip

\item\label{IT_real-SLC=negative} if $c_{\varphi,\psi}<0$, then ${S_\psi=-S_\varphi}$ and both $\varphi$ and $\psi$ are parabolic automorphisms.
\end{ourlist}
\end{theorem}

The situation when $c_{\varphi,\psi}$ is non-real is more complicated. For example, if $\varphi$ has positive hyperbolic step, then self-mappings $\psi\in{\Zn(\varphi)\setminus\{\id_\UD\}}$ with ${\Im c_{\varphi,\psi}\neq0}$ can be both of zero and positive hyperbolic step. Therefore, at  first glance it seems surprising that the situation is different when we analyze the hyperbolic step of the composition~${\varphi\circ\psi}$.
\begin{theorem}\label{Thm:coefficient-not-real}
	Let $\varphi\in\Hol(\D)$ be parabolic of positive hyperbolic step and let $\psi\in\Zn(\varphi)$.
Suppose that ${\varphi\circ\psi\neq\id_\D}$. Then $\varphi\circ \psi$ is of positive hyperbolic step if and only if ${c_{\varphi,\psi}\in \R}$.
\end{theorem}

From the latter theorem we deduce, as a corollary, Cowen's \cite[Theorem~3.1]{Cowen-comm} stating that any two commuting parabolic self-mappings $\varphi$ and~$\psi$, different from automorphisms, belong to the pseudo-iteration semigroup of their composition~${\varphi\circ\psi}$.
In our terminology, this result means that $h_{\varphi\circ\psi}$ solves the simultaneous linearization problem for these two self-mappings.

\begin{corollary}\label{CR_Cowen}
Let $\varphi$ be a parabolic self-map of positive hyperbolic step and let $\psi\in\Zn(\varphi)$. Suppose that ${\varphi\circ\psi\neq\id_\UD}$. Then for a suitable constant~$c_0\in\C$, we have
\begin{equation}\label{EQ_Cowen}
\beta_{\varphi,\psi}\circ h_\varphi\,=\,c_{\varphi,\varphi\circ\psi}\,h_{\varphi\circ\psi}\,+\,c_0.
\end{equation}
\end{corollary}
It is worth recalling here what has been already mentioned in Section~\ref{S_simmultaneous-lin}. Namely, in view of Proposition~\ref{PR_beta-uniq}, it is completely natural that our solution to the simultaneous linearization problem is essentially the same as the one given in~\cite[Theorem~3.1]{Cowen}. However, Cowen's result does not say anything on how this solution changes if one element of the centralizer of~$\varphi$ is replaced by another. As a consequence, his result does not extend to families consisting of more than two pairwise commuting self-mappings. Our approach, which is quite different from~\cite{Cowen-comm}, allows us to overcome this limitation, see e.g. Theorem~\ref{TH_abelian-part0}.

\begin{remark}\label{RM_for-zero-step-also}
In view of \cite[Theorem~5.3\+(B)]{CDG-zero-h.step}, the above corollary extends to \emph{all} commuting pairs of parabolic self-mappings if we conventionally set ${\beta_{\varphi,\psi}:=\id_\C}$ for all ${\psi\in\Zn(\varphi)}$ in case of $\varphi$ having zero hyperbolic step.
\end{remark}
\begin{remark}
If under the hypothesis of Corollary~\ref{CR_Cowen}, the simultaneous linearization coefficient~$c_{\varphi,\psi}$ is real, then $\beta_{\varphi,\psi}=\id_{S_\varphi}$; see Remark~\ref{RM_real-SLC}. As a result,  identity~\eqref{EQ_Cowen} in case of ${c_{\varphi,\psi}\in\Real}$ means simply that $${h_\varphi\,=\,c_{\varphi,\varphi\circ\psi}\,h_{\varphi\circ\psi}},$$
where the constant $c_0$ vanishes due to the normalization of the Koenigs map chosen in Definition~\ref{DF_canonical}.
\end{remark}
\begin{remark}\label{RM_another-formular-SLC}
Moreover, in the general case, using Remark~\ref{Re:Pommerenke} and Proposition~\ref{PR_beta-properties}\+\ref{IT_beta-limit},  one can easily deduce from~\eqref{EQ_Cowen} that ${c_{\varphi,\varphi\circ\psi}\,h_{\varphi\circ\psi}-h_\varphi}$ has a finite angular limit at the Denjoy\,--\,Wolff point~$\tau$ of the self-maps~$\varphi$ and~$\psi$. In view of Remark~\ref{RM_for-zero-step-also} and
Corollary~\ref{CR_Constante c_varphi_psi}\+\ref{IT_SimCoeff-multiplicative}, it further follows that
$$
 \anglim_{z\to\tau}\big(h_\varphi(z)\,-\,c_{\varphi,\psi}\,h_\psi(z)\big)
$$
exists finitely for any pair of commuting parabolic self-maps\footnote{Note that the expression under the limit sign vanishes identically if both $\varphi$ and~$\psi$ are of zero hyperbolic step, see \cite[Proposition~5.3\+(A)]{CDG-zero-h.step}.}. Taking into account that $h_\psi(z)$ has infinite angular limit as~$z\to+\infty$, we have yet another formula for the simultaneous linearization coefficient. Namely,
\begin{equation}\label{EQ_another-formular-SLC}
 c_{\varphi,\psi}=\anglim_{z\to\tau}\frac{\,h_\varphi(z)\,}{\,h_\psi(z)\,}.
\end{equation}
\end{remark}

\subsection{Proofs}
Now we proceed to the proofs. We will need the following auxiliary statement.
\begin{lemma}\label{LM_absorbs-well}
Let $\varphi\in\Hol(\UD)$ be a non-elliptic self-map and let ${V\subset\UD}$ be some $\varphi$-absorbing domain.
Then for any number ${c>0}$, we have
\begin{equation}\label{EQ_absorbs-well}
  S(V,c):=\bigcup_{n\in\Natural}\big(h_\varphi(V)-nc\big)~=~S_\varphi.
\end{equation}
\end{lemma}
\begin{proof}
Since $h_\varphi(V)\subset S_\varphi$, we have $S(V,c)\subset S_\varphi$. Therefore, to prove that equality~\eqref{EQ_absorbs-well} holds, it suffices to fix an arbitrary ${w_0\in S_\varphi}$ and show that ${w_0+nc\in h_\varphi(V)}$ for a suitable ${n\in\Natural}$. To this end, it is in turn sufficient to find ${n_0\in\Natural}$ such that
$$
  \{w_0+t:t\ge n_0\}\subset h_\varphi(V).
$$
Having noticed this we are done, because the latter statement is a special case (for $K:=\{w_0\}$) of what we saw in Remark~\ref{Re:almostasemigroup}.
\end{proof}

Now we pass to the proof of Theorem~\ref{TH_real-SLC=common-properties}.
First we prove one of the implications for the special case of univalent self-mappings.
\begin{lemma}\label{Le:coefficient-not-real}
Let $\varphi,\psi\in\Hol(\D)$ be commuting univalent self-mappings. Suppose that $\varphi$ is parabolic and that ${\Im c_{\varphi,\psi}\neq~0}$. Then ${\varphi\circ \psi}$ is of zero hyperbolic step.
\end{lemma}
\begin{proof} Since $\varphi$ and~$\psi$ are univalent, according to \cite[Theorem~4.1]{CDG-Centralizer}  there is a univalent function $h_{\varphi,\psi}:\D\to \C$ such that
	\begin{align}
			\label{EQ_uni1} h_{\varphi,\psi}\circ \varphi&=h_{\varphi,\psi}+1  \quad\text{and}\\
			\label{EQ_uni2} h_{\varphi,\psi}\circ \psi&=h_{\varphi,\psi}+c_{\varphi,\psi}.
	\end{align}
Write $\Omega:=h_{\varphi,\psi}(\D)$ and $S:=\cup_{n\in\Natural}(\Omega-n)$. Since $\varphi$ is parabolic, by \cite[Theorem~9.1\+(D)]{CDG-Centralizer}, $S$~contains a half-plane. Thus, there is ${p\in \C}$ such that the parallelogram
$$
   K:=\big\{p+s+tc_{\varphi,\psi}:s,t\in[0,1]\big\}
$$
is contained in~$S$. This means that the sets $\Omega-n$, ${n\in\Natural}$, form an open cover of the compact set~$K$. Since by~\eqref{EQ_uni1}, ${\Omega+1\subset\Omega}$, the sets $\Omega-n$, ${n\in\Natural}$, are nested. It follows that the exists ${n_0\in\Natural}$ such that
\begin{equation*}
E:=\big\{p+n_0+s+tc_{\varphi,\psi}\,:\,s\in[0,+\infty),~t\in[0,1]\big\}~=\,\bigcup_{n=n_0}^{+\infty} \big(K+n\big)
   ~\subset~
     \Omega.
\end{equation*}
Furthermore, by~\eqref{EQ_uni2}, we have ${\Omega+c_{\varphi,\psi}\subset\Omega}$. Hence,
\begin{equation}\label{EQ_uniA}
A:=\big\{p+n_0+s+tc_{\varphi,\psi}:s,t\in[0,+\infty)\big\}~=\,\bigcup_{k=0}^{+\infty}\big(E+kc_{\varphi,\psi}\big)
   ~\subset~
     \Omega.
\end{equation}

Since $\Im c_{\varphi,\psi}\neq0$, we have ${c_{\varphi,\psi}+1\neq0}$. Denote $\mu:={1/(c_{\varphi,\psi}+1)}$ and $h:={\mu h_{\varphi,\psi}}$.
Combining~\eqref{EQ_uni1} with~\eqref{EQ_uni2}, we obtain
$$
  h\circ(\varphi\circ\psi)=h+1.
$$
Thanks to~\eqref{EQ_uniA}, we have
\begin{multline*}
  \bigcup_{n\in\Natural}\big(h(\UD)-n\big)~\supset~\bigcup_{n\in\Natural}\big(\mu A-n\big)~=\\
  =~B:=\,\big\{q-n+s\mu+t\mu c_{\varphi,\psi}:s,t\in[0,+\infty),~n\in\Natural\big\}, \quad q:=\mu(p+n_0).
\end{multline*}
Note that $\mu\,+\,\mu c_{\varphi,\psi}\,=\,1$ and that $\Im\mu\neq0$. As a consequence, $\Im\mu\cdot\Im(\mu c_{\varphi,\psi})<0$. It is an elementary exercise to check that the latter inequality implies that ${B=\C}$. Thanks to \cite[Theorem~9.1\+(A)]{CDG-Centralizer}, it follows that ${\varphi\circ\psi}$ is parabolic of zero hyperbolic step, as desired.
\end{proof}

\begin{proof}[\proofof{Theorem~\ref{TH_real-SLC=common-properties}}] We consider two cases, and start with the simpler one.

\StepC1{$\varphi$ is of zero hyperbolic step} Then assertion~\ref{IT_real-SLC=Koenigs} holds by \cite[Theorem~5.3\+(C)]{CDG-zero-h.step}. In view of Remark~\ref{RM_h.step=h(D)}, it follows that $\psi$ is of zero hyperbolic step as well, i.e. assertion~\ref{IT_real-SLC=h.step} holds. Moreover, every parabolic self-map of zero hyperbolic step is of infinite shift and the base space of its canonical holomorphic model is~$\C$. This proves assertions \ref{IT_real-SLC=shift} and~\ref{IT_real-SLC=positive}. Finally,  assertion~\ref{IT_real-SLC=negative} holds trivially, because $c_{\varphi,\psi}$ cannot be a negative real number when $\varphi$ is of zero hyperbolic step. (This follows easily from Theorem~\ref{Prop:pseuso-semigroup} and Proposition~\ref{Prop:firstpropertiescoefficient}\+\ref{firstpropereties3} if we recall that all parabolic automorphisms are of positive hyperbolic step.)

\StepC2{$\varphi$ is of positive hyperbolic step} Recall that since ${c_{\varphi,\psi}\in\Real}$,  we have
\begin{equation}\label{EQ_again}
h_\varphi\circ\psi=h_\varphi+c_{\varphi,\psi},
\end{equation}
see Remark~\ref{RM_real-SLC}. Therefore, if ${c_{\varphi,\psi}<0}$ then by Proposition~\ref{Prop:firstpropertiescoefficient}, the self-maps $\varphi$ and~$\psi$ are parabolic automorphisms. In this case, the proof is elementary and hence omitted.

Note that ${\psi\neq\id_\UD}$ by the hypothesis. Hence, ${c_{\varphi,\psi}\neq0}$ again by Proposition~\ref{Prop:firstpropertiescoefficient}. So from now on, we suppose that ${c_{\varphi,\psi}>0}$. Let ${V\subset\UD}$ be a $\varphi$-absorbing domain in which $h_\varphi$ is univalent; concerning the existence of such a domain, see Sect.\,\ref{Sec:modelos}. Further, let ${c:=c_{\varphi,\psi}}$. Then applying Lemma~\ref{LM_absorbs-well} and using identity~\eqref{EQ_again}, it is not difficult to see that ${\big(S_\varphi,h_\varphi,z\mapsto z+c_{\varphi,\psi}\big)}$ is a holomorphic model for~$\psi$. In view of Theorem~\ref{Thm:model} and Definition~\ref{DF_canonical}, it follows that $\psi$ is of positive hyperbolic step, that ${h_\psi=c_{\varphi,\psi}^{-1}\,h_\varphi}$, and that  $S_\psi={c_{\varphi,\psi}^{-1}\,S_\varphi=S_\varphi}$.

It remains to see that $\psi$ and~$\varphi$ are either both of finite shift or both of infinite shift.  It is known that a parabolic self-map of positive hyperbolic step has finite shift if and only if the self-map is of class~$C_A^2(\tau)$, where $\tau$ stands for its Denjoy\,--\,Wolff point, see \cite[Theorem~4.1]{CDP}. Thus, the desired conclusion follows from Corollary~\ref{CR_second-der}\+\ref{IT_second-der1}.
\end{proof}

\begin{proof}[\proofof{Theorem~\ref{Thm:coefficient-not-real}}]
One implication is simple. Namely, if $c_{\varphi,\psi}\in\Real$, then by Corollary~\ref{CR_Constante c_varphi_psi}\+\ref{IT_SimCoeff-additive} also $c_{\varphi,\varphi\circ\psi}$ is real, and hence, ${\varphi\circ\psi}$ is of positive hyperbolic step by Theorem~\ref{TH_real-SLC=common-properties}\+\ref{IT_real-SLC=h.step} applied with~$\psi$ replaced by~${\varphi\circ\psi}$.

Assume now that ${c_{\varphi,\psi}\not\in\Real}$.  We wish to show that ${\varphi\circ\psi}$ is of zero hyperbolic step.  Let ${V:=V_{\varphi\circ\psi}(1/3)}$. Clearly, the self-maps $\varphi$, $\psi$, and $\varphi\circ\psi$ are pairwise commuting. Applying Theorem~\ref{Thm:V-sets} and Lemma~\ref{Le:V-sets} to ${\varphi\circ\psi}$ in place of~$\varphi$, it is not difficult to see that  $\varphi$ and~$\psi$ are univalent in~$V$ and map~$V$ into itself. Note also that $V$ is a simply connected domain, see Theorem~\ref{Thm:V-sets}\+\ref{IT_V-sets(2)}. Let $f$ be a conformal map of~$\UD$ onto~$V$. Then
$$
 \widetilde\varphi:=f^{-1}\circ\varphi\circ f
          \quad\text{and}\quad
 \widetilde\psi   :=f^{-1}\circ\psi   \circ f
$$
are commuting univalent self-maps of~$\UD$. By~\cite[Proposition~2.11]{CDG-zero-h.step} applied to ${\varphi\circ\psi}$ in place of~$\varphi$, the self-map
$$
 \widetilde\varphi\circ \widetilde\psi = f^{-1}\circ \varphi\circ\psi \circ f
$$
is parabolic. Clearly, $\widetilde\varphi$ and $\widetilde\psi$ commute with~${\widetilde\varphi\circ \widetilde\psi}$ and both are different from~$\id_\UD$. It follows that these self-maps are parabolic, see Theorem~\ref{TH_Behan-and-Ko}. Therefore, we can use Corollary~\ref{CR_Constante c_varphi_psi} to conclude that
$$
   c_{\widetilde\varphi,\widetilde\psi}\,=\,c_{\widetilde\varphi\circ\widetilde\psi,\widetilde\psi}^{-1}\,-\,1
       \quad\text{and}\quad
   c_{\varphi,\psi}\,=\,c_{\varphi\circ\psi,\psi}^{-1}\,-\,1.
$$
Moreover, by Proposition~\ref{Prop:Coefwidetilde} applied to ${\varphi\circ\psi}$ in place of~$\varphi$, we have
$$
  c_{\widetilde\varphi\circ\widetilde\psi,\widetilde\psi}=c_{\varphi\circ\psi,\psi},
$$
and hence, $c_{\widetilde\varphi,\widetilde\psi}=c_{\varphi,\psi}\not\in\Real$. By Lemma~\ref{Le:coefficient-not-real}, ${\widetilde\varphi\circ \widetilde\psi}$ is of zero hyperbolic step, and it remains to apply again \cite[Proposition~2.11]{CDG-zero-h.step} to see that ${\varphi\circ\psi}$ is also of zero hyperbolic step, as desired.
\end{proof}

\begin{proof}[\proofof{Corollary~\ref{CR_Cowen}}]
If ${c_{\varphi,\psi}\in\Real}$, then~\eqref{EQ_Cowen} holds by Remark~\ref{RM_real-SLC}\+(a) and Theorem~\ref{TH_real-SLC=common-properties}\+\ref{IT_real-SLC=Koenigs}.

If ${c_{\varphi,\psi}\not\in\Real}$, then by Theorem~\ref{Thm:coefficient-not-real} the parabolic self-map ${\varphi\circ\psi}$ is of zero hyperbolic step. Since ${\varphi,\psi\in\Zn(\varphi\circ\psi)}$, by Theorem~\ref{Prop:pseuso-semigroup} applied with $\varphi$ replaced by~${\varphi\circ\psi}$, we have
\begin{equation}\label{EQ_Koenigs-of-comp-linearize-both}
  h_{\varphi\circ\psi}\circ\psi=h_{\varphi\circ\psi}+c_{\varphi\circ\psi,\psi}
    \quad\text{and}\quad
  h_{\varphi\circ\psi}\circ\varphi=h_{\varphi\circ\psi}+c_{\varphi\circ\psi,\varphi}.
\end{equation}
Recall that by the hypothesis, $\varphi$ is of positive hyperbolic step. Therefore, according to Corollary~\ref{Thm:StrongUniqueness}, equalities~\eqref{EQ_Koenigs-of-comp-linearize-both} imply that
$h_{\varphi\circ\psi}$ differs from ${c_{\varphi\circ\psi,\varphi}\,\beta_{\varphi,\psi}\circ h_\varphi}$ by an additive constant.
In view of equality~\eqref{EQ_SimCoeff-1} applied with~$\psi$ replaced by~${\varphi\circ\psi}$, this implies the desired identity~\eqref{EQ_Cowen}.
\end{proof}

\hypertarget{AppendixA}{}
\section{Appendix A: Self-mappings commuting with parabolic automorphisms}
Self-maps of~$\UD$ that commute with hyperbolic automorphisms are easily described: they are again hyperbolic automorphisms by Heins' Lemma; see Theorem~\ref{TH_Behan-and-Ko}. The situation is much more intricate for parabolic automorphisms. In this context, it is more convenient to formulate the statements in the upper half-plane. Self-maps of~$\UH$ commuting with parabolic automorphisms are described
in \cite[Proposition 2.6.11, page 152]{Abate2}:

\begin{proposition}\label{Prop:com-with-autoprevio}
	Let $g:\H\to \H$ be a holomorphic self-map such that ${g(w+1)=g(w)+1}$ for all ${w\in \H}$. Then one and only one of the following cases occurs:
	\begin{ourlist}
		\item\label{ComAutocaseI} either ${g(w)=w+\theta}$ for all~{$w\in \H$} and some~${\theta\in \R}$,
		\item\label{ComAutocaseII} or there exists ${F\in \Hol(\D,\H)}$ such that
		\begin{equation}\label{Eq:com-with-auto1}
			g(w)=w+F(e^{2\pi iw}), \quad w\in \H.
		\end{equation}
	\end{ourlist}
\end{proposition}

\begin{remark} It is worth pointing out that the statement \cite[Proposition~2.6.11, page~152]{Abate2} is slightly different. Namely, in \cite{Abate2} it only appears that ${F\in \Hol(\D^*,\H)}$. Nevertheless, if the singularity of $F$ at zero is not removable, then by the Casorati\,--\,Weierstra\ss{} theorem, it is clear that it must be a pole. With the help of the Open Mapping Theorem, it is easy to see  that  in such a case $F(\UD^*)\supset{\{z\in\C:|z|>R\}}$ for some~${R>0}$. This contradiction shows that $F$ must have a removable singularity.
\end{remark}

Clearly, every self-map that appears both in case \ref{ComAutocaseI} and in case \ref{ComAutocaseII} of above proposition commutes with $w\mapsto w+1$ and, in the case \ref{ComAutocaseII}, $g$ is not an automorphism. This suggests that the self-maps appearing in both cases have different properties.
In the next result, we describe  and gather the properties of those self-maps~$g$ that appear in Proposition~\ref{Prop:com-with-autoprevio}\ref{ComAutocaseII}.
These self-maps play a significant role in our results.

\begin{proposition}\label{Prop:com-with-auto} Let $F\in\Hol(\D,\H)$ and let $g\in\Hol(\H)$ be given by
$$
  g(w)=w+F(e^{2\pi iw}),\quad {w\in\H}.
$$
  Then:
	\begin{ourlist}
		\item $\lim_{\Im w\to +\infty}\left(g(w)-w\right)=F(0)$. In particular, $\angle\lim_{w\to\infty}\left(g(w)-w\right)=F(0)$.\smallskip

		\item $g$ is parabolic with Denjoy\,--\,Wolff point $\infty$. In particular, $\Im g(w)\geq \Im w$, for all $w\in \H$.\smallskip

		\item For all ${w\in\H}$, the sequence $(g^{\circ n}(w))$ tends { to $\infty$ non-tangentially}. Thus, $g$ is of zero hyperbolic step and the sequence $\big(\Im g^{\circ n}(w)\big)$ tends to $+\infty$.
		\smallskip

		\item There exists a (unique) $f\in\Hol(\D)$ with ${f(0)=0}$ and ${0<|f'(0)|<1}$ such that
		$$
		e^{2\pi i g(w)}=f\left(e^{2\pi i w}\right)\quad \text{for all $~w\in\H$}.
		$$
		Moreover, $f^\prime(0)=e^{2\pi i F(0)}$.\smallskip

		\item There exists $A=A(g)>0$ such that $g$ is univalent in ${\Pi_{A}:=\{w\in\H:\ \Im w>A\}}$.\smallskip

		\item Let ${(\C,h_g,w\mapsto w+1)}$ be the canonical model of $g$ (see assertion~$(3)$) and let ${(\C,h_f,z\mapsto f^\prime(0)z)}$ be the canonical model of $f$ (see assertion~$(4)$). Then:\smallskip

		\begin{ourlist}
			\item[\rm(a)] $h_g$ is univalent in $\Pi_A$, where $\Pi_A$ is the same as in assertion~$(5)$.
			\item[\rm(b)] There exists $b\in\C^*$ such that, for all $w\in\H$,
			\begin{equation}\label{EQ_Koenigs-functions-relation}
			  e^{2\pi i F(0) h_{g}(w)}=bh_{f}(e^{2\pi iw}).
			\end{equation}
			\item[\rm(c)] $\,F(0) h_{g}(w+1)=F(0) h_{g}(w)+1\,$ for all $w\in \H$.
			\item[\rm(d)] There exists the limit $$\lim_{\Im w\to +\infty}(F(0)h_g(w)-w)\in\C.$$
			\item[\rm(e)] $\,\lim_{n\to+\infty} (g^{\circ n})'(w)= F(0) h'_{g}(w)\,$ for all $w\in\H$.
		\end{ourlist}	
	\end{ourlist}	 	
\end{proposition}
\begin{proof}
	\noindent To prove~{(1)} simply note $\lim_{\Im w\to+\infty} e^{2\pi i w}=0$.
	
\StepP{(2)} Since $g$ commutes with the  parabolic automorphism $w\mapsto w+1$ with Denjoy\,--\,Wolff point $\infty$ and  since $g\neq\id_\H$, the self-map $g$ is itself parabolic with the same Denjoy\,--\,Wolff point, see Theorem~\ref{TH_Behan-and-Ko}. Moreover, according to Denjoy\,--\,Wolff Theorem, see e.g. \cite[Theorem~1.8.4]{BCD-Book}, it holds that
	\begin{equation}\label{Eq:com-with-auto3}
		\Im g(w)\geq \Im w \quad \text{ for all } w\in \H.
	\end{equation}
	
\StepP{(3)} Fix an arbitrary ${w_0\in\UH}$. The image of ${\Pi_{B}:=\{w\in\H:\Im w> B\}}$, $B:=\tfrac12\Im w_0$, w.r.t. the map ${w\mapsto F(e^{2\pi i w})}$ coincides with $F(r\UD)$, where ${r:=e^{-2\pi B}<1}$. Since ${F(\UD)\subset\UH}$, it follows that there exists~${\varepsilon>0}$ such that for any ${w\in\Pi_B}$,
\begin{equation}\label{EQ_g0}
 g_0(w):=F(e^{2\pi i w})\in S_0:=\{w\in\C:\varepsilon<\mathrm{Arg}(w)<\pi-\varepsilon \}\subset \UH.
\end{equation}
Denote $S_{w_0}:=\{w\in\C:\varepsilon<\mathrm{Arg}(w-\Re w_0)<\pi-\varepsilon\}$. Recalling that ${g(w)=w+g_0(w)}$ for all ${w\in\UH}$, from~\eqref{EQ_g0} we deduce that
$$
g\big(\Pi_B\cap S_{w_0}\big) ~\subset~ \big(\Pi_B\cap S_{w_0}\big)\,+\,S_0 ~\subset~\Pi_B\cap S_{w_0}.
$$
In particular, $\big(g^{\circ n}(w_0)\big)_{n\in\Natural}$ is contained in the Stolz angle~$S_{w_0}$. In combination with assertion~(2) and by arbitrariness of~$w_0$, this means that every orbit of~$g$ tends to~$\infty$ non-tangentially.
Being $g$ parabolic, this further implies that~$g$ is of zero hyperbolic step, see e.g. \cite[Corollary~4.6.10]{Abate2}.

\StepP{(4)} By \cite[Lemma~9.3\+(I)]{CDG-Centralizer}, applied to the map $w\mapsto e^{2\pi i g(w)}$, there exists a holomorphic function $f:\D^{*}\to \D $ such that
	$$
	  e^{2\pi i g(w)}=f(e^{2\pi iw}), \quad w\in \H.
	$$
	Since $g(\UH)\subset\UH$, the function~$f$ is bounded in~$\UD^*$. It follows that $f$ can be extended to a holomorphic function in~$\D$.
	In fact, with the help of assertion~(1) we find
  $$
	f(0)=\lim_{\Im w\to+\infty}f(e^{2\pi iw})=\lim_{\Im w\to +\infty}e^{2\pi i g(w)}=0.
  $$

It remains to calculate~$f'(0)$. Using again assertion~(1), we have
	$$
	f^\prime(0)=\lim_{\Im w\to +\infty}\frac{f(e^{2\pi iw})}{e^{2\pi iw}}=
            \exp\Big(2\pi i \lim_{\Im w\to +\infty} (g(w)-w)\Big)=\exp\big(2\pi i F(0)\big)\in \D^*,
	$$
as desired.
	
\StepP{(5)} As we proved above, $f(0)=0$ and $0<|f'(0)|<1$. It follows that there exists ${\varepsilon>0}$ such that $f$ is univalent in the disc $\varepsilon\UD$. Take $A>0$ such that $e^{2\pi i w}\in \varepsilon\UD$ for all ${w\in \Pi_{A}:=\{w\in \H: \Im w>A\}}$. For any $w_{1}, w_{2}\in \Pi_{A}$ such that $g(w_{1})=g(w_{2})$, we have
	$$
	f(e^{2\pi iw_{1}})=e^{2\pi i g(w_1)}=e^{2\pi i g(w_2)}=f(e^{2\pi iw_{2}}).
	$$ Thus $ e^{2\pi iw_{1}}=e^{2\pi iw_{2}}$ and there is $k\in \Z$ such that $w_{2}=w_{2}+k$. Since $g$ commutes with $w\mapsto w+1$ we have
	$$
	g(w_{1})=g(w_{2})=g(w_{1}+k)=g(w_{1})+k.
	$$ Therefore, $k=0$ and thus $w_{1}=w_{2}$, getting the univalence of $g$ in $\Pi_{A}$.
	
\StepP{\rm(6\+a)} By \eqref{Eq:com-with-auto3}, $\Pi_A$ is $g$-invariant and by assertion~(5), the function~$g$ is  univalent in~$\Pi_A$. Appealing to \cite[Lemma~3.5.4\+(iv)]{Abate2}, we deduce that $h_g$~is also univalent in~$\Pi_A$.

\StepP{{\rm(6\+b)} and {\rm(6\+c)}} For $w\in\UH$ denote $h_0(w):=h_g(w+1)$. Then for any~${w\in\UH}$,
$$
  h_0\big(g(w)\big)=h_g\big(g(w)+1\big)=h_g\big(g(w+1)\big)=h_g\big(w+1\big)+1=h_0\big(w\big)+1,
$$
i.e. $h_0$ solves Abel's equation for~$g$. Clearly, ${h_0(\UH)=h_g(\UH)}$. Furthermore, observe that $h_g$ and hence also $h_0$ are univalent in~$\Pi_A$ by assertion~(6\+a) and that $\Pi_A$ is $g$-absorbing by assertion~(3). Using these facts, it is not difficult to see that ${\big(\C,h_0,w\mapsto w+1\big)}$ is a holomorphic model for~$g$. Thanks to the uniqueness of the holomorphic model (see Theorem~\ref{Thm:uniqness}) it follows that ${h_0=h_g+c}$ for a suitable constant~$c\in\Complex$, which we are going to identify.

Note that ${c\neq0}$ because $h_g$ is univalent in~$\Pi_A$.
Consider the function
$$
  p(\zeta):=h_g(\zeta+iA)/c, \quad \zeta\in\UH.
$$
It is univalent in~$\UH$ and satisfies ${p(\zeta+1)=h_0(\zeta+iA)/c=h_g(\zeta+iA)/c+1=p(\zeta)+1}$ for all ${\zeta\in\UH}$. As a consequence, by \cite[Lemma~9.3\+(II)]{CDG-Centralizer}, there exists a univalent function ${\sigma:\UD\to\Complex}$ with ${\sigma(0)=0}$ and such that
\begin{equation}\label{EQ_p-sigma}
  e^{2\pi i p(\zeta)}=\sigma(e^{2\pi i \zeta}),\quad\zeta\in\UH.
\end{equation}
It follows that
$
  p'(\zeta)=e^{2\pi i \zeta}{\sigma'(e^{2\pi i \zeta})}/{\sigma(e^{2\pi i \zeta})}
$ for all~${\zeta\in\UH}$.
Hence,
\begin{equation}\label{EQ_Koenig-der-lim}
  \lim_{\Im w\to+\infty} h'_g(w)~=~ \lim_{\Im \zeta\to+\infty} cp'(\zeta)~=~\lim_{z\to0}\,\,c\,\frac{z\sigma'(z)}{\sigma(z)}~=~c.
\end{equation}

Combining~\eqref{EQ_Koenig-der-lim} with Abel's equation~$ h_g(g(w))-h_g(w)=1$ and recalling assertion~(1), it is not difficult to see that~${c=1/F(0)}$.
This proves assertion~{\rm(6\+c)}.

To prove~{\rm(6\+b)}, consider the function $\sigma_0(z):={\sigma(e^{2\pi A}z)}$ defined in~${D:= e^{-2\pi A}\UD}$.  With the help of~\eqref{EQ_p-sigma}, for ${w\in\Pi_A}$ we have
\begin{equation}\label{EQ_with-sigma0}
  \sigma_0(e^{2\pi i w}) = \sigma(e^{2\pi i (w-iA)}) = e^{2\pi i p(w-iA)}=e^{2\pi i F(0) h_g(w)}.
\end{equation}
Observe that $f(D)\subset D$ by the Schwarz Lemma. Hence, using equality~\eqref{EQ_with-sigma0} and recalling that ${e^{2\pi i F(0)}=f'(0)}$, for any $w\in\Pi_A$, we can write
$$
  \sigma_0\big(f(e^{2\pi i w})\big)=\sigma_0\big(e^{2\pi i g(w)}\big)=e^{2\pi i F(0) (h_g(w)+1)}
                                   =f'(0)e^{2\pi i F(0) h_g(w)}
                                   =f'(0)\sigma_0\big(e^{2\pi i w}\big).
$$
It follows that $\sigma_0\circ f|_D=f'(0)\sigma_0$. Note that $\sigma_0$ is univalent in~$D$. Hence,  we can express $f|_D$ as ${\sigma_0^{-1}\circ\big(f'(0)\sigma_0\big)}$. Moreover, the function ${\psi:=h_f\circ\sigma_0^{-1}}$ is well-defined and holomorphic in ${\sigma_0(D)\ni\sigma_0(0)=0}$. Therefore, for any ${z\in\sigma_0(D)}$, we have ${f(\sigma_0^{-1}(z))}={\sigma_0^{-1}(f'(0)z)}$ and
$$
  \psi\big(f'(0)z\big)=h_f\big(\sigma_0^{-1}(f'(0)z)\big)=h_f\big(f(\sigma_0^{-1}(z))\big)
                   =f'(0)h_f\big(\sigma_0^{-1}(z)\big)=f'(0)\psi\big(z\big).
$$
Comparing the Taylor expansion of $\psi\big(f'(0)z\big)$ in powers of~$z$ with that of $f'(0)\psi\big(z\big)$, we conclude that for the $n$-th Taylor coefficient of~$\psi$ at~${z=0}$ to be different from zero it is necessary that ${\big(f'(0)\big)^n=f'(0)}$. Since ${f'(0)\in\UD^*}$, it follows that ${\psi(z)=az}$ for all ${z\in\sigma_0(D)}$ and some constant~${a\in\C}$. Substituting $z:={\sigma_0(e^{2\pi i w})}$, where ${w\in\Pi_A}$, and recalling~\eqref{EQ_with-sigma0}, we obtain
$$
  h_f(e^{2\pi i w})\,=\,a\sigma_0(e^{2\pi i w})\,=\,a\,e^{2\pi i F(0) h_g(w)}.
$$
Since the function $h_f$ is not constant, we have ${a\neq0}$. Thus, with the help of the Identity Principle, we see that~\eqref{EQ_Koenigs-functions-relation} holds for ${b:=1/a}$ and for any~${w\in\UH}$, as desired.

\StepP{\rm(6\+d)} Denote $Q(w):=F(0)h_g(w)-w$. By assertion~(6\+c), ${Q(w+1)=Q(w)}$ for any ${w\in\UH}$. Therefore, by \cite[Lemma~9.3\+(I)]{CDG-Centralizer}, there exists a holomorphic function ${q:\UD^*\to\C}$ such that ${Q(w)=q(e^{2\pi i w})}$ for all ${w\in\UH}$. At the same time,
by~\eqref{EQ_with-sigma0} for any $w\in\Pi_A$, we have
$$
  e^{2\pi i Q(w)}=\sigma_0(e^{2\pi i w})/e^{2\pi i w}.
$$
It follows that $e^{2\pi i q(z)}=\sigma_0(z)/z$ for all~$z$ in some punctured neighbourhood of~$0$. If $q$ has a pole at~${z=0}$, then $e^{2\pi i q(z)}$ must have an essential singularity at~$z=0$. The same conclusion can be derived with the help of the Casorati\,--\,Weierstra\ss{} theorem when $q$ has an essential singularity at~${z=0}$. However, $\sigma_0(z)/z$ extends holomorphically to~${z=0}$ because ${\sigma_0(0)=0}$. Therefore, $q$ must have a removable singularity at~${z=0}$, which implies assertion~{\rm(6\+d)}.

\StepP{\rm(6\+e)} As we proved above,
\begin{equation}
   \lim_{\Im \zeta\to+\infty} h_g'(\zeta)~=~1/F(0).
\end{equation}	
Recall also that by assertion~(3),  $\lim_{n\to+\infty}\Im g^{\circ n}(w)=+\infty$ for any $w\in\UH$. Therefore, differentiating Abel's equation $h_g(g^{\circ n}(w))=h_g(w)+n$ w.r.t.~$w$ and passing to the limit as ${n\to+\infty}$, we obtain
$$
\lim_{n\to+\infty}(g^{\circ n})'(w)=\lim_{n\to+\infty}\frac{h_g'(w)}{\,h_g'(g^{\circ n}(w))\,}=F(0)h_g'(w),
$$
as desired.
\end{proof}

\begin{remark} Concerning assertion~(1) of Proposition~\ref{Prop:com-with-auto} and bearing in mind \cite[Propositions~2.1 and~6.4\+(2)]{CDP}, for each ${w\in\H}$, we have
	$$
	\lim_{n\to+\infty}\dfrac{g^{\circ n}(w)}{n}=F(0)
	$$
	and thus,
	$$
	\lim_{n\to+\infty}\dfrac{\Re g^{\circ n}(w)}{\Im g^{\circ n}(w)}=\dfrac{\Re F(0)}{\Im F(0)}.
	$$
	In other words, all orbits $\big(g^{\circ n}(w)\big)$ tend to~$\infty$ with the following common and definite slope $$\pi/2-\arctan\Big(\,\dfrac{\Re F(0)}{\Im F(0)}\+\Big)\in(0,\pi).$$
\end{remark}

\hypertarget{AppendixB}{}
\section{Appendix B: Two examples}
Here we construct two examples of parabolic self-maps of positive hyperbolic step illustrating two situations  that can be regarded as extreme points in the spectrum of complexity of the centralizers.
We start with an example exhibiting simplest possible centralizer structure.
Clearly, $$\{\varphi^{\circ n}:n=0,1,2,\ldots\}~\subset~\Zn(\varphi)$$ for any holomorphic self-map~$\varphi$.
It is not difficult to construct a hyperbolic self-map for which equality holds in the above inclusion, see~\cite[Example~8.3]{CDG-Centralizer}. The centralizer in such a case is isomorphic to~$\big[\mathbb N_0,\,+\,\big]$.
The example below shows that the same situation may occur also for parabolic self-maps of positive hyperbolic step.

\begin{example}
For simplicity, we first construct an example of a parabolic self-map~${\varphi\in\Hol(\UD)}$ of positive hyperbolic step which is embeddable in a one-parameter semigroup $(\phi_t)$ and such that $\Zn(\varphi)=\{\phi_t:t\ge0\}$; see e.g. \cite{CDG-Centralizer} for the definitions of a one-parameter semigroup and embeddability.
At the very end, we will show how to modify the construction in order to obtain a parabolic~$\varphi$ of positive hyperbolic step satisfying $\Zn(\varphi)={\{\varphi^{\circ n}:n=0,1,2,\ldots\}}$.

Let $\mathcal D_0:=\{x+iy:x>0,~0<y<\sqrt{x}\}$. Let $h$ be a conformal mapping of~$\UD$ onto~$\mathcal D_0$. It is not difficult to see that $\varphi:=h^{-1}\circ(h+1)$ is a (univalent) non-elliptic holomorphic self-map of~$\UD$ and that the Koenigs map of~$\varphi$ is given by ${h_\varphi(z)=h(z)-\Re h(0)}$. Similarly, it is not difficult to see that for each ${t\ge0}$, the formula $\phi_t:={h_\varphi^{-1}\circ (h_\varphi+t)}={h^{-1}\circ (h+t)}$ defines an element of~$\Zn(\varphi)$. Note also that according to the definition of a holomorphic semimodel, see Definition~\ref{DF_holomorphic-model}, the base space~$S_\varphi$ of the canonical model for~$\varphi$
coincides with
$$
  \bigcup_{n\in\Natural} h_\varphi(\UD)-n ~=~ \UH.
$$
It follows, see Theorem~\ref{Thm:model}, that $\varphi$ is parabolic of positive hyperbolic step.

We are going to check that $\Zn(\varphi)$ contains no elements other than $\phi_t$'s, ${t\ge0}$.
To this end, fix some $\psi\in\Zn(\varphi)$.  Consider the following three cases.

\StepC1{$c_{\varphi,\psi}=t$ for some~$t\ge0$} Then ${\beta_{\varphi,\psi}=\beta_{\varphi,\phi_t}=\id_\UH}$, see Remark~\ref{RM_real-SLC}\+(a). Hence, by Theorem~\ref{TH_abelian-part}, $\psi$ and $\phi_t$ commute. Moreover, since ${h_\varphi\circ\phi_t}={h_\varphi+t}$, by Corollary~\ref{Thm:StrongUniqueness} we have ${c_{\varphi,\phi_t}=t=c_{\varphi,\psi}}$. Thus, applying Corollary~\ref{CR_injective-on-abelian}, we can conclude that ${\psi=\phi_t}$.

\StepC2{$c_{\varphi,\psi}<0$} This situation, in fact, cannot occur. Indeed,  by Remark~\ref{RM_real-SLC}\+(b), $${h_\varphi\circ\psi}={h_\varphi+c_{\varphi,\psi}}.$$ In particular, it follows that ${\mathcal D_0+c_{\varphi,\psi}\subset\mathcal D_0}$, which can hold only when~$c_{\varphi,\psi}$ is non-negative.

\StepC3{$c_{\varphi,\psi}\not\in\Real$} Again we are going to show that this case cannot occur. In a sense, the main idea is similar to the one employed in the previous case, but a more complicated argument is required because now we have
${h_\varphi\circ\psi}={g_\psi\circ h_\varphi}$, where the semimodel morphism $g_\psi$ associated with~$\psi$, see Definition~\ref{DF_g-psi} and Remark~\ref{RM_morphism-properties}, is not necessarily of the form $\,g_\psi(w)={w+\const}~$
any more. However, by Proposition~\ref{Prop:com-with-autoprevio} and Remark~\ref{RM_identity-is-a-solution}, there is ${F\in\Hol(\UD,\UH)}$ such that
$$
  g_\psi(w)=w+F(e^{2\pi i w}),\quad\text{for all~$w\in\UH$}.
$$

The plan is to obtain a contradiction by constructing a sequence ${(w_n)\subset\C}$ with ${\Im w_n\to+\infty}$ and such that ${w_n\in\mathcal D_0}$ but ${g_{\psi}(w_n)\not\in\mathcal D_0}$ for ${n\in\Natural}$ large enough. A very helpful observation is that
\begin{equation}\label{EQ_lim=F(0)}
F(e^{2\pi i w_n})~\longrightarrow~ F(0)=c_{\varphi,\psi},\quad\text{as~$~n\to+\infty$},
\end{equation}
where the equality $F(0)=c_{\varphi,\psi}$ holds by combining~\eqref{EQ_SLC-via-g_psi} with Proposition~\ref{Prop:com-with-auto}\+(1).

Now for $n\in\Natural$, we set
$$
   w_n:=n-\Re c_{\varphi,\psi}+i\Big(\sqrt{n-\Re c_{\varphi,\psi}}\,-\,\tfrac1n\Big).
$$
Taking into account $\Im c_{\varphi,\psi}=\Im F(0)>0$, it is not difficult to see that for all ${n\in\Natural}$ large enough,
$$
  w_n\in\mathcal D_0 \quad\text{and}\quad w_n+c_{\varphi,\psi}\not\in\mathcal D_0.
$$
Moreover, taking into account that for each ${n\in\Natural}$, the domain $\mathcal D_0$ lies below the tangent to its boundary at the point~$n+i\sqrt{n}$, by a simple calculation we get that the euclidian distance of ${w_n+c_{\varphi,\psi}}$ to~$\mathcal D_0$ satisfies
\begin{eqnarray*}
  \dist\big(w_n+c_{\varphi,\psi},\,\mathcal D_0\big) &\ge&
      \Big(\Im (w_n+c_{\varphi,\psi})\,-\,\sqrt{\Re (w_n+c_{\varphi,\psi})}\,\Big)\,\cos\big(\arctan\tfrac1{2\sqrt{n}}\big)\\
      &=& \Big(\sqrt{n-\Re c_{\varphi,\psi}}-\tfrac1n+\Im c_{\varphi,\psi}-\sqrt{n}\Big)\,\cos\big(\arctan\tfrac1{2\sqrt{n}}\big)~\to~\Im c_{\varphi,\psi}>0
\end{eqnarray*}
as ${n\to+\infty}$. We are now done because, in view of~\eqref{EQ_lim=F(0)}, ${\dist\big(g_{\psi}(w_n),\,w_n+c_{\varphi,\psi}\big)\to0}$ as ${n\to+\infty}$. Hence,
it follows that
$
 g_\psi(w_n)
$
does not belong to~$\mathcal D_0$ either when ${n\in\Natural}$ is large enough.\medskip

Essentially the same argument applies to the domain
$$
 \mathcal D~:=~\mathcal D_0~\big\backslash\,\,\Big(\bigcup_{k\in\Natural}\,\,[n,\, n+2^{-n}i]\Big)
$$
on place of~$\mathcal D_0$, yielding an example of a parabolic self-map~$\varphi$ with positive hyperbolic step such that
$\Zn(\varphi)={\{\varphi^{\circ n}:n=0,1,2,\ldots\}}$. The main difference is that the argument of Case~1 applies now only to $c_{\varphi,\psi}\in{\Natural_0}$, while the argument of Case~2 covers all $c_{\varphi,\psi}\in{\Real\setminus\Natural_0}$.
Indeed, if $p\in (0,+\infty)\setminus \N$, we can take ${n\in \N_0}$ such that ${n<p<n+1}$.
 Then $I:={(1,2)+\frac{1}{2^{n+3}}i}\subset \Omega$, but $I+p\not\subset \Omega$.
The proof of the fact that $\Zn(\varphi)$ contains no self-maps~$\psi$ with ${c_{\varphi,\psi}\not\in\Real}$ is literally the same. \qed
\end{example}

In the above example the set
\begin{equation}\label{EQ_Beta}
\mathcal B(\varphi):=\big\{\beta_{\varphi,\psi}:\psi\in\Zn(\varphi)\big\}
\end{equation}
consists of the unique element $\id_{\UH}$, and the centralizer is abelian.

Some examples of parabolic self-maps (necessarily of positive hyperbolic step) having non-abelian centralizers were constructed in~\cite[Section~8]{CDG-Centralizer}. However, in all these examples, the self-map~$\varphi$ is univalent and moreover, we do not have much  control on the ``size'' of~$\Zn(\varphi)$ and~$\mathcal B(\varphi)$.

So our next task is to construct an example of a non-univalent parabolic self-map with very big set~$\mathcal B(\varphi)$. This set characterizes the size and complexity of the centralizer in the following sense.

Fix any $\beta\in\mathcal B(\varphi)\setminus\{\id_{S_\varphi}\}$ and choose ${\psi\in\Zn(\varphi)}$ such that the simultaneous linearization function $\beta_{\varphi,\psi}$ of~$\psi$ coincides with~$\beta$. Consider the subsemigroup of~$\Zn(\varphi)$ defined by
$$
\Zen_\beta(\varphi):=\Zn(\varphi)\cap\Zn(\psi) .
$$
The structure of~$\Zen_\beta(\varphi)$ is comparable to that of centralizers of parabolic self-maps having zero hyperbolic step. In fact,
\begin{equation}\label{EQ_max-abelian-part}
  \Zen_\beta(\varphi)=\Zn(\varphi\circ\psi).
\end{equation}
Indeed, $\Zen_\beta(\varphi)=\Zn(\varphi)\cap\Zn(\psi)\subset\Zn(\varphi\circ\psi)$. Moreover, since $\beta_{\varphi,\psi}=\beta\neq\id_{S_\varphi}$, the simultaneous linearization coefficient~$c_{\varphi,\psi}$ is not real (see Remark~\ref{RM_real-SLC}\+(a)). Hence, by Theorem~\ref{Thm:coefficient-not-real}, the self-map~${\varphi\circ\psi}$ is of zero hyperbolic step. As a result, $\Zn(\varphi\circ\psi)$ is an abelian semigroup (see \hbox{Section~\ref{S_when-zero}}). Taking into account that ${\{\varphi,\psi\}}\subset\Zn(\varphi\circ\psi)$, this implies the converse inclusion, and hence, equality~\eqref{EQ_max-abelian-part}. Moreover, by Theorem~\ref{Thm:homeomorphism},
$\Zn(\varphi\circ\psi)$ is isomorphic to some topologically closed subsemigroup of $\big[\C,+\big]$ containing~$\Natural_0$.

Note that $\Zen_\beta(\varphi)$ is maximal in the sense that any other abelian subsemigroup of  $\Zn(\varphi)$ containing $\psi$ is a subset of   $\Zen_\beta(\varphi)$. With the help of Theorem~\ref{TH_abelian-part} it is possible to show that to each ${\beta\in\mathcal B(\varphi)\setminus\{\id_{S_\varphi}\}}$ there corresponds exactly one maximal abelian subsemigroup constructed as explained above. Hence, the set of all maximal abelian subsemigroups is indexed, in a sense, by~$\mathcal B(\varphi)$.

Concerning the size of the set~$\mathcal B(\varphi)$, it is worth mentioning that the largest possible set $\mathcal{B}(\varphi)$ is attained for parabolic automorpisms. Indeed, suppose without loss of generality that ${S_\varphi=\UH}$. Following the proof of Proposition~\ref{PR_beta-exists} (see page~\pageref{PG_beta-exists}) and using Propositions~\ref{Prop:com-with-autoprevio} and~\ref{Prop:com-with-auto},  it is not difficult to see that
$$
  \mathcal B(\varphi)\subset\mathcal B(p_\UD),
$$
where $p_\UD$ is the parabolic automorphism of~$\UD$ defined by $p_\UD:={C_\UH^{-1}\circ p_\UH\circ C_\UH}$, where $p_\UH(w):={w+1}$ for all ${w\in\UH}$ and
\begin{equation}\label{EQ_Cayley}
C_\UH\,:\,\UD\to\UH;~\,z\,\mapsto\, i\,\frac{1+z}{1-z}
\end{equation}
is the Cayley map of~$\UD$ onto~$\UH$.\medskip

Now we are going to construct a non-univalent parabolic self-map~$\varphi$ of positive hyperbolic step such that
every $\beta\in\mathcal B(p_\UD)\setminus\{\id_\UH\}$  satisfying a mild additional condition belongs to~$\mathcal B(\varphi)$. So
let ${\beta\in\mathcal B(p_\UD)\setminus\{\id_\UH\}}$. By the very definition of the simultaneous linearization function, we have ${\beta(w+1)}={\beta(w)+1}$ for all ${w\in\UH}$. Taking also into account~\eqref{EQ_beta-normalization} and applying~\cite[Lemma~9.3\+(I)]{CDG-Centralizer} to the function ${\beta-\id_\UH}$, it is not difficult to see that
\begin{equation}\label{EQ_for-beta}
  \beta(w)=w+q(e^{2\pi i w}),\quad w\in\UH,
\end{equation}
for a suitable  non-constant function $q\in\Hol(\UD,\C)$ with $q(0)=0$.

\begin{example}
In order to simplify technical details of our construction, it is convenient to pass, with the help of the Cayley map~\eqref{EQ_Cayley}, from the unit disk~$\UD$ to the upper half-plane~$\UH$. All the relevant definitions such as notions of the canonical holomorphic model, centralizer, simultaneous linearization coefficient and simultaneous linearization function, extend in a natural way to elements of~$\Hol(\UH)$.
For ${z\in\UH}$, we let $$\Phi(z):=z+1-\frac1z.$$ Then $\Phi\in\Hol(\UH)$. Note that $\anglim_{z\to\infty}\Phi(z)/z=1$. Therefore, $\Phi$ is a parabolic self-map of~$\UH$ with the Denjoy\,--\,Wolff point at~$\infty$. Moreover, $\Phi$ is of positive hyperbolic step, which can be seen, e.g., by using \cite[Theorem 3.7]{hStep} with~$\nu:=\delta_0$, the Dirac delta measure concentrated at~$0$. It is not difficult to see\footnote{Indeed, according to a half-plane version of the Julia\,--\,Wolff\,--\,Carath\'eodory Theorem, see e.g. \cite[\S26]{Valiron}, $h_\Phi'$ has a finite angular limit at~$\infty$, which satisfies ${h_\Phi'(\infty):=\anglim_{w\to\infty}h_\Phi'(w)\ge0}$ when ${S_\Phi=\UH}$ and ${h_\Phi'(\infty)\le0}$ when ${S_\Phi=-\UH}$. Rewriting Abel's equation as ${\int_{w}^{\Phi(w)}h_\Phi'(\zeta)\,\mathrm{d}\zeta=1}$ and passing to the limit as~${w\to\infty}$ non-tangentially, we conclude that ${h_\Phi'(\infty)=1}$ and hence, ${S_\Phi=\UH}$.} that the base space~$S_\Phi$ of the canonical model for~$\Phi$ is~$\UH$.

We are not able to precisely calculate~$\Zn(\Phi)$, but the main feature of this example is that the set $\mathcal B(\Phi)$, see~\eqref{EQ_Beta},~--- and hence the centralizer of~$\Phi$~--- are ``almost as big as possible''.

More precisely, fix any non-constant $q\in\Hol(\UD,\C)$ with $q(0)=0$ and suppose that $\ell_0:={\inf_{z\in\UD}\Im q(z)>-\infty}$. Note that our assumptions on~$q$ imply that ${\ell_0<0}$. Further,
let $\beta$ be given by~\eqref{EQ_for-beta}. We are going to show that $\Zn(\Phi)$ contains an element~$\Psi$ such that ${\beta_{\Phi,\Psi}=\beta}$ and ${\Im c_{\Phi,\Psi}>0}$. To this end, denote ${h:=\beta\circ h_\Phi}$.
Our construction of~$\Psi$ is based on the following statement. As before, given ${a\in\Real}$, we will write~$\Pi_a$ for the half-plane ${\{w:\Im w>a\}}$.

\medskip
\noindent{\fontshape{it}\fontseries{bx}\selectfont Claim.} There exists a domain~$U_\beta\subset\UH$ with ${\big(\partial U_\beta\big)\cap\big(\partial\UH\big)}=\{\infty\}$ such that
$h$ is univalent in~$U_\beta$ and
$
   h(U_\beta)=\Pi_\ell
$
for some  ${\ell\ge 0}$.
\medskip

Assuming for a while the claim, fix $c\in\Pi_{\ell-\ell_0}$ and let us check that $\Psi:=\big(h|_{U_\beta}\big)^{-1} \circ (h+c)$ has the properties declared above. Note that ${\beta(\UH)\subset\Pi_{\ell_0}}$. It follows that ${\Psi\in\Hol(\UH)}$. Moreover,
\begin{equation*}
  h\circ\Phi=h+1 \quad\text{and}\quad h\circ\Psi=h+c.
\end{equation*}
In particular, the second equality above implies that $\Psi$ is not elliptic. Since ${\Psi(\UH)\subset U_\beta}$ and since $\infty$ is the only common boundary point of~$U_\beta$ and~$\UH$, this point is the Denjoy\,--\,Wolff point of~$\Psi$.

To prove that $\Psi\in\Zn(\Phi)$, it remains to check that $h$ satisfies condition~{\rm(ii${}'$\hspace{.05em})} in Addendum~\ref{ADD_1}. To this end, fix some ${r>0}$ and consider the hyperbolic discs~$\phdisk^\UH(iy,r)$, ${y>0}$.
Using Theorem~\ref{Thm:V-sets}, it is not difficult to see that $h_\Phi$ is univalent in some domain ${V\subset\UH}$ isogonal at~$\infty$. Note that ${\phdisk^\UH(iy,r)}={y\phdisk^\UH(i,r)}$ for any~${y>0}$ and that the closure of $\phdisk^\UH(i,r)$ is a compact subset of~$\UH$. It follows that $h_\Phi$ is univalent in $\phdisk^\UH(iy,r)$ for all ${y>0}$ large enough. Moreover, note that $\beta$  is univalent in~$\Pi_d$ for some ${d=d(\beta)\ge0}$. This follows at once by Noshiro\,--\,Warschawski Criterion if we observe that
$$
\lim_{\Im w\to+\infty}\beta'(w)-1 ~=~
   \lim_{\Im w\to+\infty}2\pi i e^{2\pi i w}q'(e^{2\pi i w})~=~0.
$$

According to Pommerenke's result (see Remark~\ref{Re:Pommerenke}), for all ${y>0}$ large enough, we have ${h_\Phi\big(\phdisk^\UH(iy,r)\big)}\subset\Pi_d$ and hence, $h$ is univalent in~$\phdisk^\UH(iy,r)$.  Thus, according to Addendum~\ref{ADD_1} to Theorem~\ref{TH_SL-vs-COMM}, $\Phi$ and $\Psi$ commute.

Now it remains to observe that $\beta(w)-w\to q(0)=0$ as ${\Im w\to+\infty}$. Thus, in view of Proposition~\ref{PR_beta-uniq}, we have ${c_{\Phi,\Psi}=c}$ and ${\beta_{\Phi,\Psi}=\beta}$, as desired.
\end{example}

In the proof of the claim we will need the following technical lemma.
\begin{lemma}\label{LM_JordDom-in-the strip}
Let $W\subset\C$ be a domain satisfying the following two conditions:
\begin{romlist}
 \item\label{IT_L-Jordan} $\partial W$ is a Jordan curve on the Riemann sphere~$\ComplexE:=\Complex\cup\{\infty\}$;
 \item\label{IT_L-Re} $\{\Re w:w\in W\}=\Real$.
\end{romlist}
Then $W$ is not contained in any strip of the form ${\{w:a<\Im w<b\}}$.
\end{lemma}
\begin{proof}
In view of condition~\ref{IT_L-Re}, it is not difficult to show that there are two Jordan arcs $C_1$ and~$C_2$ having a common end-point at~$\infty$ and satisfying the following conditions:
$$
   {C_j\setminus\{\infty\}\subset W}~\text{~for~${j=1,2}$},\quad \Re w\to-\infty~\text{~as $\,C_1\ni w\to\infty$,} \quad\text{and~}~\Re w\to+\infty~\text{~as $\,C_2\ni w\to\infty$}.
$$

Let $f$ be a conformal mapping of~$W$ onto~$\UD$. According to Carath\'eodory's Theorem, $f$ extends to a homeomorphism of $\overline W$ onto $\overline\UD$. Therefore, the images of~$C_1$ and~$C_2$ under~$f$ have a common end-point~${f(\infty)\in\UC}$. It follows, see e.g. Theorem~1 in \cite[\S{}II.3]{Goluzin} (alternatively, see \cite[Corollary~2.8]{Pommerenke:BB}), that for any neigbhourhood~$\mathcal O(\infty)$ of~$\infty$ one can find a curve lying in $W\cap\mathcal O(\infty)$ that connects  $C_1$ to~$C_2$. This would not be possible if~$W$~were contained in a strip.
\end{proof}

\begin{proof}[\proofof{the Claim}]
First of all, we collect some elementary properties of the self-map~$\Phi$.\\ Denote
$$
   \Omega:=\{w\in\UH:|w-1|>2\},\qquad \Omega_1:=\{w\in\UH:\Re w<-2\},\qquad \widetilde\UD:=\{\zeta\in\UD:\Im\zeta<0\}.
$$
Then:
\begin{enumerate}[label={\rm (\hspace{.02em}\alph*\hspace{.035em})\,}, ref={\hbox{\rm (\hspace{.02em}\alph*\hspace{.035em})}}, left=.7em]
  \item \label{E_Julia}
   $\Phi(\Pi_a)\subset\Pi_a$ for any $a>0$;
  \medskip

  \item\label{E_uni}
   $\Phi$ and $h_\Phi$ are univalent in~$\Pi_1$;
  \medskip

  \item\label{E_slope}
    $|\Arg \,  \Phi'(z)|<\pi/2\,$ and also $\,| \Arg \, \big(\Phi^{\circ 2}\big)'(z)|<\pi/2\,$ for all ${z\in\Pi_2}$;
  \medskip

  \item\label{E_shift}
     $\Re \Phi(z)>\Re z+1/2$ for all~${z\in\UH}$ with $|z|>2$;
  \medskip

  \item\label{E_inverse}
    $\Phi^{-1}$ admits in ${z\in\Omega}$ the single-valued branch~$\Upsilon$ given by
    $$
      \Upsilon(w) \,:=\, w-1 + \frac2{(w-1)\big(1+Q(\frac{2}{w-1})\big)},
    $$
    where $Q$ is the branch of $\zeta\mapsto~\sqrt{1+\zeta^2}$ mapping $\widetilde\UD$ into $\{z:\Re z>0\}\setminus[1,+\infty)$.
  \medskip

  \item\label{E_invar}
    $\Upsilon(\Omega)\subset\UH$ and moreover, $\Re \Upsilon(w)<\Re w-1/2$ for all~${w\in\Omega_1}$;
  \medskip

  \item\label{E_invar1}
     $\Upsilon(\Omega_1)\subset \Omega_1$;
  \medskip

  \item\label{E_Im}
    $\,\Im \Upsilon(w)\,\ge\,\dfrac{\Im w}{\,1+\frac1{(\Re w)^2}\,}~$  for any $w\in\Omega_1$;
  \medskip

   \item\label{E_ID}
    $\Phi\circ\Upsilon=\id_\Omega$ and ${\Upsilon\circ\Phi|_{\Pi_2}=\id_{\Pi_2}}$.
  \medskip
\end{enumerate}

Since $\Phi$ is parabolic with Denjoy\,--\,Wolff point at~$\infty$, assertion~\ref{E_Julia} is precisely Julia's inequality; see e.g. \cite[Theorem~1.7.8]{BCD-Book}\footnote{Note that $\UH$ in~\cite{BCD-Book} stands for the right half-plane $\{w:\Re w>0\}$ rather than the upper half-plane.}.

The function $\Phi$ is univalent in~$\Pi_1$ by the Noshiro\,--\,Warschawski Criterion, because
$$
  \Re\Phi'(z)=1+\Re(1/z^2)\ge 1-|1/z^2| > 0\quad \text{for all $~z\in\Pi_1$.}
$$
By assertion~\ref{E_Julia} with ${a:=1}$, $\Pi_1$ is $\Phi$-invariant. By \cite[Lemma~3.5.4\+(iv)]{Abate2}, it follows that $h_\Phi$ is also univalent in~$\Pi_1$. This proves assertion~\ref{E_uni}.
Using again~\ref{E_Julia} but with ${a:=2}$, for all ${z\in\Pi_2}$, we have
\begin{eqnarray*}
  \Re\big(\Phi^{\circ 2}\big)'(z)~ &=& ~\Re\Big(1+\frac1{z^2}\Big)\Big(1+\frac1{\Phi(z)^2}\,\Big) ~=~
  \Re\bigg(1+\frac{1}{z^2}+\frac{1}{\Phi(z)^2}+\frac{1}{\big(z\Phi(z)\big)^2}\bigg)\\
  &\ge&~ 1\,-\,\bigg|\frac{1}{z^2}+\frac{1}{\Phi(z)^2}+\frac{1}{\big(z\Phi(z)\big)^2}\bigg| ~>~ \frac{1}{4}.
\end{eqnarray*}
This proves~\ref{E_slope}. The proof of~\ref{E_shift} is also elementary: if $|z|>2$, then
$$\Re\Phi(z) ~=~ \Re z+1-\Re(1/z) ~\ge~ \Re z+1-|1/z|~>~\Re z+1/2.$$

The fact that the expression for~$\Upsilon$ given in~\ref{E_inverse} is the inverse of~$\Phi$, considered as a rational function of degree two, can be checked by direct computation. The map $w\mapsto{\tfrac2{w-1}}$ transforms $\Omega$ onto~$\widetilde\UD$. In particular, choosing the branch of~$Q$ having positive real part in~$\widetilde\UD$, we obtain a single-valued branch~$\Upsilon$ of~$\Phi^{-1}$ in~$\Omega$.

 To prove~\ref{E_invar}, let $D:={\big(\C\setminus\Real\big)\cup(0,1)}$ and consider the map $V:D\to\UH$ defined by $V(z):={\sqrt{(z+1)/(z-1)}}$.
 It is again just a simple computation to check that $\Upsilon(w)^2={V\big(Q(\tfrac{2}{w-1})\big)^2}$ for all ${w\in\Omega}$. One of the branches of $\sqrt{(z+1)/(z-1)}$ in~$D$ takes values in~$\UH$, while the other one in~$-\UH$. Therefore, noticing that ${\Upsilon(1+iv)\in\UH}$ for any ${v>2}$ and recalling that $Q(\widetilde\UD)\subset D$, we may conclude that $\Upsilon(w)={V\big(Q(\tfrac{2}{w-1})\big)}$ for all ${w\in\Omega}$.
 It follows that ${\Upsilon(\Omega)\subset\UH}$.

 In order to complete the proof of~\ref{E_invar}, fix an arbitrary ${w\in\Omega_1}$ and write ${z:=\Upsilon(w)}$, $\zeta:={2/(w-1)}$. Then ${|\zeta|<1}$ and ${\Re Q(\zeta)>0}$. As a consequence,
$$
\big|z|~\ge~\big|w-1\big|\,-\,\left|\tfrac{\zeta}{1+Q(\zeta)}\right| ~>~ 2.
$$
Therefore, thanks to~\ref{E_shift}, ${\Re\Upsilon(w)=\Re z<\Re \Phi(z)-1/2=\Re w-1/2}$.

Assertion~\ref{E_invar1} follows at once from~\ref{E_invar}.
To prove~\ref{E_Im}, note that in the above notation, as we have just shown, ${\Re z <\Re w<0}$. Hence, we have
$$
   \Im w ~=~\bigg(1+\frac{1}{|z|^2}\bigg)\,\Im z
      ~<~ \bigg(1+\frac{1}{|\Re z|^2}\bigg)\,\Im z
      ~<~ \bigg(1+\frac{1}{|\Re w|^2}\bigg)\,\Im z,
$$
and the desired inequality follows immediately.

The first equality in~\ref{E_ID} holds simply because $\Upsilon$ is a branch of~$\Phi$ in~$\Omega$. To prove the second equality, first recall that by~\ref{E_Julia}, ${\Phi(\Pi_2)\subset\Pi_2\subset\Omega}$. Hence, the composition ${\Upsilon\circ\Phi|_{\Pi_2}}$ is well-defined holomorphic function in~$\Pi_2$. Next we check by a direct computation that ${\Upsilon\big(\Phi(iy)\big)=iy}$ for any~${y>2}$. By the Identity Principle it then follows that ${\Upsilon\big(\Phi(z)\big)=z}$ for any~${z\in\Pi_2}$, as desired.
\medskip

The rest of the proof is divided into several steps.

\Step1 We start by constructing a closed curve~$\Gamma$ passing through~$\infty$ and otherwise contained in~$\Pi_2$ and such that
\begin{equation}\label{EQ_h-of-gamma}
   h_\Phi\big(\Gamma\setminus\{\infty\}\big)~=~\bigcup_{k\in\mathbb Z}h_\Phi(\Gamma_0)+k
\end{equation}
for some Jordan arc~$\Gamma_0\subset\Pi_2$.

   $$ A:=\Big(1+\frac{1}{3^2}\Big)\prod_{n=0}^{+\infty}\bigg(1+\frac{1}{\big(3+\tfrac n2\big)^2}\bigg) \leqno{\text{Let}}$$ and define a two-sided sequence ${(z_n)_{n\in\mathbb Z}\subset\UH}$ by
   $$
     z_0:=-3+iA\,\in\,\Omega_1,\quad z_{n}:=\Phi^{\circ n}(z_0) ~\,~\text{~and~}~\,~
                                    z_{-n}:=\Upsilon^{\circ n}(z_0)\quad\text{for all~$~{n\in\Natural}$.}
   $$
   Note that $z_{-n}$ is well-defined thanks to assertion~\ref{E_invar1}. Moreover, applying~\ref{E_invar} repeatedly, we obtain $${|\Re z_{-n}|\ge3+\tfrac{n}{2}}\quad\text{for all $\,{n\in\Natural_0}$.}$$ Combining this with \ref{E_Im} and~\ref{E_Julia}, it is not difficult to see that ${(z_n)\subset\Pi_2}$.

   Now we join $z_n$'s by a continuous curve. Let $\Gamma_0:=[z_{-1},z_0]$, i.e. the straight line segment joining~$ z_0$ and~$z_{-1}$. Using~\ref{E_Im}, for any $w\in\Gamma_0$, we obtain
   $$
    \Im w ~\ge~ \min\big\{\Im z_0, \, \Im z_{-1}\big\} ~\ge~ \frac{\Im z_0}{\,1+\frac1{(\Re z_0)^2}\,} ~=~ \prod_{n=0}^{+\infty}\bigg(1+\frac{1}{\big(3+\tfrac n2\big)^2}\bigg).
   $$
   Similarly, according to~\ref{E_shift} and~\ref{E_invar}, we have
   $$
    {\Re\Phi^{\circ n}(w) > \Re w+\tfrac n2 \ge \Re z_{-1}+\tfrac n2} \quad\text{and}\quad
    {\Re \Upsilon^{\circ n}(w)<\Re w-\tfrac n2\le -3-\tfrac n2}
   $$
   for all ${n\in\Natural}$ and all~${w\in\Gamma_0}$.

   As a consequence, the union
   $$
     \Gamma:=\Big(\,\bigcup_{k\in\mathbb Z}\Gamma_k\,\Big)\cup\{\infty\},\qquad
     \Gamma_{-n}:=\Upsilon^{\circ n}(\Gamma_0),~~\Gamma_n:=\Phi^{\circ n}(\Gamma_0),~~n\in\Natural,
   $$
   is a closed curve passing through~$\infty$ and contained in~$\Pi_2\cup\{\infty\}$.

   To prove~\eqref{EQ_h-of-gamma}, recall that $h_\Phi$ satisfies Abel's equation ${h_\Phi\circ\Phi}={h_\Phi+1}$. Hence ${\bigcup_{k=0}^{+\infty}h_{\Phi}(\Gamma_k)}={\bigcup_{k=0}^{+\infty}\big(h_{\Phi}(\Gamma_0)+k\big)}$.
   Moreover, in view of~\ref{E_ID}, we have
   $$
     h_\Phi\circ\Upsilon+1=(h_\Phi+1)\circ\Upsilon=h_\Phi\circ\Phi\circ\Upsilon=h_\Phi\circ\id_\Omega=h_\Phi|_\Omega.
   $$
   It follows that ${\bigcup_{k=1}^{+\infty}h_{\Phi}(\Gamma_{-k})}={\bigcup_{k=1}^{+\infty}\big(h_{\Phi}(\Gamma_0)-k\big)}$.
   This implies~\eqref{EQ_h-of-gamma}.

\Step2 We are going to show that~$\Gamma$ is a Jordan curve on the Riemann sphere~$\ComplexE$, and that~$\Pi_2$ contains a domain~$U$ satisfying ${\partial U=\Gamma}$.\medskip

  Recall that $\Phi$ is injective in~$\Pi_1$ and that ${\Gamma\setminus\{\infty\}\subset\Pi_2\subset\Pi_1}$. Hence, $\Gamma_k$ is Jordan arc for each~${k\in\Natural}$. Note also that $\Upsilon$ is univalent in $\Omega_1$, because it is a branch of the inverse of a single-valued function. Taking into account that ${\Gamma_k\in\Omega_1}$ for all negative integers~$k$. As a result, ${\Gamma_k}$~is a Jordan arc for each~${k\in\mathbb Z}$.

  Moreover, by construction, for each ${k\in\mathbb Z}$ the arcs $\Gamma_k$ and $\Gamma_{k+1}$ have a common end-point at~$z_k$. To show that $\Gamma$ is a Jordan arc, we have to verify that there can be no other intersections of $\Gamma_k$ and~$\Gamma_m$ for ${k\neq m}$. Thanks to assertion~\ref{E_ID}, it is sufficient to consider the case ${m=0}$ and ${k\in\Natural}$. For $k=1$ and $k=2$, the desired conclusion follows from~\ref{E_slope}:
  $$
    \frac{\di}{\di t} \Re\frac{\Phi^{\circ k}\big(tz_0+(1-t)z_{-1}\big)}{z_0-z_{-1}}=\Re \big(\Phi^{\circ k}\big)'\big(tz_0+(1-t)z_{-1}\big)>0,\quad t\in[0,1],\quad k=1,2.
  $$
  So suppose now that $k>2$. Then $\Gamma_k\cap\Gamma_0=\emptyset$ because $\Re z\le \Re z_0=-3$ for any~${z\in\Gamma_0}$, while for all~${w\in\Gamma_k}$ we have $w=\Phi^{\circ k}(\zeta)$ for some ${\zeta\in\Gamma_0}$ and hence,
  $$
     \Re w > \Re \zeta+\tfrac{k}2 \ge \Re z_{-1} +\tfrac32 = \Re \Upsilon(z_0)+\tfrac32 > -3,
  $$
  where the last inequality holds because
\begin{equation*}
  \begin{split}
  \Re\Upsilon(z_0)&=\Re z_0-1+\Re \frac2{(z_0-1)\big(1+Q(\frac{2}{z_0-1})\big)}\\
  & >
         -4-\frac{\left|\frac{2}{z_0-1}\right|}{\left|1+Q(\frac{2}{z_0-1})\right|} >
         -4-\frac{1/2}{1+\Re Q(\frac{2}{z_0-1})} > -4-\tfrac12.
  \end{split}
  \end{equation*}

  In view of the Jordan Curve Theorem, it remains to see that one of the two connected components of ${\C\setminus\Gamma}$ is contained in~$\Pi_2$. Denote $L:=\{w:\Im w=2\}=\big(\partial\Pi_2\big)\setminus\{\infty\}$. Since $\Gamma\cap L=\emptyset$, the line~$L$ lies in one of the connected components of~${\C\setminus\Gamma}$. Denote the other connected component by~$U$.
  Note that ${U\cap L=\emptyset}$, which means that either ${U\subset\Pi_2}$, or ${U\cap\Pi_2=\emptyset}$. Since $\partial U=\Gamma\subset\Pi_2$, the latter alternative is not possible; so we have ${U\subset\Pi_2}$ as desired.

\Step3 Note that in view of~\ref{E_uni}, $h_\Phi$ is injective on~$\overline U\setminus\{\infty\}$. We are going to show that ${h_\Phi(U)\supset\Pi_b}$ for some~$b>0$.\medskip

   The key point on this step is to see that
   \begin{equation}\label{EQ_claim}
     {\partial h_\Phi(U)}=\widetilde\Gamma:={h_\Phi\big(\Gamma\setminus\{\infty\}\big)\cup\{\infty\}}.
   \end{equation}
   Once this fact is established, from~\eqref{EQ_h-of-gamma} we would immediately deduce that $$b:=\sup\big\{\Im w:w\in\partial h_\Phi(U)\big\}=\max\big\{\Im h_\Phi(z):z\in\Gamma_0\big\}<+\infty.$$
   This would imply that either ${\Pi_b\subset h_\Phi(U)}$, or ${\Pi_b\cap h_\Phi(U)}=\emptyset$. Note that in fact, the latter alternative cannot occur, because otherwise $h_\Phi(U)$ would be contained in the strip ${\{w:0<\Im w<b\}}$,  which is impossible by Lemma~\ref{LM_JordDom-in-the strip}.

   To prove~\eqref{EQ_claim}, notice that $\widetilde\Gamma$ is a Jordan curve on~$\ComplexE$. This follows from the injectivity of~$h_\Phi$ on~${\Gamma\setminus\{\infty\}}$, Step 2,  and from equality~\eqref{EQ_h-of-gamma}. Clearly, $h_\Phi(U)$ is entirely contained in one of the two connected components of~$\C\setminus\widetilde\Gamma$. Denote that component by~$\widetilde U$. Consider some conformal mappings $f_1$ and~$f_2$ of~$\UH$ onto $U$ and~$\widetilde U$, respectively. In view of Carath\'eodory's Theorem, we can normalize these mappings in such a way that ${f_1(\infty)=f_2(\infty)=\infty}$.  The univalent self-map of~$\UH$ defined by $\chi:={f_2^{-1}\circ h_\Phi\circ f_1}$ extends continuously to~$\Real$ and this extension maps~$\Real$ into itself. Then by the Schwarz Reflection Principle, $\chi$ extends to a conformal self-map of~$\Complex$. Thus, $\chi$ is a linear map of the form $\chi(\zeta)=a\zeta+b$ for all~${\zeta\in\C}$, where ${a>0}$ and ${b\in\Real}$. The required conclusion~\eqref{EQ_claim} follows now easily.

\Step4 Now we can finally construct the domain $U_\beta$ with the required properties.\medskip

  Recall that  $\beta$ is univalent in~$\Pi_d$ for some~$d\ge0$. Clearly, we may suppose that~${d>b}$. Then  by the previous step, ${\Pi_d\subset h_\Phi(U)}$.

  The key point here is to show that $\beta(\Pi_d)\supset\Pi_\ell$ for some ${\ell>0}$. Once this is proved, recalling that $h_\Phi$ is univalent in~$U$, we can define
  $$
     U_\beta=\big(h_\Phi|_U\big)^{-1}\big(\big(\beta|_{\Pi_d}\big)^{-1}(\Pi_\ell)\big).
  $$
  Then clearly, $h(U_\beta)=\Pi_\ell$ and $h$ is univalent in~$U_\beta$. Moreover, ${U_\beta\subset U}$. Since ${\overline{U}\setminus\{\infty\}\subset\UH}$, it follows that
 ${\big(\partial U_\beta\big)\cap\big(\partial\UH\big)}\subset\{\infty\}$. It remains to notice that in fact, ${\infty\in\partial U_\beta}$ because $h$ is not bounded in~$U_\beta$.

 To show that ${\beta(\Pi_d)\supset\Pi_\ell}$ for some ${\ell>0}$, notice that thanks to the representation~\eqref{EQ_for-beta}, we have
 $$M:={\sup_{w\in\Pi_d}\big|\beta(w)-w\big|}<+\infty.$$
 Set ${\ell:=d+2 M}$ and take now any ${w_0\in\Pi_\ell}$. Then the closure of ${E:=\{w:|w-w_0|<2 M\}}$ is contained in~$\Pi_d$. Applying Rouch\'e's Theorem to the functions $f(w):={w-w_0}$ and $g(w):={\beta(w)-w}$, we see that ${\beta(w)-w_0}={f(w)+g(w)}$ has the same number of zeroes in~$E$ as the function $f(w)$, i.e. exactly one zero.
 Since $w_0$ in this argument is arbitrary, we are done.
\end{proof}

\end{document}